\documentclass[reqno,12p]{amsart}
\usepackage[T1]{fontenc}
\usepackage{geometry}
\usepackage{mathtools,amssymb,amsthm,mathrsfs,color,lineno,paralist,graphicx,float}
\usepackage[colorlinks,
linkcolor=red,
anchorcolor=green,
citecolor=blue,
]{hyperref}
\usepackage{amsmath}
\usepackage{pdfpages}
\usepackage{amstext}
\usepackage{stmaryrd}
\usepackage{bbm}
\usepackage{comment}
\usepackage{faktor,url}
\usepackage{xfrac}
\usepackage{mathabx}
\usepackage{enumitem}
\usepackage{refcount}
\usepackage{amssymb}
\graphicspath{{Dibujos/}}

\numberwithin{equation}{section}
\numberwithin{figure}{section}
\theoremstyle{plain}
\newtheorem{thm}
{\protect\theoremname}[section]
  \theoremstyle{definition}
  \newtheorem{defn}[thm]{\protect\definitionname}
  \theoremstyle{remark}
  \newtheorem{rem}{\protect\remarkname}
  \theoremstyle{plain}
  \newtheorem{lem}[thm]{\protect\lemmaname}
  \theoremstyle{plain}
  \newtheorem{prop}[thm]{\protect\propositionname}
  \theoremstyle{plain}
  \newtheorem{cor}[thm]{\protect\corollaryname}

\newtheorem{thmx}{Theorem}

\newtheorem{thmxprime}{Theorem}

\newcommand{\bs}{\boldsymbol}

  \providecommand{\definitionname}{Definition}
  \providecommand{\lemmaname}{Lemma}
  \providecommand{\propositionname}{Proposition}
  \providecommand{\remarkname}{Remark}
\providecommand{\theoremname}{Theorem}
\providecommand{\corollaryname}{Corollary}

\title{Robustly transitive behavior in symplectic dynamics}

\author[J. Paradela]{J. Paradela}
\address[JP]{ Department of Mathematics, University of Maryland, 4176 Campus Drive, 20782, MD, US}
\email{paradela@umd.edu}

\begin{document}

\begin{abstract}
We consider the direct product of two symplectomorphisms, one of which exhibits a basic set and the other one a non-degenerate elliptic equilibrium. Under local a domination condition we show that a typical \textit{real-analytic}  deformation of this system  displays a large normally hyperbolic lamination on which the induced dynamics is \textit{robustly transitive}. As a corollary of our construction we also obtain new open sets of real-analytic symplectomorphisms  which are \textit{topologically mixing} but not uniformly hyperbolic.

To establish these results we develop perturbation techniques to create \textit{blender} horseshoes in the real-analytic setting and import ideas from control theory which show that, typically, these objects have a large domain of influence. 
\end{abstract}

\maketitle

\tableofcontents

\section{Introduction}

  A longstanding problem in classical mechanics is to understand the topological and measure theoretic properties of typical symplectic dynamics. On the topological counterpart,  genericity of transitivity was established in the $C^1$ symplectic setting by Arnaud, Bonatti and Crovisier \cite{MR2173426} (following earlier work of  Bonatti and Crovisier \cite{MR2090361} in the conservative scenario). Despite these spectacular early advances, understanding the mechanisms that give rise to transitive behavior in symplectic dynamics remains today a central challenge for more \textit{regular topologies}. For instance, in the smooth setting, not only the tools used in \cite{MR2090361,MR2173426} (based on a $C^1$ version of the closing lemma) are not available,  but, we know after the work of Herman, there exist open sets of symplectomorphisms which are not transitive  (see the survey article by Yoccoz  \cite{MR1206072}).

\medskip

The phenomenon which better highlights the subtleties hindering transitivity in the symplectic smooth scenario is the \textit{robust} persistence of \textit{elliptic behavior}. On one hand, under mild non-degeneracy conditions and suitable regularity assumptions, it is well known that elliptic equilibria are robustly accumulated by invariant tori whose union forms a set of positive measure (we refer the interested reader to \cite{MR3146593}). We call this union the KAM set, as the existence of these tori is deduced by means of the machinery developed by Kolmogorov, Arnold and Moser (see \cite{MR209436,MR2269239} and the references therein). Note that, in view of this phenomenon, a positive measure set of orbits (which indeed becomes of relative measure one as we approach the equilibrium), are confined forever within a neighbourhood of the equilibrium. On the other hand, even if some orbits may dodge the KAM set surrounding them, elliptic equilibria are typically very \textit{sticky} \cite{MR4072793}, and orbits which approach them must spend very large times confined therein before being able to exit (if they do so).

\medskip

Despite these obstacles, a mechanism to construct orbits  leaking between the KAM set and connecting far apart regions of the phase space was already proposed by Arnold in \cite{MR163026}. Roughly speaking, the global instability in Arnold's mechanism relies on the accumulation of local instability created by a chain of partially hyperbolic objects (i.e. their tangent space splits into expanding, contracting and center subspaces, the latter one being ``dominated'' by the other two)  located in the complement of the KAM set (see Figure \ref{fig:fig1}).   Although the mechanism was already implemented by Arnold in a simple but representative toy model \cite{MR163026}, realizing this construction in practice is, nevertheless,  rather challenging. We refer the interested reader to \cite{MR2981810,MR3646879,GelfreichTuraev,MR4298716,MR4033892,MR4509324,MR4495839,MR4729212,MR4913967} and the references therein for recent advances on this topic. Some of these results include  real-analytic systems and  are well suited  to be applied to concrete systems (just to cite a few recent applications see \cite{DiffusionEllipticKaloshin,MR4544807,MR4624370,Clarke25,MR5013757}).

\begin{figure}
    \centering
    \includegraphics[scale=0.55]{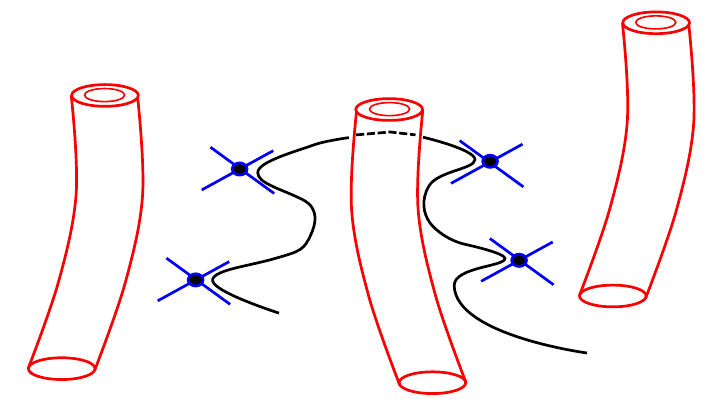}
    \caption{Sketch of an orbit gaining momentum from a chain of partially hyperbolic objects (in blue) and  leaking between the KAM tori (in red).}
    \label{fig:fig1}
\end{figure}
\medskip

The main feature shared by the different variations of Arnold's instability mechanism  is the extensive use of partially hyperbolic objects. It is then natural to ask if some of the techniques developed within the general theory of partially hyperbolic dynamics could be adapted to  produce robustly transitive behavior in the symplectic framework. One of the first contributions to this program can be found in \cite{NassiriPujalsTransitivity}, where Nassiri and Pujals   developed a symplectic version of the so-called \textit{blender horseshoes} (introduced by   Bonatti and Díaz in \cite{BonattiDiazblenders}) to obtain $C^\infty$ arcs of symplectomorphisms exhibiting large robustly transitive sets in the complement of the KAM set (see also the related work \cite{MR3917797}).

Loosely speaking, a blender is a thick part of a basic set such that its stable (unstable) manifold meets every unstable (stable) manifold that comes near it and, moreover, it does so in a persistent fashion. In this way, these objects (as well as their symplectic counterpart in \cite{NassiriPujalsTransitivity})  fix the lack of transversality between the strong stable and unstable directions due to the presence of a center subspace. By means of this phenomenon, blenders act as a robust semi-local source of transitivity in the sense that any two objects connected to the blender  can be joined by a path alternating segments of stable and unstable manifolds.
\medskip

Blenders have found numerous applications within smooth dynamics. Recently, a number of works have proposed new mechanisms for their construction in symplectic settings  where the center subspace is two-dimensional \cite{guardia2025partiallyhyperbolicdynamics3body,LiTuraevPreprint}. Importantly, these mechanisms do not avoid regions with elliptic behavior but exploit instead  the presence of quasi-elliptic points\footnote{One indeed expects that  quasi-elliptic points are abundant within the set of non-Anosov symplectomorphisms. This fact was indeed established in the $C^1$ topology \cite{MR455049} but, to the best of the author's knowledge, remains a conjecture in more regular topologies.}   (also called saddle-center). On one hand, the rather elementary construction in \cite{guardia2025partiallyhyperbolicdynamics3body} is tailored for the application to given one-parametric families and was used to construct symplectic blenders in the 3-body problem of celestial mechanics. On the other hand, the mechanism in \cite{LiTuraevPreprint} is much more general and allowed Li and Turaev to show that symplectic blenders typically emerge around symplectic maps displaying a one-dimensional KAM curve with a transverse homoclinic orbit to it.

The purpose of this work is to develop a higher dimensional version of the elementary mechanism in \cite{guardia2025partiallyhyperbolicdynamics3body} and construct symplectic blenders in scenarios with arbitrarily large center dimension. By doing so we  extend the main result  in \cite{NassiriPujalsTransitivity} to the real-analytic category and even obtain a more general statement. In order to state our main results in a precise way and give some more intuition on our construction, we first need to introduce some definitions.

\subsection{Hamiltonian diffeomorphisms}
 Given a real-analytic   symplectic manifold $(X,\varpi)$  we we denote by 
\begin{equation}\label{eq:spaceofHams}
\mathcal H^\omega(X)=\{H:X\times [0,1]\to \mathbb R\colon H\text{ is real-analytic}\}
\end{equation}
the space of real-analytic  time-dependent Hamiltonians equipped with the real-analytic topology (see Appendix \ref{sec:realanalytictop} for the precise definition of this topology).
We write $H_t(\cdot)=H(\cdot,t)$ and recall that to any $H\in \mathcal H^\omega(X)$ we can associate a time-dependent Hamiltonian vector field $X_{H_t}: X\to TX$ defined (implicitly) by
\begin{equation}\label{eq:correspondenceformsvectors}
\varpi(\cdot,X_{H_t})=\mathrm dH_t(\cdot).
\end{equation}
We denote by 
\begin{align*}
\Phi_H:X&\to X
\end{align*}
the time-one map associated to the flow of this vector field and recall that it is a symplectic transformation, i.e. $(\Phi_H)^*\ \varpi=\varpi$. Finally, we introduce the group of real-analytic Hamiltonian diffeomorphisms
\[
\mathrm{Ham}^\omega(X)=\{\Phi_H:X\to X\colon H\in \mathcal H^\omega(X)\}.
\]
The interested reader is referred to  \cite{MR3674984} for an introduction to Hamiltonian and symplectic dynamics.\medskip

\subsection{Direct products of Hamiltonian diffeomorphisms}

Consider two real-analytic symplectic manifolds $(N,\varpi_N)$ and $(M,\varpi_M)$ together with Hamiltonian diffeomorphisms $F_N\in\mathrm{Ham}^\omega(N,\varpi_N)$ and $F_M\in \mathrm{Ham}^\omega(M,\varpi_M)$. Then, we define their \textit{direct product} 
    \begin{equation}\label{eq:directproductmap}
\begin{aligned}
  F_0:M\times N&\to M\times N\\
  (x_M,x_N)&\mapsto (F_M(x_M),F_N(x_N)).
\end{aligned}
\end{equation}
In Theorem \ref{thm:mainham} below we will show that that large robustly transtive sets appear for typical real-analytic unfoldings
of direct products satisfying certain assumptions. Note that the direct product \eqref{eq:directproductmap} is again a real-analytic Hamiltonian diffeomorphism (i.e. $F_0\in \mathrm{Ham}^\omega(M\times N)$). Indeed,  there exist $H_N\in \mathcal H^\omega(N)$ and $H_M\in \mathcal H^\omega(M)$ such that $F_N=\Phi_{H_N}$ and $F_M=\Phi_{H_M}$. Then, an equivalent definition of the direct product \eqref{eq:directproductmap} is $F_0=\Phi_{H_0}$ where $H_0=H_M+H_N\in \mathcal H^\omega(M\times N)$.
\medskip

\subsubsection*{Non-degenerate elliptic equilibria}

Our first assumption (see \ref{it:Assumption1} below) requires a mild amount of elliptic behavior for the dynamics on $N$. Recall that a  fixed point $q\in N$ of a  symplectic map $F:N\to N$ is \textit{elliptic} if the spectrum of the linearization of $F$ at $q$ is simple and it is contained  in the unit circle in $\mathbb C$. In addition, we say that an elliptic fixed point $q$ is  \textit{non-degenerate} if there exists a neighbourhood  $U\subset N$ of $q$ and a local coordinate system $\psi:z\in \mathbb B^d\subset\mathbb C^{d}\to U\subset N$ such that $\psi^*\varpi_N=i \mathrm d\bar z\wedge \mathrm dz$ and conjugates $F$ to 
    \[
    \psi^{-1}\circ F\circ\psi:z\mapsto  z(\exp (i \Omega(|z|^2))+O(|z|^3))
    \]
    where $|z|^2$ is shorthand notation for $(|z_1|^2,\dots,|z_d|^2)\in \mathbb R^d$ and  $\Omega(|z|^2)=\mathrm{diag}(\Omega_1(|z|^2),\dots,\Omega_d(|z|^2))$ is a diagonal matrix satisfying
    \[
    \mathrm{det} D\Omega(0)\neq 0.
    \]
Namely, up to higher order terms, the map $F$ is conjugated to a product of rotations with frequencies $\{\Omega_i(|z|^2)\}_i$ and the frequency map $|z|^2\mapsto \Omega(|z|^2)$ is a local diffeomorphism. 

\medskip

Given a direct product $F_0=F_M\times F_N$ as in \eqref{eq:directproductmap} we say that they satisfy the property:
\begin{enumerate}[label={[H\arabic*]}]
    \item \label{it:Assumption1} if the map $F_N:N\to N$ exhibits a non-degenerate elliptic equilibrium.
    \end{enumerate}

\subsubsection*{Normally hyperbolic submanifolds of direct products}

Our second assumption (\ref{it:Assumption2} below) asks for a mild amount of hyperbolicity for the dynamics on $M$. To state it precisely we need the following definition.

\begin{defn}\label{defn:bunchingpinching}
   Let $F_0=F_M\times F_N$ be the direct product of two diffeomorphisms $F_M:M\to M$ and $F_N:N\to N$. Let $a\in M$ be a hyperbolic fixed point for $F_M$ and let $T_a M=E^s_a\oplus E^u_a$ be the corresponding splitting of the tangent space of $M$ at $a$. For $\kappa>0$ we say that the invariant submanifold $\Lambda=\{a\}\times N\subset M\times N$ is $\kappa$-normally hyperbolic  if there exist metrics $\lVert\cdot\rVert_M$ on $M$ and $\lVert\cdot\rVert_N$ on  $N$ such that
       \begin{equation}\label{eq:bunchingdefn}
         \lVert  DF_N^{-1}\rVert_N^{\kappa+1} \lVert DF_M|_{E^s_a}\rVert_M < 1\qquad\qquad\text{and}\qquad\qquad \lVert DF_N\rVert^{\kappa+1}_N \lVert DF_M^{-1}|_{E^u_a}\rVert_M< 1.
       \end{equation}
\end{defn}

In words, the stable  forward contraction (unstable backwards contraction) dominates strongly the center backwards (forward) behavior. It is well known that  the exponent $\kappa$ in this property can be related to the regularity of certain invariant foliations associated to (the continuation) of the submanifold $\Lambda$ and which are important to  understand the homoclinic picture associated to  $\Lambda$ (see  \cite{MR501173} or Section \ref{sec:outline} of this manuscript). 

\begin{defn}
     In the setting of Definition \ref{defn:bunchingpinching} we say that $\Lambda$ is \textit{$\kappa$-pinched} if \begin{equation}\label{eq:pinchingdefn}
        \bar \alpha^{\kappa+1} \lambda<1
       \end{equation}
       where 
       \[
       \bar\alpha=\frac{\max\{
\lVert DF_M|_{E_a^s}\rVert_M,\,
\lVert DF_M^{-1}|_{E_a^u}\rVert _M
\}}{\min\{
\lVert DF_M^{-1}|_{E_a^s}\rVert^{-1}_M,\,
\lVert DF_M|_{E_a^u}\rVert^{-1} _M
\}}\qquad\text{and}\qquad \lambda=\max\{
\lVert DF_M|_{E_a^s}\rVert_M,\,
\lVert DF_M^{-1}|_{E_a^u}\rVert _M
\}.
       \]
\end{defn}

Pinching compares the maximal rate of normal expansion/contraction with the minimal rate of normal expansion/contraction. Observe that in the case $\mathrm{dim}(M)=2$ the assumption is trivial.
\medskip

Given a direct product $F_0=F_M\times F_N$ as in \eqref{eq:directproductmap} we say that it satisfies property:
 \begin{enumerate}[label={[H\arabic*]},start=2]
    \item \label{it:Assumption2} 
    if the map $F_M:M\to M$ displays a hyperbolic fixed point $a\in M$ whose stable and unstable manifolds intersect transversally and, moreover, $\Lambda=\{a\}\times N$ is $\kappa$-normally hyperbolic and  $\kappa$-pinched  for  $F_0$ with some $\kappa>0$. 
    \end{enumerate}

   \subsubsection*{Emergence of robustly transitive laminations} 
   Note that under assumption \ref{it:Assumption2}, the classical Smale horseshoe Theorem (see \cite{MR1326374}) shows that for any alphabet $A\subset \mathbb N$, the map $F_M:M\to M$ displays an open set $Q\subset M$ and a subset   $\Sigma\subset Q$ homeomorphic to $A^\mathbb Z$ such that $\Sigma$ is a basic set for the return map $R$ induced on $Q$. Hence, if we let $\mathcal Q_0=Q\times N$, the return map 
   \begin{equation}\label{eq:returnmap unperturbed heuristic}
   \mathcal R_0:\breve{\mathcal Q}_0\subset\mathcal Q_0\to \mathcal Q_0
   \end{equation}
   induced by $F_0$ on $\mathcal Q_0$,  exhibits an invariant set 
   \begin{equation}\label{eq:lamin unperturbed heuristic}
   \mathcal X_0=\Sigma\times N\subset M\times N.
   \end{equation}
Moreover, thanks to the domination assumption in \ref{it:Assumption2}, is robust under small perturbations (see \cite{MR501173} or Section \ref{sec:outline} of this manuscript). Although for $F_0|_{\mathcal X}$ the fiber dynamics (along $N$) is decoupled from the chaotic base dynamics (along $\Sigma$) driven by the shift, one expects that after an arbitrarily small perturbation, the chaoticity of the base ``contaminates'' the fiber dynamics, generating rich dynamics in the continuation of $\mathcal X$.

This intuition has been confirmed in a number of works (see for instance \cite{Moeckeldrift,MR2104598,NassiriPujalsTransitivity,GelfreichTuraev,MR3917797} and the references therein). In particular, the aforementioned $C^\infty$-arcs of symplectomorphisms constructed by Nassiri and Pujals in \cite{NassiriPujalsTransitivity} correspond to $C^\infty$-unfoldings of direct products \eqref{eq:directproductmap} in which the second factor, i.e. the map $F_N$, is integrable (see \cite{MR1345386} for a definition of the concept of integrability  in Hamiltonian dynamics) and exhibits a weakly hyperbolic fixed point. The corresponding large robustly transitive sets are given by the continuation of $\mathcal X$ along these arcs.  
\medskip

In our main Theorem \ref{thm:mainham} below, under assumptions \ref{it:Assumption1}-\ref{it:Assumption2} we upgrade their construction to the real-analytic setting and remove the integrability assumption for the map $F_N$. To do so, we introduce a rather elementary geometric mechanism for the creation of symplectic blenders, which only requires assumption \ref{it:Assumption1} and a local version of \ref{it:Assumption2}. As we already discussed in the introduction, some of our tools can be seen as a higher dimensional generalization of our previous work with Guàrdia \cite{guardia2025partiallyhyperbolicdynamics3body}.

\subsection{Real-analytic Hamiltonian  unfoldings of direct products}
Given $r>1$ we define  the open subset 
\begin{equation}\label{eq:Ballperturbations}
\mathcal B^\omega_r(X)=\{H\in \mathcal H^\omega(X)\colon |H|_{C^r(X)}< 1\}.
\end{equation}
In Theorem \ref{thm:mainham} below, given a direct product of Hamiltonian diffeomorphisms $F_0=F_M\times F_N$ as in \eqref{eq:directproductmap} associated to a Hamiltonian $H_0=H_M+H_N$ we consider one-parametric unfoldings of the form
\begin{equation}\label{eq:deformationmap}
F_{\varepsilon,f}=\Phi_{H_{\varepsilon, f}} \qquad \text{where}\qquad H_{\varepsilon, f}= H_0+\varepsilon f\qquad\text{and}\qquad f\in \mathcal B^\omega(M\times N).
\end{equation}
    Observe that for any $f\in \mathcal B^\omega(X)$ and any $0\leq |\varepsilon|\leq 1$, the  time-one map
    \begin{equation}\label{eq:unfolding}\Phi_{\varepsilon f}:M\times N\to M\times N
    \end{equation}
      is uniformly $C^{r-1}$-close to the identity. In particular, $F_{\varepsilon,f}$ is $C^{r-1}$-close to $F_0=\Phi_{H_0}$.
\medskip

\subsection{Main results}
 We recall that a map $F:X\to X$ of a topological space $X$ is said to be \textit{topologically mixing} if for any two open sets $U,V\subset X$ there exists $N\in\mathbb N$ such that $F^n(U)\cap V\neq \emptyset$ for all $n\geq N$. Clearly, topological mixing implies transitivity.

\subsection*{Topologically mixing laminations for Hamiltonian diffeomorphisms}
The following is our first main result.

\begin{thmx}\label{thm:mainham}
Let $(N,\varpi_N)$ and $(M,\varpi_M)$ be  real-analytic symplectic manifolds and let $N$ be compact. There exists $
  r_0(\mathrm{dim}N),\kappa_0(\mathrm{dim} N)>0$  such that the following holds. Let $F_0\in \mathrm{Ham}^\omega(M\times N)$ be a direct product of Hamiltonian diffeomorphisms  which satisfies assumptions \ref{it:Assumption1}-\ref{it:Assumption2} with $\kappa\geq \kappa_0$.
 
Then, for any $f$ belonging to an open and dense subset of $\mathcal B^\omega_{r_0}(M\times N)$  there exists $\varepsilon_0(F_0,f)>0$  such that for  any $0<\varepsilon\leq \varepsilon_0$, the map $
   F_{\varepsilon,f}\in \mathrm{Ham}^\omega(M\times N)$ in \eqref{eq:deformationmap} displays an open set $\mathcal Q_\varepsilon\subset M\times N$ and a subset $\mathcal X_\varepsilon\subset \mathcal Q_\varepsilon$  homeomorphic to $\mathbb N^\mathbb Z\times N$ such that: 
    \begin{enumerate}
        \item $\mathcal X_\varepsilon$ is locally maximal and normally hyperbolic for the return map $\mathcal R_\varepsilon$ induced on $\mathcal Q_\varepsilon$.  Moreover, its projection onto $N$ is the whole manifold,
        \item The restriction  $\mathcal R_\varepsilon|_{\mathcal X_\varepsilon}:\mathcal X_\varepsilon\to \mathcal X_\varepsilon$ is  topologically mixing.
    \end{enumerate}
\end{thmx}

We present the proof of Theorem \ref{thm:mainham} in Section \ref{sec:outline} by reducing it to a number of more technical statements which will be proved in the remaining sections of the paper.  Before continuing, let us dicuss the assumptions and the main novelties of our result in comparison with those in the previous literature.

\subsubsection*{Domination   assumption:} In our current proof of Theorem \ref{thm:mainham}, the exponent $\kappa_0$ associated  to the domination assumption in \ref{it:Assumption2} depends on the dimension on $N$.  We have not tried to optimize the lower value of $\kappa_0$ for which Theorem \ref{thm:mainham} holds.

On the other hand, it is worth remarking that in some cases of interest, the domination assumption in \ref{it:Assumption2} is automatically satisfied. This is the case if, for instance, the unperturbed dynamics on $N$ is given by a completely integrable map all whose singularities are elliptic  and $\mathrm{dim}(M)=2$. A natural example is to consider $N=\mathbb C\mathbb P^d$  (for any $d\in \mathbb N$) equipped with the Fubini-Study form $\varpi_{\mathrm FS}$ and let $F_N$ be the time-one map associated to a Hamiltonian given by a non-linear function of the moment map coordinates (see \cite{MR1853077}).
\medskip

\subsubsection*{Strong form of topological instability}
In the case where the unperturbed dynamics on $N$ is integrable, classical results show that, after arbitrarily small perturbations, there exist orbits which leak through the remnants of the invariant toric foliation (of the integrable system) and connect regions separated by a distance independent of the size of the perturbation. Some  recent contributions are \cite{MR3646879,GelfreichTuraev,MR4298716,MR4033892,MR4509324,MR4729212,MR4913967}. 

Under stronger assumptions (direct product structure of the unperturbed map), our result implies that a generic arbitrarily small real-analytic perturbation, gives rise to the existence of robustly transitive sets. Moreover, these sets are \textit{large} in the sense that  their projection to $N$ does not depend on the size of the perturbation (indeed it coincides with   $N$). In particular, we deduce the existence of orbits whose projection to $N$ is dense. It seems possible that, if in assumption \ref{it:Assumption2} we ask in addition for the existence of a homoclinic tangency, robustly transitive sets can be obtained with Hausdorff dimension arbitrarily close to maximal, i.e. $\mathrm{dim}M+\mathrm{dim}N$ (see also Section 1.5 in \cite{NassiriPujalsTransitivity}).

\medskip

\subsubsection*{Conceptual novelties}
We now explain briefly which are the main tools that we introduce and highlight the main differences between our result and that in \cite{NassiriPujalsTransitivity}.

First, the perturbation techniques to create blenders  in \cite{NassiriPujalsTransitivity} are inherently $C^\infty$, while our construction uses only real-analytic perturbations. Even if from their construction one can obtain real-analytic symplectic maps with robustly transitive sets (by approximation), a priori one cannot guarantee that these are close to the given  direct product in the real-analytic topology. 

Second, the arcs of symplectormorphisms with robustly transitive sets constructed in \cite{NassiriPujalsTransitivity} correspond to perturbations in somehow particular directions, while in Theorem \ref{thm:mainham} we show that robustly transitive sets are found for perturbations along an open and dense set of directions. The key idea leading to  this extension is that our construction of blenders is based on a rather elementary mechanism, the so-called transversality torsion mechanism, which appears typically in unfoldings of a direct product as that in Theorem \ref{thm:mainham}. Moreover, our techniques are well suited for applications to concrete parametric families. Indeed, in our previous work with  Guàrdia we have used a lower dimensional  version of our construction in this paper to build symplectic blenders in the 3-body problem of Celestial Mechanics \cite{guardia2025partiallyhyperbolicdynamics3body}.

Third, the authors in \cite{NassiriPujalsTransitivity}  assume that the unperturbed dynamics on $N$ is given by an integrable map\footnote{That is the case in Theorem B of \cite{NassiriPujalsTransitivity}. In  Theorem A in \cite{NassiriPujalsTransitivity} the authors only assume that the unperturbed dynamics on $N$ displays a weakly hyperbolic periodic orbit and deduce the existence of robustly transitive sets which, a priori, might be smaller than the lamination $\mathcal X_\varepsilon$.} while we only need assumption \ref{it:Assumption1}. To gain this generality  we will incorporate two new tools into their construction. One of them is the extensive use of the so-called \textit{scattering map} \cite{DelaLLavegaps} to describe the dynamics along homoclinic excursions, together with the perturbation theory for its asymptotic computation developed in \cite{DelaLLaveScattmap}. The second one is a classical idea in sub-Riemannian geometry (coming from the proof of the Ball-Box Theorem \cite{MR1867362})  which will be used to derive a property which we call $\varepsilon$-Reachability for the system of scattering maps. As we will see below this set of tools can be combined with the  blender dynamics - which one should think of as a semi-local source of robust transitivity - to generate robustly transitive dynamics in the whole lamination.
\medskip

\subsection*{Topologically mixing laminations for symplectomorphisms}
   Theorem \ref{thm:mainham} has an analogue for symplectomorphisms which we now present. To do so we will use the  well known fact  that any  neighbourhood of the identity in  $\mathrm{Symp}^\omega(X,\varpi)$ (the group of real-analytic symplectomorphisms, see Appendix \ref{sec:realanalytictop}) can be identified with a neighbourhood of zero in the vector space of real-analytic closed one-forms (see \cite{MR3674984}). More precisely, we denote by $\mathfrak H^\omega$ the space of real-analytic closed one-forms $\theta:X\to T^*X$ equipped with the real-analytic topology  and let
\begin{equation}\label{eq:...}
\mathfrak B^\omega_r(X)=\{\theta\in \mathfrak h^\omega(X)\colon |\theta|_{C^{r}}\leq 1\}.
\end{equation}
To  any $\theta\in \mathfrak B_r^\omega(X)$ we can associate a symplectic vector field $X_\theta: X\to TX$ defined (implicitly) as
\begin{equation}\label{eq:correspondenceformsvectors}
\varpi(\cdot,X_\theta)=\theta(\cdot).
\end{equation}
In this way, we obtain a one-parametric family of real-analytic symplectomorphisms $\{\Phi^\varepsilon_\theta\}_\varepsilon:X\to X$ defined implicitly by the initial value problem
\begin{equation}\label{eq:deformationmap one form}
\begin{cases}
   \frac{\mathrm{d}\Phi^\varepsilon_\theta}{\mathrm d\varepsilon}=X_\theta\circ\Phi^\varepsilon_\theta\\
  \Phi^0_\theta=\mathrm{id}.
\end{cases}
\end{equation}

\begin{rem}
    Observe that for any $\theta\in \mathfrak B^\omega_r(X)$ and any $0\leq |\varepsilon|\leq 1$, the real-analytic symplectomorphism $\Phi^\varepsilon_\theta:X\to X$ defined by \eqref{eq:correspondenceformsvectors}-\eqref{eq:deformationmap} is uniformly $C^{r}$-close to identity.
\end{rem}

\begin{thmxprime}\label{thm:mainsymp}
Let $(N,\varpi_N)$ and $(M,\varpi_M)$ be  real-analytic symplectic manifolds and let $N$ be compact. There exists $
 r_0(\mathrm{dim}N),\kappa_0(\mathrm{dim}N)>0$  such that the following holds. Let $F_0\in \mathrm{Symp}^\omega(M\times N)$ be a direct product of symplectomorphisms  which satisfies assumptions \ref{it:Assumption1}-\ref{it:Assumption2} with $\kappa\geq \kappa_0$.
 
Then, for any $\theta$ belonging to an open and dense subset of $\mathfrak B^\omega_{r_0}(M\times N)$  there exists $\varepsilon_0(F_0,\theta)>0$ such that for  any $0<\varepsilon\leq \varepsilon_0$, the map
\begin{equation}\label{eq:unfoldingforsymplectic}
   F_{\varepsilon,\theta}=F_0\circ\Phi_{\theta}^\varepsilon\in \mathrm{Symp}^\omega(M\times N)
   \end{equation}
   with $\Phi_{\theta}^\varepsilon$ as in \eqref{eq:deformationmap one form} displays an open set $\mathcal Q_\varepsilon\subset M\times N$ and a subset $\mathcal X_\varepsilon\subset \mathcal Q_\varepsilon$  homeomorphic to $\mathbb N^\mathbb Z\times N$ such that: 
    \begin{enumerate}
        \item $\mathcal X_\varepsilon$ is locally maximal and normally hyperbolic for the return map $\mathcal R_\varepsilon$ induced on $\mathcal Q_\varepsilon$.  Moreover, its projection onto $N$ is the whole manifold,
        \item The restriction  $\mathcal R_\varepsilon|_{\mathcal X_\varepsilon}:\mathcal X_\varepsilon\to \mathcal X_\varepsilon$ is  topologically mixing.
    \end{enumerate}
\end{thmxprime}

\subsection*{Topologically mixing symplectomorphisms}
We finally present the following corollary of Theorem \ref{thm:mainham} in which we construct  new open sets of  real-analytic topologically mixing symplectomorphisms which are not uniformly hyperbolic. The proof of Theorem \ref{thm:corollaryMain} is presented in Section \ref{sec:proof top mix symplectomorphisms}.

\begin{thmx}\label{thm:corollaryMain}
    Let $(M,\varpi_N)$ and $(N,\varpi_M)$ be  real-analytic symplectic compact  manifolds. There exists $
r_0(\mathrm{dim}N),\kappa_0(\mathrm{dim}N)>0$  such that the following holds.  Let $F_0\in \mathrm{Symp}^\omega(M\times N)$ be a direct product of real-analytic symplectomorphisms as in \eqref{eq:directproductmap} which satisfies \ref{it:Assumption1}-\ref{it:Assumption2} with $\kappa\geq \kappa_0$, is partially hyperbolic with center bundle tangent to $N$ and   such that $F_M:M\to M$ is a transitive Anosov diffeomorphism. 
 
Then, for any $\theta$ belonging to an open and dense subset of $\mathfrak B^\omega_{r_0}(M\times N)$ there exists $\varepsilon_0(F_0,\theta)>0$ such that for  any $0<\varepsilon\leq \varepsilon_0$, the  symplectomorphism $
   F_{\varepsilon,\theta}\in \mathrm{Symp}^\omega(M\times N)$ in \eqref{eq:unfoldingforsymplectic} is topologically mixing and not uniformly hyperbolic.
\end{thmx}

\begin{rem}
    In the setting of Theorem \ref{thm:corollaryMain} we say that $F_0=F_M\times F_N$ is  \textit{partially hyperbolic} with center bundle tangent to $N$ if there exists a continuous invariant  splitting of  $TM=E^s\oplus E^u$ and adapted metrics $\lVert\cdot\rVert_M$ on $M$ and $\lVert\cdot\rVert_N$ on  $N$ such that
       \[
         \lVert  DF_N^{-1}\rVert_N \lVert DF_M|_{E^s}\rVert_M < 1\qquad\qquad\text{and}\qquad\qquad \lVert DF_N\rVert\lVert DF_M^{-1}|_{E^u}\rVert_M< 1.
       \]
\end{rem}
\medskip

Under the additional assumption that the unperturbed  dynamics on $N$ is integrable, a version of Theorem \ref{thm:corollaryMain} in which the direct product $F_0$ is $C^\infty$ accumulated by  symplectomorphisms which are robustly transitive can be easily deduced from Theorem B in \cite{NassiriPujalsTransitivity}. In the  case where $\mathrm{dim} N=2$, a much stronger result was obtained by Horita and Sambarino \cite{MR3682778}. In the particular setting of Theorem \ref{thm:corollaryMain} (their result applies to a larger class of interesting scenarios), they show that if $\mathrm{dim}N=2$, the map $F_0$ can be $C^r$ approximated (for any $r\in\{2,\dots,\infty\}$) by a stably ergodic symplectomorphism. To the best of our knowledge, it is still not known whether a higher dimensional version of their result is true (although that it is expected to be the case) and whether their construction can be upgraded to obtain density in the real-analytic topology. A family of real-analytic nonuniformly hyperbolic symplectomorphisms which are stably ergodic are the so-called Berger-Carrasco maps  \cite{MR4069249}. In these examples, the center bundle is also two dimensional.
\medskip

\subsection{Some open problems}

We present here a small list of open problems which are connected in some way to our main result. A more extensive list can be found in \cite{NassiriPujalsTransitivity}.

\subsubsection*{Stable ergodicity around direct products} In the setting of Theorem \ref{thm:mainham} it is natural to ask if the dynamics in the lamination $\mathcal X_\varepsilon$ is ergodic with respect to the natural product measure. This question already appears in \cite{NassiriPujalsTransitivity} but, to the best of our knowledge, has not yet been answered even in less regular topologies. 

A possibly simpler question is to understand if the sets of symplectomorphisms in Corollary \ref{thm:corollaryMain} are  stably ergodic (see the discussion following Corollary \ref{thm:corollaryMain}).

\subsubsection*{Symplectic blenders beyond the direct product setting}

 It seems possible that  for some steps in the proof of Theorem \ref{thm:mainham}, in particular the creation of symplectic blenders,  the direct product assumption can be dropped (by exploiting that indeed, close to certain partially  hyperbolic objects, the dynamics can always locally decoupled by a smooth coordinate transformation).
 
 For instance, Li and Turaev show in \cite{LiTuraevPreprint} that given a exact symplectomoprhism with a partially hyperbolic KAM curve admitting a transverse homoclinic orbit, symplectic blenders emerge under an arbitrarily small real-analytic perturbation.  It would be interesting to generalize their result to the case of partially hyperbolic KAM tori of arbitrary dimension.

\subsubsection*{Symplectic blenders for a priori stable Hamiltonians}
A related problem which, we believe, could be addressed with the techniques developed in this paper, is to show that, for an open and dense set of directions, a real-analytic one-parametric unfolding of a  completely integrable Hamiltonian dynamical system exhibits symplectic blenders of any possible center dimension.

\subsubsection*{The influence of blenders in celestial mechanics}
In our previous work with Guàrdia \cite{guardia2025partiallyhyperbolicdynamics3body} we established the existence of symplectic blenders in the 3-body problem of celestial mechanics and used them to construct \textit{oscillatory orbits} with large closure. Along these orbits, although their distance is unbounded in time, the three bodies come together infinitely often. Moreover, the orbit of two of the bodies remains bounded, and visits a locally dense set of the space of oriented ellipses.

Another interesting scenario where symplectic blenders may emerge corresponds to the homoclinic picture around the Lagrange periodic orbit $L_1$ of the restricted 3-body problem (see \cite{LiTuraevPreprint}). Finally, it would be of great interest to understand  if symplectic blenders can be  effective tools to tackle Alekseev's conjecture on the existence of a locally dense set of initial conditions leading to collision \cite{MR629685,MR3951693}.

\subsection*{Acknowledgements}
The author would like to warmly thank  B. Fayad and M. Guàrdia for their interest in this project and valuable insights and comments on a first draft of this manuscript.  We are also indebted to A. Clarke, who pointed to us that Theorem \ref{thm:mainham} could be extended to symplectic diffeomorphisms (not necessarily Hamiltonian), yielding Theorem \ref{thm:mainsymp}.

Finally, we want to thank M. Gidea, R. de la Llave and T.M. Seara for informing us about their ongoing work  \cite{HormanderGdlLS}, in which they also make use of ideas from control theory  to prove a result in the same spirit of Proposition \ref{prop:maintransportIFS} of the present manuscript, and establish the existence of a strong form of Arnold diffusion in nearly-integrable Hamiltonians. We have also been informed that they had announced such a result a number of years ago in \cite{MR4160091}.

\section{Definitions, strategy and  proof of Theorem \ref{thm:mainham}}\label{sec:outline}

In this section we describe the strategy of the proof of Theorems \ref{thm:mainham}-\ref{thm:corollaryMain} and introduce the more relevant definitions.
\medskip
 
 First, in Section \ref{sec:nihimintro} we recall the definition of \textit{normally hyperbolic manifold} and prepare ourselves to study \textit{homoclinic phenomena} associated to normally hyperbolic manifolds such as the existence of \textit{normally hyperbolic laminations}.  To that end we review the construction of the so-called \textit{scattering map} introduced by Delshams, de la Llave and Seara \cite{DelaLLavegaps,DelaLLaveScattmap}. As we will see below, the scattering map, defined as a composition of stable and unstable holonomies, encodes the dynamics along the set of homoclinic orbits to a normally hyperbolic manifold. Having in mind the setting in Theorem \ref{thm:mainham}, in Proposition \ref{prop:mainnhim} we give a number of  sufficient conditions  for the existence of normally hyperbolic lamination and relate the dynamics on the lamination to the scattering map dynamics. The results in this section rely on ``soft techniques'' and require less stringent assumptions on the domination exponent $\kappa_0$ and the regularity exponent $r_0$ than the ones needed for Theorem \ref{thm:mainham}.

 \medskip

Section \ref{sec:IFSintro} is devoted to the study of certain \textit{iterated function systems} (IFSs) on compact manifolds. The reason for this (apparent) detour is that the dynamics in the lamination in Proposition \ref{prop:mainnhim} is given by a \textit{skew-product} over the shift. Hence one might expect to be able to translate some dynamical constructions from the IFS setting to the skew-product setting. 

In particular, having in mind the structure of the skew-product dynamics in Proposition \ref{prop:mainnhim}, our goal in Section \ref{sec:IFSintro} is to present sufficient conditions for a suitable class of IFS's to be \textit{topologically mixing}. This is done in two steps. The first step focuses on the creation of  a \textit{semi-local source} of robust transitivity: the so-called symbolic blenders (see \cite{NassiriPujalsTransitivity} and Section \ref{sec:IFSintro}). Sufficient conditions for their construction are given in  Theorem \ref{thm:mainIFS}. The second step involves \textit{global transport} in the IFS, namely, wether the IFS has the property of connecting any two points up to an small (but finite) error. Exploiting the fact  that the IFS's of interest for us include a couple of close to identity maps,  a simple argument borrowed from the \textit{control theory} literature shows that this is the case. A precise statement is given in Proposition \ref{prop:maintransportIFS}.

\medskip

In Section \ref{sec:skewprodintro},  building on the results   obtained in Section \ref{sec:IFSintro}, we give sufficient conditions which guarantee that the skew-product in  Proposition \ref{prop:mainnhim} is topologically mixing. This is the content of Proposition \ref{prop:skewprodmain}.
\medskip

Finally, in Section \ref{sec:genericityintro} we verify that for an open and dense set of perturbations in $\mathcal B^\omega_{r_0}(M\times N)$ (recall its definition in \eqref{eq:Ballperturbations}), the corresponding one-parametric family of Hamiltonian maps 
\[
F_{\varepsilon,f}:M\times N\to M\times N
\]
as in \eqref{eq:deformationmap}: 
\begin{enumerate}
    \item satisfies the asumptions in Proposition \ref{prop:mainnhim} and,
    \item  the corresponding skew-product dynamics  can be understood in terms of an IFS which verifies the assumptions introduced in Section \ref{sec:IFSintro} (namely the ones needed to apply Theorem \ref{thm:mainIFS} and Proposition \ref{prop:maintransportIFS}). In particular, Proposition \ref{prop:skewprodmain} applies.
\end{enumerate}

\begin{rem}\label{rem:remarksofttechniques}
    The presentation in  Sections \ref{sec:nihimintro}-\ref{sec:skewprodintro} does not rely on the particular scenario of Theorem \ref{thm:mainham} and is presented in a more abstract setting. In particular, throghout these sections we will consider $C^r$  diffeomorphisms  $F:X\to X$ of a compact Riemannian manifold with $r\geq 2$ and do not assume any product structure for $X$ nor that $F$ is a symplectic map.  Only in Section \ref{sec:genericityintro} we will come back to the setting of Theorem \ref{thm:mainham}.
\end{rem}

\subsection{Normally hyperbolic cylinders and homoclinic channels}\label{sec:nihimintro}
Consider a $C^r$ diffeomorphism $F:X\to X$ of a  Riemannain manifold $X$ with $r\geq2$.
\begin{defn}\label{defn:nhim}
    Let $\Lambda\subset X$ be an invariant submanifold for $F$, i.e. $F(\Lambda)=\Lambda$. We say that $\Lambda$ is \textit{normally hyperbolic} if there exists $C>0$ and constants
    \[
    0<\lambda<\alpha^{-1}<1
    \]
    such that for any $x\in \Lambda$ there exists a splitting $T_x X=E^s_x\oplus T_x\Lambda \oplus E^u_x$ verifying
    \begin{align*}
        v\in E^s_x &\Longleftrightarrow \lVert DF^n v\rVert\leq C\lambda^n \lVert v\rVert,\qquad n\geq 0\\
         v\in T_x\Lambda &\Longleftrightarrow \lVert DF^n v\rVert\leq C\alpha^{|n|} \lVert v\rVert,\qquad n\in\mathbb Z\\
          v\in E^u_x &\Longleftrightarrow \lVert DF^n v\rVert\leq C\lambda^{ -n} \lVert v\rVert,\qquad n\leq 0.
    \end{align*}
\end{defn}

\begin{rem}
    It is not difficult to check that the particular definition we gave for the direct product case (see Definition \ref{defn:bunchingpinching}) satisfies the requirements in Definition \ref{defn:nhim} (see for instance \cite{MR501173}).
\end{rem}

Given a normally hyperbolic manifold we define its \textit{stable and unstable sets}
\begin{equation}\label{eq:stablemanifoldsdefn}
\begin{aligned}
W^s(\Lambda)=&\{x_*\in X\colon \exists x\in \Lambda \text{ such that }\mathrm{dist}(F^n(x),F^n(x_*))\to 0\text{ as } n\to +\infty\}\\
W^u(\Lambda)=&\{x_*\in X\colon \exists x\in \Lambda \text{ such that }\mathrm{dist}(F^n(x),F^n(x_*))\to 0\text{ as } n\to -\infty\}
\end{aligned}
\end{equation}
and, given $x\in \Lambda$ introduce
\begin{equation}\label{eq:stablemanifoldsfibersdefn}
\begin{aligned}
W^s_x=&\{x_*\in X\colon \mathrm{dist}(F^n(x),F^n(x_*))\to 0\text{ as } n\to +\infty\}\\
W^{u}_x=&\{x_*\in X\colon \mathrm{dist}(F^n(x),F^n(x_*))\to 0\text{ as } n\to -\infty\}.
\end{aligned}
\end{equation}
The following facts are standard (see \cite{MR501173,DelaLLaveScattmap})
\begin{enumerate}
   
 \item The manifold $\Lambda$ is $C^l$ with 
    \begin{equation}\label{eq:spectralgap}
        1\leq l<\mathrm{min}\left(r,\frac{|\log\lambda|}{\log \alpha}\right),
    \end{equation}

    \item The maps  $x\mapsto W^{u,s}_x$ are $C^{l-1-j}$ when $W^{u,s}_x$ are given the $C^j$ topology and at any $x\in \Lambda$ we have $T_x W^{u,s}_x=E_x^{u,s}$,

    \item $W^{u,s}_x$ are $C^{r}$  manifolds for any $x\in \Lambda$.    
\end{enumerate}
In particular, it follows that $\{W^{u,s}_x\}_{x\in\Lambda}$ give a $C^{l-1}$ foliation of $W^{u,s}(\Lambda)$.
\medskip

\subsubsection*{Homoclinic channels and the scattering map}

We now introduce some tools to study homoclinic phenomena associated to normally hyperbolic manifolds. Our presentation follows closely that in \cite{DelaLLaveScattmap} and we refer the interested reader there for a much more complete exposition. 

We start by introducing the \textit{wave maps}
\begin{align*}
\Omega^{u,s}:W^{u,s}(\Lambda)&\mapsto \Lambda\\
x&\mapsto x^{u,s}
\end{align*}
where $x^{u,s}\in \Lambda$ is the unique point such that $x\in W^{u,s}_{x^{u,s}}$. Notice that these maps are $C^{l-1}$.
\begin{defn}[Homoclinic channel] Let $\Lambda$ be a normally hyperbolic manifold of a $C^r$ diffeomorphism $F:X\to X$ and let  $\Omega^{u,s}:x\in W^{u,s}(\Lambda)\to x^{u,s}\in\Lambda$ be the corresponding wave maps. We say that a submanifold $\Gamma\subset W^s(\Lambda)\cap W^u(\Lambda)$ is a \textit{homoclinic channel} to $\Lambda$ if at any $x\in \Gamma$
\[
T_x\Gamma=T_x W^s(\Lambda)\cap T_x W^u(\Lambda)\qquad\qquad T_x \Gamma \oplus T_x W^{u,s}_{x^{u,s}}=T_x W^{u,s}(\Lambda)
\]
\end{defn}

By analogy with their lower dimensional counterpart, i.e. hyperbolic periodic orbits displaying transverse homoclinic orbits to them, one usually expects that the existence of homoclinic channels to a normally hyperbolic submanifold leads to some sort of ``chaotic behavior''. 

One of the more relevant tools for the study of the dynamics around homoclinic channels is the so-called scattering map, introduced by Delshams, de la Llave and Seara in \cite{DelaLLavegaps,DelaLLaveScattmap}.

\begin{defn}[Scattering map]
    Let $\Lambda$ be a normally hyperbolic manifold of a $C^r$ diffeomorphism  $F:X\to X$, let $\Omega^{u,s}:x\in W^{u,s}(\Lambda)\to x^{u,s}\in \Lambda$ be the corresponding wave maps and let $\Gamma\in W^s(\Lambda)\cap W^u(\Lambda)$ be a homoclinic channel. Then, we define the \textit{scattering map}
    \begin{align*}
    S_\Gamma: D_\Gamma\subset \Lambda\to \Lambda\qquad\qquad \text{ where }\qquad D_\Gamma=\Omega^u|_\Gamma(\Gamma)
    \end{align*}
    by 
    \[
S_\Gamma=\Omega^s|_\Gamma\circ (\Omega^u|_\Gamma)^{-1}: x^u\mapsto x^s
    \]
\end{defn}
Observe that, by construction, the scattering map is a $C^{l-1}$ diffeomorphism onto their image.

\medskip

\subsubsection*{Existence of normally hyperbolic laminations}
 Throughout this section we will consider a one-parametric family of $C^r$ $(r\geq  2$) diffeomorphisms  $\{F_\varepsilon\}_{\varepsilon}$ of a Riemaniann manifold $X$ (see Remark \ref{rem:remarksofttechniques}). We shall suppose that there exists $\varepsilon_0(F_0)>0$ such that for any $0\leq \varepsilon\leq\varepsilon_0(F_0)$:
\begin{enumerate}[label={[A\arabic*]}]
    \item \label{it:H1} \textit{$\kappa$-normal hyperbolicity:} $F_\varepsilon:X\to X$ has a compact normally hyperbolic invariant manifold $\Lambda_\varepsilon$ which associated rates $0<\lambda_\varepsilon<\alpha_\varepsilon^{-1}<1$ satisfying
    \begin{equation}\label{eq:spectralgap}
        \alpha_\varepsilon^{\kappa+1}\lambda_\varepsilon<1.
    \end{equation}
    for some $\kappa\geq 2$. Moreover, we assume that  there exists a reference manifold $N$ and a parametrization $\psi_\varepsilon:N\to \Lambda_\varepsilon$. We denote by 
    \begin{equation}\label{eq:innerdynamicsdefn}
        T_\varepsilon=\psi_\varepsilon^{-1}\circ F_\varepsilon|_{\Lambda_\varepsilon}\circ\psi_\varepsilon:N\to N
    \end{equation}
the map describing the restriction of the dynamics to $\Lambda_\varepsilon$.
    \item\label{it:H1+1/2} \textit{$\kappa$-pinching:}  Let 
\[
\underline\lambda_\varepsilon=\min\{\lVert D F_\varepsilon^{-1}|_{E^s}\rVert^{-1},\lVert D F_\varepsilon|_{E^u}\rVert^{-1}\}\qquad\text{and}\qquad \bar\alpha_\varepsilon=\frac{\lambda_\varepsilon}{\underline\lambda_\varepsilon}.
\]
    Then, 
    \begin{equation}\label{eq:pinching}
    \bar\alpha_\varepsilon^{\kappa+1}\lambda_\varepsilon <1
    \end{equation}
    for some $\kappa\geq 1$.
    \item \label{it:H2} \textit{Full homoclinic intersections:} There exist $m$ ($m\geq 1$)  homoclinic manifolds 
    \[\{\Gamma^i_\varepsilon\}_{i=1}^m\subset W^s(\Lambda_\varepsilon)\pitchfork W^u(\Lambda_\varepsilon)
    \]
    such that for every $i=1,\dots,m$ the wave maps    $\Omega^{u,s}|_{\Gamma^i_\varepsilon}:\Gamma^i_\varepsilon\to \Lambda_\varepsilon$  are global diffeomorphisms  and 
    \[
    \bigg( \bigcup_{n\in\mathbb Z} F_\varepsilon^n(\Gamma^i_\varepsilon)\bigg)\cap \bigg( \bigcup_{n\in\mathbb Z} F_\varepsilon^n(\Gamma^j_\varepsilon)\bigg)=\emptyset.
    \]
    for any pair $i\neq j$. 
    We denote by $S_{\Gamma^i_\varepsilon}:\Lambda\to \Lambda$ the corresponding scattering maps and let 
    \begin{equation}\label{eq:scattmapdynamicsdefn}S_{i,\varepsilon}=\psi_\varepsilon^{-1}\circ S_{\Gamma^i_{\varepsilon}}\circ\psi_\varepsilon:N\to N
    \end{equation}
\end{enumerate}
\medskip

Before continuing, let us make some comments on the assumptions. First of all, one should notice that since both conditions are open, it follows from regular dependence on parameters (see \cite{MR501173,DelaLLaveScattmap}), that it is only necessary to verify Assumptions \ref{it:H1}-\ref{it:H2} for $F_0$ and then, they automatically hold for all $F_\varepsilon$ with $\varepsilon>0$ sufficiently small (see Section \ref{sec:genericity}). We prefer to state the assumptions already for the family to avoid innecessary repetitions. 

Concerning the nature of the assumptions,   \ref{it:H1} guarantees that the stable and unstable foliations associated to $\Lambda_0$ are at least $C^{\kappa}$ if $r\geq\kappa$ while \ref{it:H2} will be  of use to show that the homoclinic picture associated to $\Lambda_\varepsilon$ is sufficiently rich.    The assumption \ref{it:H1+1/2} is of more technical nature: it measures the ratio between maximal and minimal normal contraction or expansion. 
\medskip

The following result describes the set in which, under suitable assumptions to be made precise later, the map $F_\varepsilon$ induces a  topologically mixing dynamical system.

\begin{prop}\label{prop:mainnhim}
    Let $F_\varepsilon:X\to X$ be a parametric family of $C^r$ maps ($r\geq 2$)  which satisfy assumptions \ref{it:H1}-\ref{it:H2}. Let $A=\{1,\dots,m\}$ and let $\Sigma=(\mathbb N\times A)^\mathbb Z$. Then, there exists 
    \[
    C(F_0)>0, \qquad\bar\lambda(F_0)<1\quad\text{and}\quad \varepsilon_0(F_0)>0
    \]
    such that for any $0\leq \varepsilon\leq \varepsilon_0$ the map $F_\varepsilon$ displays a set $\mathcal X_\varepsilon\subset X$, an open neighbourhood $\mathcal Q_\varepsilon\supset \mathcal X_\varepsilon$   and a homeomorphism $\Psi:\Sigma\times N\to \mathcal X_\varepsilon$ such that the following holds:
    \begin{enumerate}
        \item \textit{(Normal hyperbolicity):} The set $\mathcal X_\varepsilon$ is normally hyperbolic and locally maximal for the return map $\mathcal R_\varepsilon$ induced on $\mathcal Q_\varepsilon$,
        \item \textit{(Skew-product dynamics):} The homeomorphism $\Psi$ conjugates the dynamics of $\mathcal R_\varepsilon$ to the skew-product
        \begin{equation}\label{eq:skewprodnhim}
        \begin{aligned}
        \mathcal F_\varepsilon=\Psi^{-1}\circ\mathcal R_\varepsilon\circ\Psi:\Sigma\times N&\to \Sigma\times N\\
        (\bs\omega,z)&\mapsto (\sigma(\bs\omega), F_{\bs\omega}(z))
        \end{aligned}
        \end{equation}
        where $\sigma$ is the full right shift and, for every  $\bs\omega=(\omega,\iota)\in\Sigma$ and $z\in N$
        \begin{equation}\label{eq:skewprodformula}
        F_{\bs\omega}=T^{\omega_0}_\varepsilon\circ (S_{\iota_0,\varepsilon}+R_{\bs\omega})
        \end{equation}
    with $T_\varepsilon$ as in \eqref{eq:innerdynamicsdefn} being of class $ C^{\min\{r,\kappa+1\}}$, $\{S_{i,\varepsilon}\}_i$ as in \eqref{eq:scattmapdynamicsdefn} being of class $ C^{\min\{r,\kappa\}}$ and $R_{\bs\omega}: N\to N$ satisfying 
    \[
    |R_{\bs \omega}|_{C^1}\leq  C \bar\lambda ^{\min\{ \omega_{-1},\omega_{0}\}}.
    \]
    \end{enumerate}
\end{prop}

The proof of Proposition \ref{prop:mainnhim} is presented in Section \ref{sec:nhimproof}. Indeed, in that section we prove a slightly expanded version of this result, Proposition \ref{prop:nhimexpanded}, where we include some additional information on the dependence of $R_{\bs\omega}$ on $\bs\omega$ (which will prove useful for the proof of Theorem \ref{thm:mainham}).

Before proceeding, a couple of remarks are in order. First, to obtain a lamination in which the base is homeomorphic to $\Sigma$, we study the return map to an open set which contains the homoclinic channels $\{\Gamma_i\}$ in its closure. The point of doing that instead of considering a lamination close but uniformly away from the channels (and hence with base defined over a sequence space on a finite alphabet), is the following. By allowing the symbols of $\omega$ to be arbitrarily large, the corresponding orbit comes arbitrarily close to the normally hyperbolic manifold $\Lambda_\varepsilon$. As we will see below, this is crucial for shadowing certain pseudo-orbits of the iterated function system generated by the map $T_\varepsilon$ and the maps $\{S_{i,\varepsilon}\}_{i=1}^m$.
\medskip

\subsection{Iterated function systems}\label{sec:IFSintro}
In order to prove that the skew-product dynamics \eqref{eq:skewprodnhim} is robustly topologically mixing, the first natural step is to study the IFS generated by the symplectic diffeomorphisms $T_\varepsilon$ in \eqref{eq:innerdynamicsdefn} and $\{S_{i,\varepsilon}\}_{i=1}^m$ in \eqref{eq:scattmapdynamicsdefn} of the compact manifold $N$.

\subsubsection*{Symbolic blenders for IFS}
 In this section we show that for $m\geq 2$,  under suitable assumptions which we will later show are satisfied for a generic family $F_\varepsilon$ as in the setting of Theorem \ref{thm:mainham}, the IFS generated by $\{T_\varepsilon,\{S_{i,\varepsilon}\}_{i=1}^m\}$  displays symbolic blenders, a semi-local source of robust transitivity.  Since this result might be of independent interest, instead of focusing on the maps $\{T_\varepsilon,\{S_{i,\varepsilon}\}_{i=1}^m\}$ derived from the study of the diffeomorphism $F_\varepsilon$ above, we state the main result of this section in a more general setting.
 \medskip
 
 We start by properly defining the concept of symbolic blender.
 Consider an IFS $\{T_i\}_{i=1,\dots,m}$ of smooth diffeomorphisms of a smooth manifold $N$. 
Let  $Q\subset N$ be an open ball and denote by $\{\mathtt T_i\}_{i=1,\dots,m}$ the corresponding induced return maps on $Q$. We now suppose that:
\begin{enumerate}
    \item for at least one $j\in\{1,\dots,m\}$, the map $\mathtt T_j$ has a hyperbolic fixed point $P_j\in Q$. We  denote by $W^u(P_j; \mathtt T_j)$ and $W^s(P_j; \mathtt T_j)$ the connected component in $Q$ of the local  unstable and stable manifolds (for the map $\mathtt T_j$),
    \item there exist families of cone fields $\mathcal C^u,\mathcal C^s$ which are common for all the maps $\mathtt T_i$, $i=1,\dots, m$.
\end{enumerate}
 In this setting we say that a $C^1$ submanifold $\Delta\subset Q$ is a cs-strip (resp. cu-strip) if its tangent bundle is contained in the cone $\mathcal C^s$ (resp. $\mathcal C^u$) 
\begin{defn}[Symbolic $cs$-blender]\label{defn:symboliccsblender}
    We say that the pair $(P_j,Q)$ is a \textit{symbolic $cs$-blender} for the IFS $\{T_i\}_{i=1,\dots,k}$ if for any cs-strip $\Delta\subset Q$ there exists $L\in\mathbb N$ and $\iota\in\{1,\dots,m\}^L$ such that 
    \[
    T_{\iota_{L-1}}^{-1}\circ\cdots\circ  T_{\iota_{0}}^{-1}(\Delta)\pitchfork W^u(P_j;\mathtt T_j)\neq \emptyset.
    \]
\end{defn}

Analogously, the pair $(P_j,Q)$ is a \textit{symbolic $cu$-blender} if for any cu-strip $\Delta\subset Q$ there exist $L$ and $\iota$ as above such that $T_{\iota_{L-1}}\circ\cdots\circ T_{\iota_0}(\Delta)\pitchfork W^s(P_j;\mathtt T_j)\neq \emptyset$.

\begin{defn}[Symbolic double blender]
    Let $(P_j,Q)$ be a symbolic $cu$-blender and $(P_i,Q')$ be a symbolic $cs$-blender for the IFS $\{T_i\}_{i=1,\dots,m}$. Together they form a \textit{symbolic double blender}  if 
    \[
    W^s(P_j;\mathtt T_j)\pitchfork W^u(P_i;\mathtt T_i)\neq\emptyset.
    \]
\end{defn}

Observe that symbolic double blenders act as some sort of ``mixer''. Indeed, using these objects it is possible to show that, given \textit{any} pair $\Delta\in Q$, $\Delta'\in Q'$ formed by a cs-strip in $Q$  and a cu-strip in $Q'$, the orbit of $\Delta$ under the IFS generated by $\{T_i\}_{i=1,\dots,m}$ intersects $\Delta'$. In particular, given any two open balls $B\in Q$ and $B'\in Q'$   we have that the orbit of $B$ intersects $B'$ (since we can foliate $B$ by cs-strips and $B'$  by cu-strips).

\medskip
We now introduce a suitable class of parametric families of diffeomorphisms of a $2d$-dimensional manifold (for any $d\in \mathbb N$) and give sufficient conditions which guarantee that the IFS that they generate displays symbolic double blenders. The conditions are local (i.e. only require control of the dynamics on a small region of the phase space) and, as we will see below, are general enough so that, in the context of Theorem \ref{thm:mainham}, we can check that the corresponding inner map \eqref{eq:innerdynamicsdefn} and system of scattering maps \eqref{eq:scattmapdynamicsdefn} verify these hypothesis.

To be precise, we consider a  pair of (parametric families of) maps $T_\varepsilon,S_\varepsilon$ of the form
\begin{equation}\label{eq:innermapsymbolic}
T_\varepsilon:\binom{\varphi}{J}\mapsto \binom{\varphi+\beta+A_\varepsilon J+R_{\varepsilon,\varphi}(\varphi,J)}{J+R_{\varepsilon,J}(\varphi,J)}
\end{equation}
and
\begin{equation}\label{eq:scattmapsymbolic}
S_\varepsilon:\binom{\varphi}{J}\mapsto \binom{\varphi+\varepsilon P_{\varepsilon,\varphi}(\varphi,J)}{J+\varepsilon P_{\varepsilon,J}(\varphi,J)}.
\end{equation}
We will suppose that $T_\varepsilon,S_\varepsilon$ are $C^2$ and satisfy the following:
\begin{enumerate}[label={[B\arabic*]}]
\item \label{it:itemB1} \textit{Rotation vector of the map $T_\varepsilon$:} The vector $\beta\in\mathbb R^d$ is of constant type, i.e. there exists $\gamma>0$ such that 
\begin{equation}\label{eq:constanttype}
\beta\in\mathcal B_\gamma=\{v\in\mathbb R^d\colon |v\cdot k-l|\geq \gamma |k|^{-d}\ \text{ for all }k\in\mathbb Z^d\setminus\{0\}\text{ and all }l\in\mathbb Z\}.
\end{equation}
   \item \label{it:itemB2} \textit{Small nonlinearity of the map $T_\varepsilon$}: The functions $R_{\varepsilon,\varphi},R_{\varepsilon,J}$ satisfy $ \partial_{J}^n R_{\varepsilon,\star}(\varphi,0)=0$ for $n\in\{0,1\}$ if $\star=\varphi$ and $n\in\{0,1,2\}$ if $\star=J$.
   \item \label{it:itemB3} \textit{Transversality-torsion}: Let 
   \[
  A=A_0\qquad\qquad  B=D_\varphi P_{0,J}(0,0).
   \]
   Then, the spectrum of the matrix $AB$ is simple, real, and does not contain zero. Moreover, $P_{0,J}(0,0)=0$.
\item \textit{Regularity with respect parameters} \label{it:itemB5} There exists  $\varepsilon_0>0$ such that $R_{\varphi,\star}$ and $P_{\varphi,\star}$  ($\star=\varphi,J$) depend $C^2$ on $\varepsilon$ for $-\varepsilon_0\leq\varepsilon\leq \varepsilon_0$.
\end{enumerate}
We denote by 
\begin{equation}\label{eq:C^2estimatesmapsIFSassumptions}
K=\max_{ |\varepsilon|\leq \varepsilon_0}\{|\partial_\varepsilon A_\varepsilon|,|R_{\varepsilon,\varphi}|_{C^2},|R_{\varepsilon,J}|_{C^3}, |P_{\varepsilon,\varphi}|_{C^2},|P_{\varepsilon,J}|_{C^2}\}.
\end{equation}

\begin{thm}[Symbolic blender]\label{thm:mainIFS}
Let $T_\varepsilon,S_\varepsilon$ be $C^2$ and satisfy assumptions \ref{it:itemB1}-\ref{it:itemB5},  and let $K$ be the constant in \eqref{eq:C^2estimatesmapsIFSassumptions}. Then, there exists $\varepsilon_0(\gamma,K)$ such that for any $0<\varepsilon\leq \varepsilon_0$ the IFS generated by $\{T_\varepsilon,S_\varepsilon\}$ displays a symbolic double blender. Moreover,   under the same conditions, this IFS is topologically mixing on the annulus 
\begin{equation}\label{eq:smallannulus}
    \mathbb A^d_{\sqrt\varepsilon}=\mathbb T^d\times [-\sqrt\varepsilon,\sqrt\varepsilon]^d.
\end{equation}
That is, for any open sets $B_0,B_1\in \mathbb A_{\sqrt\varepsilon}^d$ there exist $\widetilde L\in\mathbb N$ such that for any $L\geq \widetilde L$ there exists  $\omega\in\{1,2\}^{L}$ such that 
\[
G_{\omega_{L-1}}\circ\cdots \circ G_{\omega_0} (B_0)\cap B_1\neq \emptyset\qquad\qquad \text{ where }\quad G_{1}=T_{\varepsilon}\quad\text{and}\quad G_2=S_\varepsilon.
\]
\end{thm}
The proof of this result is presented in Section \ref{sec:blenderIFS}. It can be seen as a higher dimensional analogue of Theorem A in \cite{guardia2025partiallyhyperbolicdynamics3body} (see Remark \ref{rem:comparisonIFSresult} below and also Remark \ref{rem:lituraev} concerning the related work \cite{LiTuraevPreprint}). As in that paper, the basic idea is to implement a variant of the transversality torsion mechanism \cite{MR1949441} in which the transversality is arbitrarily small with respect to the arithmetic properties of the rotation vector. Namely,  locally around $\varphi=0=J$,  the family of maps $\{T^n_\varepsilon\circ S_\varepsilon\}_n$ can be described in terms of their linear approximation, given by the composition of the affine (positive) shear map
    \[
    \binom{\varphi}{J}\mapsto \bs b_n+\begin{pmatrix}
        \mathrm{id} &nA\\
        0&\mathrm{id}
    \end{pmatrix}\binom{\varphi}{J}
    \]
    with $\bs b_n=([n\beta],0)$ and the (negative) weak-shear map:
    \[
    \binom{\varphi}{J}\mapsto \begin{pmatrix}
        \mathrm{id}&0\\
        \varepsilon B&\mathrm{id}
    \end{pmatrix}\binom{\varphi}{J}
    \]
For $n\in\mathbb N$ large enough, the composition of these maps is hyperbolic under the assumption \ref{it:itemB3} (see Proposition \ref{lem:hyperbolicityiteratesIFS}). The key idea now is that, if we restrict our attention to the range of iterates $n\varepsilon\sim 1$:
\begin{enumerate}
    \item the hyperbolicity is of order one, so images of a small (suitably oriented) rectangle $Q^{cs}$ around $\varphi=0=J$ do not get very sequeezed and,
    \item thanks to the arithmetic assumption \ref{it:itemB1} on the rotation vector $\beta\in\mathbb R^d$, there exists a large subset of iterates within this range for which the image of $Q^{cs}$ comes very close and equidistribute fast around $\varphi=0=J$. In particular, the union of these images covers $Q^{cs}$ (see Figure \ref{fig:fig2}).
\end{enumerate}

\begin{figure}
    \centering
    \includegraphics[scale=0.55]{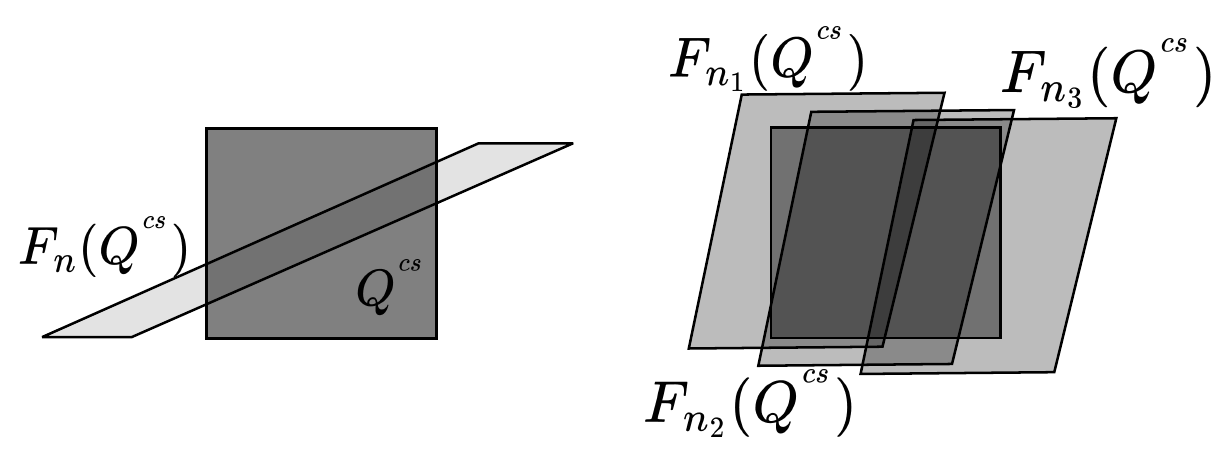}
    \caption{Let $F_n=T_0^n\circ T_1$. In the left we show the image of a small rectangle $Q^{cs}$ under the map $F_n$ for $n\gg 1/\varepsilon$. In the right we show the image of $D$ under $F_{n_i}$, $i=1,2,3$ for  $n_i\sim 1/\varepsilon$. In this regime the expansion/contraction is  close to one. Moreover, if $\beta$ is sufficiently irrational it is possible to chose $n_i$ such that the union  $\bigcup_{i}F_{n_i}(Q^{cs})$ contains $Q^{cs}$.}
    \label{fig:fig2}
\end{figure}

We then verify that in $Q^{cs}$ there exists a hyperbolic fixed point $P^{cs}$ (of one of the maps $T_\varepsilon^n\circ S_\varepsilon$) such that the pair $(Q^{cs},P^{cs})$ is a symbolic cs-blender for the IFS generated by $\{T_\varepsilon,S_\varepsilon\}$. A symmetry argument shows the existence of a symbolic cu-blender. From the construction both objects are homoclinically related and hence, they constitute a symbolic double blender.

\begin{rem}\label{rem:comparisonIFSresult}
    The main difference between Theorem \ref{thm:mainIFS} and Theorem A in \cite{guardia2025partiallyhyperbolicdynamics3body} is that here we allow any $d\in\mathbb N$. There are though a few technical differences between both results. First we note that  in Theorem A in \cite{guardia2025partiallyhyperbolicdynamics3body} the smallness quantity $\varepsilon_0$ is made explicit in terms of the torsion and the Diophantine constant, what allows for applications to degenerate situations in which these are also asymptotically small. Although we have not made it explicit to do not overcomplicate the statement, a quantitative smallness condition for the dependence of $\varepsilon_0$ on $\gamma$ and the matrix $B$ can be deduced from the proof of Theorem \ref{thm:mainIFS}.
    
    Second, in Theorem A in \cite{guardia2025partiallyhyperbolicdynamics3body}, transitivity is obtained on a larger domain (depending logarithmically on $\varepsilon$) at the expense of demanding more regularity. For our purposes, proving Theorem \ref{thm:mainham} (of this paper), the conclusion in Theorem \ref{thm:mainIFS} will be more than enough.
\end{rem}

\begin{rem}\label{rem:lituraev}
    The idea of working with small transversality first appeared in the work of Li and Turaev \cite{LiTuraevPreprint}, where, for $d=1$ the authors consider the case of a cubic tangency between the curve $\{J=0\}$ and its image under the map $S$. Even if \cite{LiTuraevPreprint} became available only very recently, the result and main ideas behind its proof were discussed by its authors some years ago at several conferences.
\end{rem}

\begin{rem}
    It has been shown by Koropecki and Nassiri \cite{MR2719428} that the IFS generated by a  $C^\infty$ generic pair of area-preserving maps of a compact surface is transitive. To the best of our knowledge it is still not known if the same conclusion holds with transitive replaced by robustly transitive. Also, the higher dimensional analogue of their result seems to be open (some progress has been made in \cite{MR3893271}).
\end{rem}
\medskip

\subsubsection*{Transport in IFS}

The symbolic double blenders in Theorem \ref{thm:mainIFS} produce a semi-local source of transitivity. In particular, they reduce the problem of global transitivity of the IFS generated by $\{T_\varepsilon,\{S_{i,\varepsilon}\}_{i=1}^m\}$ (recall that these are the maps in \eqref{eq:innerdynamicsdefn}-\eqref{eq:scattmapdynamicsdefn}) to the question of whether the blender is ``reachable'' from anywhere in the manifold $N$. Our next result shows that this is the case under a mild condition on the system of scattering maps. 

As we did above, we state this condition in abstract terms and then relate its assumptions to the system of scattering maps \eqref{eq:scattmapdynamicsdefn} which arise in the setting of Proposition \ref{prop:mainnhim}. More precisely, given $m\geq 2$, and a collection of (parametric families of) maps  $\{S_{i,\varepsilon}\}_{i=1}^m:N\to N$ satisfies:
\begin{enumerate}[label={[B\arabic*]},start=5]
    \item \label{it:B6} \textit{(Hörmander's  condition):} There exists $C>0$, $\tilde r>0$ $\varepsilon_0>0$ and a covering $\{U_j\}_j$ of $N$ by charts $\psi_j:U_j\subset \mathbb R^{2d} \to N$ such that, for any $0\leq \varepsilon\leq \varepsilon_0$, on each $z\in U_j$
    \[
    \psi_j^{-1}\circ S_{i,\varepsilon}\circ\psi_j(z)=z+\varepsilon X_i^{(j)}(z)+\varepsilon^2 R_i^{(j)}(z;\varepsilon)\qquad\qquad i=1,\dots,m
    \]
    for some $X_i^{(j)},R_i^{(j)}:U_j\to \mathbb R^{2d}$ such that $|X_i^{(j)}|_{C^{\tilde r}},|R_i^{(j)}|_{C^{\tilde r}}\leq C$ and at any $z\in U_j$
    \begin{equation}\label{eq:assumptionLiedistribution}
        \mathrm{Lie}^{(\tilde r)}_z\{X_1^{(j)},\dots, X_m^{(j)}\}=\mathbb R^{2d}
    \end{equation}
    where, for each $1\leq s\leq \tilde r$, given $m$ vector fields $\{Y_i\}_{i=1}^m:U_j\to \mathbb R^{2d}$ 
    \begin{equation}\label{eq:Liedistributions}
    \mathrm{Lie}^{(s)}_z\{Y_1,\dots, Y_m\}=\mathrm{Lie}_z^{(s-1)}\{Y_1,\dots,Y_m\}+[\mathrm{Lie}^0\{Y_1,\dots,Y_m\},\mathrm{Lie}^{(s-1)}\{Y_1,\dots,Y_m\}]_z
    \end{equation}
    and $\mathrm{Lie}_z^0\{Y_1,\dots,Y_m\}=\mathrm{span}\{Y_1,\dots,Y_m\}$.
\end{enumerate}

Hypothesis \ref{it:B6} is nothing but a suitably adapted reformulation of those in the classical Ball-Box Theorem in sub-Riemannian geometry (see \cite{MR1867362}).  From an abstract viewpoint it  can be understood as a sufficient condition on a collection of close to identity diffeomorphisms of a compact manifold $N$ which guarantees $\varepsilon$-\textit{Reachability} of the corresponding IFS (see Proposition \ref{prop:maintransportIFS} for the precise formulation). For our purposes (giving a proof of Theorem \ref{thm:mainham}), assumption \ref{it:B6} will be understood jointly with assumptions \ref{it:H1} and \ref{it:H2}. Hence, the regularity class $\tilde r$ in assumption \ref{it:B6} will be determined by that in the strong spectral gap condition in assumption \ref{it:H1}.

\begin{rem}
    As far as we know, a version of hypothesis \ref{it:B6} first appeared in the work of Chow \cite{MR1880} in sub-Riemannian geometry. Later, this condition gained popularity in the work of Hörmander \cite{MR222474}, where a version of hypothesis \ref{it:B6}  is shown to be a sufficient condition for the  smoothness of transition probabilities associated to a class of   diffusion processes.
\end{rem}

\begin{prop}[$\varepsilon$-Reachability] \label{prop:maintransportIFS} Suppose that, a collection of diffeomorphisms $\{S_{i,\varepsilon}\}_{i=1}^m:N\to N$  satisfies assumption \ref{it:B6}  and preserves a smooth measure.  Then, there exists $C>0$ and $\varepsilon_0>0$ such that the following holds.  For any $0<\varepsilon\leq \varepsilon_0$ there exists a set $\mathcal R_\varepsilon\subset N$ of full measure such that for any $z,z_*\in N$ with $z\in \mathcal R_\varepsilon$ there exists $L\in\mathbb N$ and $\iota\in \{1,\dots,m\}^L$ such that 
\[
\mathrm{dist}(S_{\varepsilon,\iota_L}\circ\cdots\circ S_{\varepsilon,\iota_0}(z),\ z_*)\leq C\varepsilon.
\]
\end{prop}

\begin{rem}
    Another popular name in the control literature for the property that we call $\varepsilon$-Reachability is ``$\varepsilon$-Accesibility''. However, in the context of partially hyperbolic dynamics, the latter is usually reserved for a different property involving the distributions generated (under brackets) by  the strong stable and unstable distributions (see \cite{MR2039999}).
\end{rem}

The proof of Proposition \ref{prop:maintransportIFS} is given in Section \ref{sec:transporIFS}. The idea, similar to that in the proof of the Ball-Box theorem, boils down to the observation that suitable commutators of the pair of diffeomorphisms $\{S_{i,\varepsilon}\}_{i}$ can be used to approximate the set displacements generated by the time-$\varepsilon$ flow of vector fields forming a basis of $\mathrm{Lie}^{(\tilde r)}\{X_1,\dots,X_m\}$.

\begin{rem}
While discussing this work with  M. Gidea, R. de la Llave and T.M. Seara they  have brought to our attention that in their project in preparation \cite{HormanderGdlLS}  they have successfully implemented a similar idea to that in Proposition \ref{prop:maintransportIFS} in the context of Arnold diffusion and that they had already announced such a result a number of years ago in \cite{MR4160091}. 
\end{rem}

\medskip

\subsection{Robustly transitive skew-product dynamics}\label{sec:skewprodintro}

In Section \ref{sec:skewprod} we will  deduce Theorem \ref{thm:mainham} as a corollary of Proposition \ref{prop:mainnhim}, Theorem \ref{thm:mainIFS}, Proposition \ref{prop:maintransportIFS} and a shadowing argument. The main idea behind this shadowing argument boils down to the use of Poincar\'e's recurrence theorem to insert long strings of iterations of the local map $T_\varepsilon$ in the codes constructed in Theorem \ref{thm:mainIFS} and Proposition \ref{prop:maintransportIFS}. This idea is certainly not new and has already  been  exploited in a number of related works (see \cite{NassiriPujalsTransitivity}, \cite{GelfreichTuraev} and \cite{MR4033892}). 

However, we must face a number of (technical) difficulties when compared to the applications of these arguments in these previous works. For instance by making use of inherently $C^\infty$ perturbation techniques, the authors in \cite{NassiriPujalsTransitivity} can assume that the skew-product dynamics is locally constant, so the results for IFSs can be readily applied. On the other hand,  the shadowing arguments in \cite{GelfreichTuraev,MR4033892} are essentially one-sided while ours is (necessarily) two-sided. We will prove the following in Section \ref{sec:skewprod}.

\begin{prop}\label{prop:skewprodmain}
Let $F_\varepsilon:X\to X$ be a parametric family of $C^r$ maps $(r\geq 2)$ of a smooth Riemannian manifold which satisfy \ref{it:H1}-\ref{it:H2} and let  $T_\varepsilon:N\to N$ and  $\{S_i\}_{i=1}^m:N\to N$ be given by \eqref{eq:innerdynamicsdefn} and \eqref{eq:scattmapdynamicsdefn}. Suppose that there exists $\tilde\varepsilon>0$ such that for all $0<\tilde\varepsilon\leq \varepsilon_0$:
\begin{enumerate}
\item The maps $T_\varepsilon$ and $\{S_i\}^m_{i=1}$ preserve a common\footnote{Of course, this measure might depend on $\varepsilon$.} smooth measure on $N$,
    \item There exists a subset $\mathcal A\subset N$ diffeomorphic to the $2d$-dimensional annulus $\mathbb A^d=\mathbb T^d\times[-1,1]^d$ (here $2d=\mathrm{dim}(N)$) on which the maps $\{T_\varepsilon, S_{1,\varepsilon}\}$ are of the form \eqref{eq:innermapsymbolic} and \eqref{eq:scattmapsymbolic} and verify hypothesis \ref{it:itemB1}-\ref{it:itemB5} with $\gamma=\gamma(F_0)$.   \item The scattering maps $\{S_{i,\varepsilon}\}_{i=1}^m$ satisfy assumption \ref{it:B6} for some $0\leq \tilde r\leq r$. 
    
\end{enumerate}
Then, there exists $\varepsilon_0>0$ such that for any $0<\varepsilon\leq \varepsilon_0$ the corresponding skew-product $\mathcal F_\varepsilon:\Sigma\times N\to \Sigma\times N$ in \eqref{eq:skewprodnhim} is topologically mixing. 
\end{prop}

\subsection{Typicality of the assumptions}\label{sec:genericityintro}

Finally, we observe that hypothesis \ref{it:H1}-\ref{it:H2} and \ref{it:itemB1}-\ref{it:B6} are typically satisfied for  one-parametric families $\{F_{\varepsilon,f}\}_\varepsilon:M\times N\to M\times N$ of real-analytic symplectic maps unfolding from a direct product $F_0$ as in Theorem \ref{thm:mainham}.

\begin{prop}\label{prop:genericityofassumptions}
  Let $F_0\in \mathrm{Ham}^\omega(M\times N)$ be a direct product  satisfying \ref{it:Assumption1}-\ref{it:Assumption2}  for some $\kappa>0$ large enough (depending on the dimension of $N$)  and which, in addition, displays a collection of full disjoint homoclinic channels $\{\Gamma_0^i\}_{i=1}^m\subset W^s(\Lambda_0)\pitchfork W^u(\Lambda_0)$ for some $m$ large enough (depending on the dimension of $N$).  Let $r_0$ be large enough (depending only on the dimension of $N$).
  
  Then, there exists $\gamma(F_0)>0$ and   $\varepsilon_0(F_0)>0$  such that for any  $f$ belonging to an open and dense subset of $\mathcal B^\omega_{r_0}(M\times N)$, the unfolding  $F_{\varepsilon,f}\in \mathrm{Ham}^\omega(M\times N)$ \eqref{eq:deformationmap} satisfies the following for any $0<\varepsilon\leq \varepsilon_0$:
    \begin{enumerate}
        \item \label{it:Item1genericity} $F_{\varepsilon,f}$ displays a normally hyperbolic invariant manifold $\Lambda_\varepsilon$ for which assumptions \ref{it:H1}-\ref{it:H2} are verified. The corresponding inner map  $T_\varepsilon:N\to N$ in \eqref{eq:innerdynamicsdefn} and scattering maps (associated to the continuation of $\{\Gamma_0^i\}_{i=1}^m$)  $\{S_{i,\varepsilon}\}_{i=1}^m:N\to N$ in \eqref{eq:scattmapdynamicsdefn}, preserve a  common smooth measure on $N$.
        \item \label{it:Item2genericity} There exists a subset $\mathcal A\subset N$ diffeomorphic to the $d$-dimensional annulus $\mathbb A^d=\mathbb T^d\times[-1,1]^d$ (here $2d=\mathrm{dim}(N)$) on which the maps $\{T_\varepsilon, S_{1,\varepsilon}\}$ are of the form \eqref{eq:innermapsymbolic} and \eqref{eq:scattmapsymbolic} and verify hypothesis \ref{it:itemB1}-\ref{it:itemB5} with $\gamma=\gamma(F_0)$.
        \item \label{it:Item3genericity} The scattering maps $\{S_{i,\varepsilon}\}_{i=1}^m$ satisfy assumption \ref{it:B6}. 
    \end{enumerate}
\end{prop}

We present the proof of Proposition \ref{prop:genericityofassumptions} in Section \ref{sec:genericity}. Here we just give a heuristic idea of why these assumptions are abundant around maps $F_0$ as that in Theorem \ref{thm:mainham}.
\begin{rem}
     The very same conclusion in Proposition \ref{prop:genericityofassumptions} holds if $F_0\in \mathrm{Symp}^\omega(M\times N)$ is a direct product of symplectomorphisms satisfying assumptions \ref{it:Assumption1}-\ref{it:Assumption2} and we consider an unfolding $F_{\varepsilon,\theta}=F_0\circ\Phi^{\varepsilon}_\theta\in \mathrm{Symp}^\omega(M\times N)$
   with $\Phi_\theta^{\varepsilon}$ as in \eqref{eq:unfoldingforsymplectic}. See Section \ref{sec:genericity} for more details.
\end{rem}

First, Item \ref{it:Item1genericity} is straightforward since, by assumption, for $F_0$ there exists a normally hyperbolic cylinder $\Lambda_0=\{\gamma_0\}\times N$ which is already $\kappa$-normally hyperbolic and $\kappa$-pinched and  exhibits one full homoclinic channel. It is straightforward to observe that, in this setting this implies that there actually exists a countable collection of disjoint full homoclinic channels. In particular, for any $m\in\mathbb N$ we can select a family of full and disjoint homoclinic channels $\{\Gamma_0^i\}_{i}\subset M\times N$. The corresponding results for $\varepsilon>0$ follow by regular perturbation theory (see \cite{MR501173}, \cite{DelaLLaveScattmap} or  Section \ref{sec:genericity}).  The fact that the inner and scattering map preserve a common smooth measure is established in \cite{DelaLLaveScattmap}.

Concerning Item \ref{it:Item2genericity}, since we now assume that $r_0$ is large enough and that \ref{it:Assumption1} holds with a sufficiently large exponent, the normally hyperbolic invariant manifold $\Lambda_\varepsilon$ can be assumed to be $C^k$-close to $\Lambda_0$ for $k$ as large as needed. Then, a standard argument shows that close to the non-degenerate elliptic equilibrium,  there exists a $d$-dimensional torus which, for $\varepsilon>0$ small enough, admits a continuation as a KAM torus $\mathcal T_\varepsilon\subset\Lambda_\varepsilon$ (see \cite{PoschelKAM} or Section  \ref{sec:genericity}). Around this torus the inner dynamics  can be put in normal form \eqref{eq:innermapsymbolic}. This gives us the assumptions \ref{it:itemB1}-\ref{it:itemB2} and \ref{it:itemB5}. On the other hand, for the assumptions involving the scattering map associated to the homoclinic channel $\Gamma_\varepsilon^1\subset X$, we will make use of the perturbative tools developed in \cite{DelaLLaveScattmap}. Indeed, the results obtained  in \cite{DelaLLaveScattmap}, expressed in terms of deformation theory, guarantee that, in our setting, the scattering maps associated to the homoclinic channels $\{\Gamma_\varepsilon^i\}_i$ (given by the continuation of $\{\Gamma_0^i\}_i$) are close to identity symplectic diffeomorphisms. Moreover, using the asymptotic expansions in \cite{DelaLLaveScattmap}, it is not difficult to check that the transversality-torsion assumption in \ref{it:itemB3} holds for perturbations $f$ from an open and dense subset of $\mathcal B^\omega_{r_0}(M\times N)$.

Finally, Item \ref{it:Item3genericity}, i.e. genericity of assumption \ref{it:B6}, can be easily deduced by means of the asymptotic expansions in \cite{DelaLLaveScattmap} and a suitable parametric transversality argument.

\subsubsection*{Towards applications to given parametric families}

In Proposition \ref{prop:genericityofassumptions} we fix $m\in \mathbb N$  and look at the scattering maps associated to the continuation of $m$ different homoclinic channels $\{\Gamma^i_0\}_{i=1}^m\subset X$. Provided $m\gg \mathrm{dim}(N)$ it is easy to check (see Section \ref{sec:genericity}) that for an open and dense set of deformations, the corresponding scattering maps $\{S_{i,\varepsilon}\}_{i=1}^m:N\to N$ satisfy \ref{it:B6} with $\tilde r=0$, that is, their corresponding vector fields already span $\mathbb R^{2d}$ (so it is not necessary to look at higher order distributions as in \eqref{eq:Liedistributions}).
\medskip

However, in applications to concrete parametric families it is typically difficult to have a precise description of a large family of homoclinic channels and one often only looks at primary homoclinic channels (corresponding to first enocounter intersections of the stable and unstable manifolds). In those situations, it seems more appropriate to fix a small number of homoclinic channels $m\geq 2$ and look at the corresponding higher order distributions \eqref{eq:Liedistributions}. In that scenario, the verification of \ref{it:B6} seems feasible by computer assisted technology, or, in some cases, by asymptotic perturbation techniques. To support this belief we also prove the following result (see Appendix \ref{sec:hormander}).

\begin{prop}\label{prop:genericityhormanderwithtwomaps}
The same conclusion in  Proposition \ref{prop:genericityofassumptions} holds for any $m\geq 2$.
\end{prop}

\subsection{Proof of Theorem \ref{thm:mainham}}

Proposition \ref{prop:genericityofassumptions} shows that for $F_{0}$ as in Theorem \ref{thm:mainham}, assumptions \ref{it:H1}-\ref{it:H2} are verified and the corresponding inner map  $T_\varepsilon:N\to N$ and scattering maps $\{S_i\}_{i=1}^m:N\to N$ satisfy \ref{it:itemB1}-\ref{it:B6}. The fact that these maps preserve a common smooth  measure is standard (see \cite{DelaLLaveScattmap} or Section \ref{sec:genericity}). In this setting, Proposition  \ref{prop:mainnhim} and Proposition \ref{prop:skewprodmain} imply that there exists a set $\mathcal X_\varepsilon\subset M\times N$, an open neighbourhood $\mathcal Q_\varepsilon\supset \mathcal X_\varepsilon$   and a homeomorphism $\Psi:\Sigma\times N\to \mathcal X_\varepsilon$ such that:
 \begin{enumerate}
        \item $\mathcal X_\varepsilon$ is locally maximal and normally hyperbolic for the return map $\mathcal R_\varepsilon$ induced on $\mathcal Q$,
        \item $\mathcal X_\varepsilon$ contains the continuation of $\Lambda_0$ in its closure,
        \item The restriction  $\mathcal R_\varepsilon|_{\mathcal X_\varepsilon}:\mathcal X_\varepsilon\to \mathcal X_\varepsilon$ is robustly topologically mixing.
    \end{enumerate}
The proof of Theorem \ref{thm:mainham} is complete.
\medskip

\section{Existence of normally hyperbolic laminations}\label{sec:nhimproof}

In this section we consider a $C^r$ one-parametric family of  diffeomorphisms $\{F_\varepsilon\}_\varepsilon$ of a smooth Riemannian manifold $X$ (with $r\geq 2$)  which satisfies assumptions \ref{it:H1}-\ref{it:H2} and construct a normally hyperbolic lamination accumulating on the corresponding homoclinic channels $\{\Gamma_\varepsilon^i\}_i\subset W^s(\Lambda_\varepsilon)\pitchfork W^u(\Lambda_\varepsilon)$. Namely, we give a proof of Theorem \ref{prop:mainnhim}. 

\begin{rem}
 As already remarked in Section \ref{sec:nihimintro} (see the discussion after \ref{it:H1}-\ref{it:H2} are introduced), assumptions \ref{it:H1}-\ref{it:H2} are open and \textit{uniform in $\varepsilon$} (provided $\varepsilon>0$ is sufficiently small depending only on $F_0$). Hence, all of our constructions (and quantitative estimates) in this section are also uniform in $\varepsilon$. For that reason, we \textit{fix any value} $0\leq \varepsilon\leq \varepsilon_0(F_0)$ for which \ref{it:H1}-\ref{it:H2} hold and drop the dependence on $\varepsilon$ from the notation.
\end{rem}
\medskip

We will construct a return map to a neighbourhood of the channels $\{\Gamma^i\}_i\subset X$ by gluing the local dynamics around $\Lambda$ together with the outer dynamics along the orbits of the homoclinic channels. As the qualitative description of both regimes is rather different, we split this section in several subsections. First, in Section \ref{sec:straighteningnormalfibration} we introduce local coordinates in a neighbourhood of $\Lambda$ in which the  normal fibration is \textit{straightened}. This will simplify considerably the coordinate description of the maps below.
Second, in Section \ref{sec:localmap} we study the \textit{local} dynamics in a neighbourhood of $\Lambda$. In Section \ref{sec:outermap} we describe the \textit{outer} dynamics, i.e. of orbits following closely the orbit of one of the homoclinic channels $\Gamma^i$. Then, in Section \ref{sec:globalmaps} we \textit{patch together} the local and outer maps to build a collection of maps $\{\Phi_{i\to j}^{(n)}\}_n$ defined on a suitable neighbourhood of $\bigcup_i\Gamma^i$ and  \textit{indexed by their return time} to this neighbourhood (these will be later used to piecewise define the return map to a neighbourhood of $\bigcup_i\Gamma^i$). In Section \ref{sec:partiallygraphtransform} we study the action of each $\Phi_{i\to j}^{(n)}$ on a suitable class of codimension one submanifolds, which we name partially horizontal and partially vertical. From this analysis we will deduce that the \textit{return map} that $F$ induces on an open set  $\mathcal Q\subset X$ which contains $\bigcup_i\Gamma^i$ on its closure, admits a \textit{Markov partition} with infinitely many symbols and, in Section \ref{sec:Returnmap} conclude the proof of Proposition \ref{prop:mainnhim}.

 \medskip

\subsubsection*{Remarks on the quantitative relations between the parameters}
Throughout this section we will introduce two auxiliary parameters $0<\delta\ll \delta_0\ll 1$ specifying the shape of the domain in which we will define the aforementioned return map ($\delta_0$ locates the center while $\delta$ determines the size). We will find $n_*\in \mathbb N$ (depending only on $\delta_0$) such that this return map is defined piecewise in terms of $n$-th iterates of the map $F_\varepsilon$ with $n\geq n_*$.

Throughout this section, given two functions $f(n,z),h(n,z)$ we write $f\lesssim h$ if there exists $C>0$ such that for all $n\in\mathbb N$ sufficiently large  and all $z\in N$ we have $f(n,z)\leq C h(n,z)$. 
\medskip

\subsubsection*{Remark on the dimension of $M$} With the intention of avoiding unnecessary technicalities, we will suppose that $\mathrm{dim} M=2$. The reader can easily adapt the proof to the general case.

\medskip

\subsection{Straightening the normal fibration}\label{sec:straighteningnormalfibration}
By \ref{it:H1},  the $C^2$ diffeomorphism $F:X\to X$ possesses a compact normally hyperbolic invariant manifold $\Lambda$ on which the strong spectral gap condition \ref{eq:spectralgap} is satisfied for some $\kappa\geq 2$. Recall that this implies that $\Lambda$ is a $C^2$ submanifold and that $W^{u,s}(\Lambda)$ are $C^2$ (see Section \ref{sec:outline}).
Proceeding as Section 4.1. in \cite{MR4521949} (alternatively, see the proof of Theorem 1, case 3 in \cite{MR283825} or Section 5.1. in \cite{MR2163534}), for any $\delta_0>0$ small enough there exists a neighbourhood $U_{\delta_0}$ of $\Lambda$ and a  $C^{2}$ local coordinate system
\begin{equation}\label{eq:localcoordinatesnhim}
\psi_{\mathrm{loc}}:(q,p,z)\in [-2\delta_0,2\delta_0]^2\times N\mapsto x\in U_{\delta_0}\subset X
\end{equation}
in which
\[
\Lambda=\psi_{\mathrm{loc}}(\{q=p=0\})
\]
and  the coordinate expression for $F$ is given by 
\begin{equation}\label{eq:mapF}
\Phi_{\mathrm{loc}}=\psi_{\mathrm{loc}}^{-1}\circ F\circ \psi_{\mathrm{loc}}:\begin{pmatrix}
    q\\p\\z
\end{pmatrix}\mapsto \begin{pmatrix}
    \lambda(z) q +Q(q,p,z)\\
    \mu (z) p+P(q,p,z)\\
    T(z)
\end{pmatrix}\qquad\qquad 
\end{equation}
where $T:N\to N$ is the map in \eqref{eq:innerdynamicsdefn} and:
\begin{enumerate}
    \item \textit{Domination for the linear dynamics:}  The constants
\begin{equation}\label{eq:maxminlambda}
\bar \lambda=\max_{z\in N} \{\lambda(z),1/\mu(z)\},\qquad 
   \qquad \underline \lambda=\min_{z\in N} \{\lambda(z),1/\mu(z)\}\end{equation}
   and
   \begin{equation}\label{eq:maxminalpha}\bar\alpha=\bar\lambda /\underline\lambda,\qquad\qquad \alpha=\max_{z\in N}\{\lVert DT\rVert,\lVert DT^{-1} \rVert\}
\end{equation}
satisfy
\[
\text{(Bunching)}\quad\alpha^{\kappa+1}\bar\lambda<1,\qquad\qquad\text{and}\qquad\qquad \text{(Pinching)}\quad\bar\alpha^{\kappa+1}\bar\lambda<1.
\]
 \item \textit{Small nonlinear remainder:} The functions $Q,P$ are $C^2$ and at any $(q,p,z)\in [-2\delta_0,2\delta_0]^2\times N$ we have 
 \[
 Q(0,p,z)=0=\partial_qQ(0,0,z)\qquad\qquad  P(q,0,z)=0=\partial_pP(0,0,z).
\]

\end{enumerate}
\subsection{The local map}\label{sec:localmap}

In the following proposition we give a precise description of the dynamics of $F:X\to X$ around $\Lambda$. A statement with a similar flavor can be found in Lemma 1 of \cite{GelfreichTuraev} (see also Lemma 3.1 in \cite{LiTuraevPreprint}).
\begin{rem}
    In Lemma 1 of \cite{GelfreichTuraev} the authors only assume that the spectral gap condition \eqref{eq:spectralgap} holds with $\kappa\geq 1$ (while we assume $\kappa\geq 2$ plus an extra pinching hypothesis) but obtain less detailed information about the orbit segment $\{q_i,p_i,z_i\}_i$. In particular, in that setting it does not seem to be possible to obtain quantitative $C^1$ estimates (compare \eqref{eq:quantitativedecayorbitnearnhim}-\eqref{eq:quantitativedecayorbitnearnhim2} to estimates (19) in Lemma 1 of \cite{GelfreichTuraev}). These quantitative estimates will be important in Section \ref{sec:skewprod}.
\end{rem}

\begin{prop}\label{prop:localdynamics}
    Fix any $w=( q ,  \bar p ,\bar z)\in [-\delta_0,\delta_0]^2\times N$. Then, provided $\delta_0>0$ is small enough,  for any $n\in\mathbb N$ large enough  there exists a unique orbit segment $\{q_i,p_i,z_i\}_{i=0}^n$ of the map $\Phi_{\mathrm{loc}}$ in \eqref{eq:mapF} satisfying $q_0= q $, $p_n=  \bar p $ and $z_n=\bar z$. Moreover, if we define 
    \[
    \lambda^{(n)}(\bar z)=\prod_{i=1}^{n} \lambda (T^{-i}(\bar z))\qquad\qquad \mu^{(n)}(\bar z)=\prod_{i=1}^{n} \mu (T^{-i}(\bar z))
    \]
    the asymptotic formulas
    \begin{equation}\label{eq:asymptoticslocal}
    q_n( q ,  \bar p ,\bar z)=\lambda^{(n)}(\bar z)  q + h( q ,  \bar p ,\bar z)\quad\quad y_0( q ,  \bar p ,\bar z)=\frac{1}{\mu^{(n)}(\bar z)}  \bar p +g( q ,  \bar p ,\bar z)\quad\quad z_0(q ,  \bar p ,\bar z)=T^{-n}(\bar z)
    \end{equation}
    hold with $h,g$ satisfying the pointwise estimates 
    \begin{equation}\label{eq:quantitativedecayorbitnearnhim}
    \frac{1}{\delta_0}|h|,|\partial_ q  h|,|\partial_  {\bar p}  h|\lesssim \delta_0 \lambda^{(n)}\qquad\qquad  \frac{1}{\delta_0}|g|,|\partial_ q  g|,|\partial_  {\bar p}  g|\lesssim \delta_0/ \mu^{(n)}
    \end{equation}
    and 
    \begin{equation}\label{eq:quantitativedecayorbitnearnhim2}
    |\partial_{\bar z} h|\lesssim \delta^2_0 \alpha^n\lambda^{(n)} \qquad\qquad |\partial_{\bar z} g|\lesssim \delta^2_0  \alpha^n/\mu^{(n)}.
    \end{equation}
\end{prop}

The proof of Proposition \ref{prop:localdynamics} follows from a standard fixed point argument and is deferred to Appendix \ref{sec:prooflocaldynamics}.

\begin{figure}
    \centering
    \includegraphics[scale=0.45]{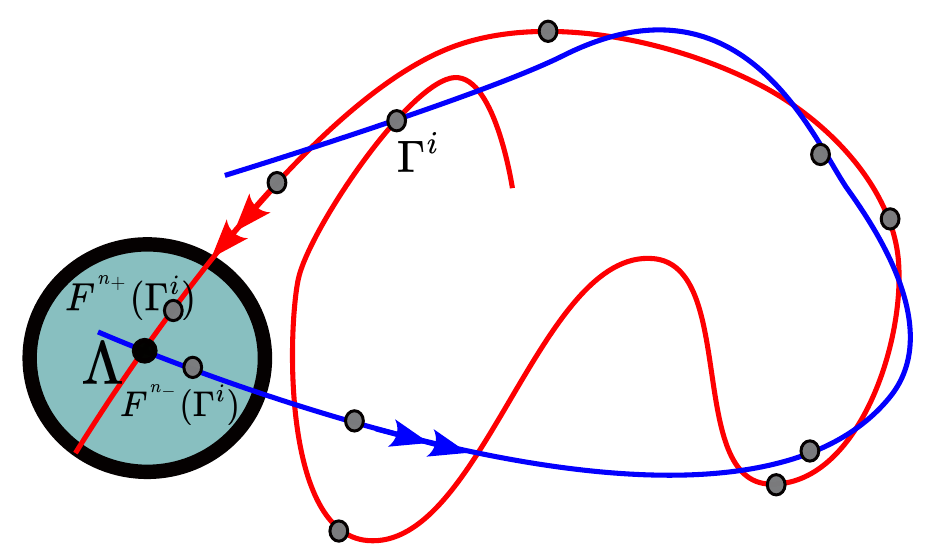}
    \caption{A homoclinic channel $\Gamma^i\subset W^s(\Lambda)\pitchfork W^u(\Lambda)$ and its orbit.}
    \label{fig:fig3}
\end{figure}

\subsection{The outer map}\label{sec:outermap}

We fix any $\delta_0>0$ small enough such that Proposition \ref{prop:localdynamics} holds. Let $m\geq 1$ be as in Assumption \ref{it:H2}. Then, let $n_\pm\in\mathbb N$ be any natural  numbers for which
\[
\{F^{n_+}(\Gamma^i)\}_{i=1}^m\subset U_{\delta_0/2}\qquad\qquad\text{and}\qquad\qquad \{F^{-n_-}(\Gamma^i)\}_{i=1}^m\subset  U_{\delta_0/2}
\]
and define the submanifolds (see Figure \ref{fig:fig3})
\[
\Gamma^{i}_{\pm}:=F^{\pm n_\pm}(\Gamma^i)\subset W^s(\Lambda)\pitchfork W^u(\Lambda).
\]
The purpose of this section is to study the map $ F^{(n_-+n_+)}:X\to X$
    when restricted to small neighbourhoods around $\{\Gamma^{i}_-\}_i$ with  $i=1,\dots,m$. We start by chosing suitable neighbourhoods of $\{\Gamma^{i}_{\pm}\}_i\subset X$. To do so, observe that from their very definition,  we can deduce the existence of  $C^{2}$ functions ($i=1,2$)
\begin{align*}
q^{i}_{+}:N &\to [0,\delta_0/2] &p^{i}_{-}: N &\to [0,\delta_0/2]\\
z&\mapsto q^{i}_{+}(z) &z&\mapsto p^{i}_{-}(z)
\end{align*}
such that 
\[
\Gamma^{i}_{+}=\{(q^i_+(z),0,z)\colon z\in N\}\qquad\qquad \Gamma^i_-=\{(0,p^i_-(z),z)\colon z\in N\}.
\]
Moreover, for $0<\delta\ll \delta_0\ll 1$ small enough, we can find functions 
\begin{align*}
q^i_u:[-\delta,\delta]\times N&\to [0,\delta_0]& p^i_s:[-\delta,\delta]\times N&\to [0,\delta_0]\\
(p,z)&\mapsto q^{i}_{u}(p,z) &(q,z)&\mapsto p^{i}_{s}(q,z)
\end{align*}
such that for $i=1,2$:
\begin{enumerate}
    \item There exists a compact piece $\widetilde W^{u}_i(\Lambda)\subset W^u(\Lambda)$ (resp. $\widetilde W^{s}_i(\Lambda)\subset W^s(\Lambda)$) which can be parametrized as 
    \begin{align*}
      \widetilde W^u_i(\Lambda)=&\{(q^i_u(p,z),p,z)\colon (p,z)\in [-\delta,\delta]\times N\}\\
      (\text{resp. } \widetilde W^s_i(\Lambda)=&\{(q,p^i_s(q,z),z)\colon (q,z)\in [-\delta,\delta]\times N\}),
    \end{align*}
   
    \item $\widetilde W^u_i(\Lambda)\pitchfork W^s_{\mathrm{loc}}(\Lambda)=\Gamma^i_-$ (resp. $\widetilde W^s_i(\Lambda)\pitchfork W^u_{\mathrm{loc}}(\Lambda)=\Gamma^i_+$) so, in particular, 
    \[
    q^i_u(0,z)=q^i_+(z)\qquad\qquad (\text{resp. } p^i_s(0,z)=p^i_-(z))
    \]
    and 
    \[
    \partial_p q ^{i}_{u}(0,z)\neq 0\qquad\qquad (\text{resp. } \partial_q p^{i}_{s}(0,z)\neq 0)
    \]
\end{enumerate}
Then, for $0<\delta\ll \delta_0\ll 1$  as above,  we define the small regions  (see Figure \ref{fig:fig4})
\begin{align*}
\mathcal Q^i_{\delta}=&\psi_{\mathrm{loc}}\left(\{(q,p,z)\in [-2\delta_0,2\delta_0]^2\times N\colon 0\leq p-p_{s}^{i}(q,z)\leq \delta,\  0\leq q\leq \delta\}\right)\subset X
\end{align*}
around the submanifolds $\{\Gamma_-^i\}_i$ (i.e. the $n_-$-backward iterates of the homoclinic channels $\{\Gamma^i\}_i$).

\medskip

We now study the map $ F^{(n_-+n_+)}:X\to X$ restricted to the sets  $\mathcal Q^i_{\delta}$. The first observation is that, by construction, 
\[
\Gamma^i_+\subset F^{(n_-+n_+)}(\mathcal Q_{\delta}^{i})
\]
To obtain more quantitative information it will prove convenient to introduce the local coordinate patches
\begin{equation}\label{eq:localcoordinateshomchannel}
\begin{split}
\psi_i:[0,\delta]^2\times N&\to \mathcal Q^i_\delta\subset X\\
(q,p,z)&\mapsto \psi_{\mathrm{loc}}(q,p^i_s(q,z)+p,z).
\end{split}
\end{equation}
which cover an open set containing of the submanifolds $\{\Gamma^i_-\}_i\subset X$ on their closure.

\begin{figure}
    \centering
    \includegraphics[scale=0.55]{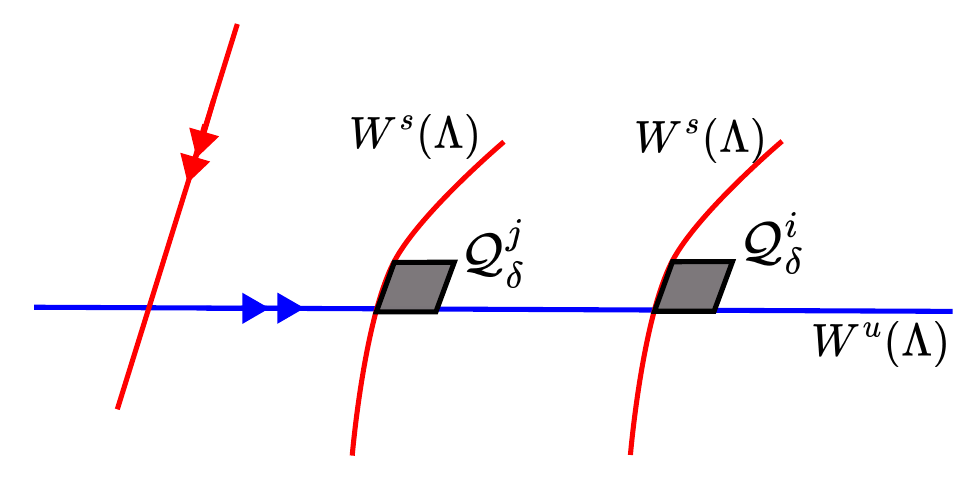}
    \caption{The sets $\mathcal Q_\delta^i$ for $i=1,\dots,m$.}
    \label{fig:fig4}
\end{figure}

\begin{lem}\label{lem:outerdynamics}
Let $i=1,\dots,m$. In the local coordinate systems given by $\eqref{eq:localcoordinatesnhim}$ on $U_{\delta_0}$ and   \eqref{eq:localcoordinateshomchannel} on $\mathcal Q^i_\delta$ the map 
\[
F^{(n_-+n_+)}:\mathcal Q^i_\delta\to U_{\delta_0} 
\]
is of the form 
    \[
    \Phi_{\mathrm{out},i}=\psi_{\mathrm{loc}}^{-1}\circ F^{(n_-+n_+)}\circ\psi_i:\begin{pmatrix}
        q\\p\\z
    \end{pmatrix}\mapsto \begin{pmatrix}
       \begin{pmatrix}
        \tilde q_{i}^+(p,z)\\0
    \end{pmatrix}+ A_i(z)\begin{pmatrix}
            q\\ p
        \end{pmatrix}+\begin{pmatrix}
            \widetilde Q_i(q,p,z)\\ p\widetilde P_i(q,p,z)
        \end{pmatrix}\\ S_{i}(z)+ \widetilde S_i(q,p,z)
    \end{pmatrix}
    \]
    with $A_i(z)$ being a two dimensional matrix which for all $z\in N$ is of the form 
    \[
    A_i(z)=\begin{pmatrix}
        a_i(z)&b_i(z)\\
        0&c_i(z)
    \end{pmatrix}\qquad\qquad\qquad  c_i(z)> 0,\qquad \mathrm {det} A_i(z)\neq 0,
    \]
     $S_i:N\to N$ being the coordinate expressions of the scattering maps in \eqref{eq:scattmapdynamicsdefn}, the function $\tilde q^+_i(q,z)=q^+_i(z)+b_\varepsilon(z)p_i^s(q,z)$, and with $\widetilde Q_i,\widetilde S_i$ being $C^2$, $\widetilde P_i$ being $C^1$ and satisfying 
    \begin{equation}\label{eq:propertiesfunctionsoutermap}
  \widetilde Q_i(0,p,z)=0=\widetilde Q_i(q,0,z),\qquad\qquad \widetilde P_i(0,0,z)=0,\qquad\qquad \widetilde S_i(0,0,z)=0.
    \end{equation}
\end{lem}
The proof of Lemma \ref{lem:outerdynamics} is straightforward from the definition of the local coordinate system \eqref{eq:localcoordinateshomchannel} and it is left to the reader. 
\medskip

\subsection{The global maps}\label{sec:globalmaps}
Now for $n\in\mathbb N$ large enough (depending only on $\delta_0$) we consider the collection of maps 
\begin{equation}\label{eq:globalmap}
    \Phi_{i\to j}^{(n)}=:\psi_{j}^{-1}\circ F^n\circ F^{(n_-+n_+)}\circ \psi_i
    \end{equation}
  defined between (a suitable subset of)  $\mathcal Q^{i}_{\delta}\subset X$ and a small neighbourhood of $\Gamma_-^j\subset \mathcal Q^j_\delta$ for $i,j=1,\dots,m$. It is not difficult to see (use Proposition \ref{prop:localdynamics}) that there exists $n_*\in\mathbb N$ such that for any $n\geq n_*$ and any pair $i,j=1,\dots,m$ 
  \[
  \mathrm{Dom}(\Phi_{i\to j}^{(n)})\neq \emptyset.
  \]
In the next proposition we gather some quantitative estimates which we will use in Sections \ref{sec:partiallygraphtransform}-\ref{sec:Returnmap} to fully characterize these domains and their images.
\begin{prop}\label{prop:globalmap}
   Fix any $K>0$, let $0<\delta\ll \delta_0 \ll 1$. Let $\bar\lambda,\underline\lambda,\bar\alpha$ and $\alpha$ be the positive constants defined in \eqref{eq:maxminlambda} and \eqref{eq:maxminalpha}. There exists $n_*\in\mathbb N$ (depending only on $\delta_0$) such that, 
 for any $i,j=1,\dots,m$:
\begin{enumerate}
\item \textit{Estimates for the map $\Phi^{(n)}_{i\to j}$:} at any $(q,p,z)\in[0,\delta]\times [0, \bar\lambda^nK]\times N$ at which $\Phi_{i\to j}^{(n)}$ is well defined, the functions $
  (\bar q,\bar p,\bar z)=\Phi_{i\to j}^{(n)}(q,p,z)$ satisfy the estimates
  \begin{equation}\label{eq:estimatesforwardmap}
  \begin{aligned}
  |\bar q|\lesssim& \bar \lambda^n &\qquad |\partial_q \bar q|\lesssim & \left(\bar \lambda \alpha \bar \alpha\right)^n & |\partial_p \bar q|\lesssim & \bar\alpha^n & |\partial_z \bar q|\lesssim& (\bar\lambda \alpha)^n\\
  &  &\qquad |\partial_q \bar p|\lesssim & \left(\alpha \bar \alpha\right)^n & |\partial_p \bar p|\gtrsim & 1/\bar\lambda^n & |\partial_z \bar p|\lesssim& (\alpha \bar \alpha)^n\\
  \end{aligned}
  \end{equation}
    and 
    \begin{equation}\label{eq:globalmapzcomponent}
    \bar z(q,p,z)= T^n(S_{i}(z)+\widetilde S_i(q,p,z))
    \end{equation}
    with $\widetilde S_i$ as in Lemma \ref{lem:outerdynamics}.
\item \textit{Estimates for  the inverse map $(\Phi^{(n)}_{i\to j})^{-1}$:} at any $(q,p,z)\in[0, \bar\lambda^n K]\times[0,\delta]\times N$ at which $(\Phi_{i\to j}^{(n)})^{-1}$ is well defined, the functions $(\bar q,\bar p,\bar z)=(\Phi^{(n)}_{i\to j})^{-1}(q,p,z)$ satisfy the estimates 
  \begin{equation}
  \begin{aligned}
  & &\qquad |\partial_q \bar q|\gtrsim & 1/\bar\lambda ^n & |\partial_p \bar q|\lesssim & (\alpha\bar\alpha)^n & |\partial_z \bar q|\lesssim& (\alpha\bar\alpha)^n\\
  |\bar p|\lesssim& \bar\lambda^n &\qquad |\partial_q \bar p|\lesssim &  \bar\alpha^n & |\partial_p \bar p|\gtrsim & \left(\bar\lambda\alpha \bar \alpha\right)^n & |\partial_z \bar p|\lesssim& (\bar\lambda \alpha)^n\\
  \end{aligned}
  \end{equation}
    and 
    \begin{equation}\label{eq:globalmapzcomponentinverse}
    \bar z(x,y,z)= (S_{i}+\widetilde S_i)^{-1} (q,p,T^{-n}(z))
    \end{equation}
\end{enumerate}
\end{prop}

The proof of this proposition consists of a number of straightforward (though somewhat tedious) estimates, and is deferred to Appendix \ref{sec:appendixreturnmap}. The important observation for what follows is that the $q$-direction gets contracted while the $p$-direction gets expanded under $\Phi^{(n)}_{i\to j}$ and viceversa under the inverse map.
\medskip

\subsection{Graph transform for  horizontal/vertical submanifolds}\label{sec:partiallygraphtransform} With a view towards the construction of the first return map induced by $F$ in the union 
\begin{equation}\label{eq:finalneighbourhood}
\mathcal Q_\delta=\bigcup_{i=1}^m\mathcal Q^i_{\delta}\subset X
\end{equation}
we now study the action of the collection of maps $\{\Phi^{(n)}_{i\to j}\}_n$ on a suitable class of  submanifolds.

\begin{defn}\label{defn:verticalsubmanifold}
    A  \textit{horizontal} (resp. \textit{vertical}) submanifold is an inmersed, codimension-one $C^1$ submanifold $\Delta\subset \mathcal Q^i_{\delta}$ for some $i=1,\dots,m$ which, in local coordinates $(q,p,z)$ as in  \eqref{eq:localcoordinateshomchannel} admits a parametrization of the form
     \begin{align*}
     \Delta=&\{(h(p,z),p, z)\colon z\in N, p\in[0,\delta] \}\\
     \big(\text{resp.}\quad \Delta=&\{(q,v(q,z),  z)\colon z\in  N,\ q\in [0,\delta]\} \big)
     \end{align*}
     where: the function $h$ (resp. $v$) satisfies $0<h<\delta$ (resp. $0<v<\delta$) and
         \[
        |\partial_{(p,z)} h|\leq 1\qquad\qquad (\text{resp.}  |\partial_{(q,z)} v|\leq 1).
         \]
\end{defn}

The definition of  horizontal and vertical submanifolds is justified as follows. The main idea is that, in view of item $(1)$ in Proposition \ref{prop:globalmap},  horizontal submanifolds (resp.  vertical) are well aligned with the most expanding (resp. contracting) directions of the maps $\Phi_{i\to j}^{(n)}$. In particular, the horizontal (resp. vertical) boundaries of $\mathcal Q^i_{\delta}$, $i=1,\dots,m$, are  horizontal (resp. vertical) submanifolds. Studying the action of $\Phi^{(n)}_{i\to j}$ on these classes of manifolds will allow us to show the existence of a Markov partition for the induced return map in the set  $\mathcal Q_\delta\subset X$ defined in \eqref{eq:finalneighbourhood}.

\begin{prop}\label{prop:graphtransform}
   Let  $0<\delta\ll \delta_0 \ll 1$. There exists $n_*\in\mathbb N$ (depending only on $\delta_0$) such that the following holds. Let $i=1,\dots,m$,  and let $\Delta\subset \mathcal Q^i_\delta$ be  a horizontal (resp. vertical) submanifold. Then, for any $n\geq n_*$ and any $j=1,\dots,m$ there exists a  horizontal submanifold $\Delta^{(n)}\subset \Phi^{(n)}_{i\to j}(\Delta)\subset \mathcal Q^j_\delta$ (resp. vertical submanifold $\Delta^{(n)}\subset (\Psi^{(n)}_{j\to i})^{-1}(\Delta)\subset \mathcal Q^j_\delta$). 
\end{prop}

\begin{proof}
 The proof boils down to a graph transform argument which makes use of the estimates in Proposition \ref{prop:globalmap}. For any $n\geq n_*$ as in Proposition \ref{prop:globalmap} and for a choice of $i,j=1,\dots,m$, let 
 \[
 (\bar q,\bar p,\bar z)=\Phi^{(n)}_{i\to j}(q,p,z)
 \]
 Now observe that, in view of the estimates \eqref{eq:estimatesforwardmap} and the definition of a horizontal manifold, for any  $(P,Z)\in [0,\delta]\times N$, the system of implicit equations 
 \[
 P=\bar p(h(p,z),p,z)\qquad\qquad Z=\bar z(h(p,z),p,z)
 \]
admits a unique solution $(p,z)=(\tilde p(P,Z),\tilde z(P,Z))$ defined for $P\in[0,\delta]$ and $Z\in N$. Moreover, making use of the estimates in Proposition \ref{prop:globalmap}, it is not difficult to see that
 \begin{equation}\label{eq:graphtransformproof}
 |\partial_P  \tilde p|\lesssim \bar\lambda ^n,\qquad\qquad |\partial_Z \tilde p|\lesssim (\bar\lambda\alpha\bar\alpha)^n.
 \end{equation}
 To show that 
 \[
 \Delta^{(n)}=\{(h^{(n)}(P,Z),P,Z)\colon  (P,Z)\in[0,\delta]\times N\}
 \]
 with
 \[
 h^{(n)}(P,Z)=\bar q(h(\tilde p(P,Z),\tilde z(P,Z)),\tilde p(P,Z), \tilde z(P,Z))
 \]
 is a horizontal submanifold it is enough to combine the estimates in Proposition \ref{prop:globalmap} and \eqref{eq:graphtransformproof} to deduce  (possibly after enlarging $n_*$) that
 \[
 |\partial_{(P,Z)} h^{(n)}|\leq (\alpha\bar\alpha\bar\lambda)^{n}.
 \]
 By our bunching plus pinching assumption \ref{it:H1}-\ref{it:H1+1/2}, the product $\vartheta=\alpha\bar\alpha\bar\lambda$ can be made smaller than one. The claim for vertical submanifolds is obtained by a similar argument.
\end{proof}

\medskip

\subsection{The induced  return map}\label{sec:Returnmap}

Finally, in this section, we use the results obtained above to describe in detail the first return map (wherever it is defined) 
\begin{equation}\label{eq:returnmap}
\mathcal R:\mathcal Q_\delta\to \mathcal Q_\delta
\end{equation}
that $F:X\to X$ induces in the set $\mathcal Q_\delta\subset X$ defined in \eqref{eq:finalneighbourhood}. We start by observing that $\mathcal R$ admits a Markov partition. To do so it is convenient for first introduce the following concept.

\begin{defn}
 Let $i=1,\dots,m$. We say that a set $\mathcal V\subset \mathcal Q_\delta^i$ (resp. $\mathcal H\subset \mathcal Q_\delta^i$)  is a vertical block (resp. horizontal block) if it can be foliated by vertical (resp. horizontal) submanifolds.
\end{defn}

We now describe the Markov partition for the return map $\mathcal R$. The fact that $\mathcal Q_\delta$ is not connected introduces some small, but rather irrelevant, technicalities. We have included a sketch of the partition in Figure \ref{fig:fig5} for the convenience of the reader.

\begin{figure}
    \centering
    \includegraphics[scale=0.65]{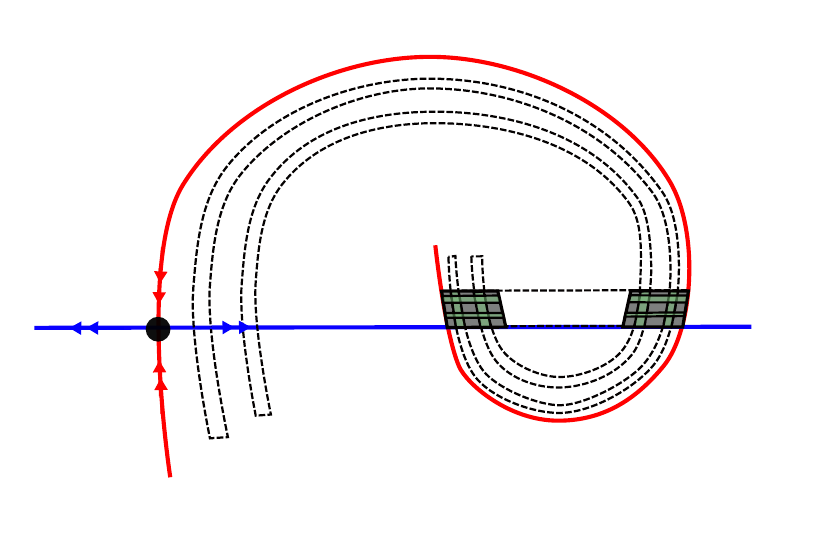}
    \caption{Sketch of the Markov partition elements for the case where $m=2$, i.e. we only involve two homoclinic channels. To build the partition it is helpful to analyze the preimages (and images) of the connected dashed rectangle and then look at where the left and right shaded subrectangles land.}
    \label{fig:fig5}
\end{figure}

\begin{prop}\label{prop:markovproperty}
Let $A=\{1,\dots,m\}$ and let $0<\delta\ll \delta_0 \ll 1$. There exists $n_*\in\mathbb N$ and $C>0$  and collections $\{\mathcal H^{(n)}_{i;j}\},\{\mathcal V^{(n)}_{i;j}\}\subset \mathcal Q_\delta$ of horizontal and vertical blocks such that
\[
\bigcup_{n,i,j}\mathcal V^{(n)}_{i;j}\subset \mathrm{Dom} (\mathcal R)\qquad\qquad \bigcup_{n,i,j}\mathcal H^{(n)}_{i;j}\subset \mathrm{Dom}(\mathcal R^{-1}),
\]
for each $(n,i,j)\in \mathbb N\times A^2$ with $n\geq n_*$
\[
 \mathcal V_{i;j}^{(n)}\subset \mathcal Q_\delta^{i}\qquad\qquad \mathcal H_{i;j}^{(n)}\subset \mathcal Q_\delta^{j}
\]
and for any $(n,i,j),(n',i',j')\in\mathbb N\times A^2$ with $n,n'\geq n_*$:
\begin{enumerate}
\item \textit{(Domain Partition):}if $(n,i,j)\neq (n',i',j')$ 
\[
\mathcal V^{(n)}_{i;j}\cap\mathcal V^{(n')}_{i';j'}=\emptyset\qquad\text{and}\qquad\mathcal H^{(n)}_{i;j}\cap \mathcal H^{(n')}_{i';j'}=\emptyset,
\]
   
\item \textit{(Markov property):} We have
\[
 \mathcal R(\mathcal V_{i;j}^{(n)})
=\Phi_{i\to j}^{(n)}
(\mathcal V_{i;j}^{(n)})
=\mathcal H_{i;j}^{(n)}.
\]
Moreover,  the set
\[
\mathcal H_{j;i}^{(n)}
\cap\mathcal V_{i;j'}^{(n')}
\]
is mapped by
\[
\mathcal R=\Phi_{i\to j'}^{(n')}
\]
to a horizontal block and by
\[
\mathcal R^{-1}
=(\Phi_{j\to i}^{(n)})^{-1}
\]
to a vertical block.

     \item \textit{(Localization):} In local coordinates $(q,p,z)$ as in \eqref{eq:localcoordinateshomchannel} 
     \begin{equation}\label{eq:localizationblocks}\mathcal V_{i;j}^{(n)}\subset \{0<p\leq C\bar\lambda^n\}\qquad\qquad \mathcal H_{i;j}^{(n)}\subset\{0<q<C\bar\lambda^n\}
     \end{equation}
   
\item \textit{(Center dynamics):}
On  $\mathcal H_{j;i}^{(n)}
\cap \mathcal V_{i;j'}^{(n')}$
  we have
\[
\pi_z\mathcal R(q,p,z)
=T^{n'}\big(S_i(z)+\widetilde S_i(q,p,z)\big).
\]
Moreover, if
\[
(\hat q,\hat p,\hat z)
=\mathcal R^{-1}(q,p,z),
\]
then
\[
T^{-n}(z)
=S_j(\hat z)+
\widetilde S_j(\hat q,\hat p,\hat z).
\]

\end{enumerate}
with $T:N\to N$ as in \eqref{eq:innerdynamicsdefn}, $\{S_i\}_i:N\to N$ as in \eqref{eq:scattmapdynamicsdefn} and $\{\widetilde S_i\}_i:N\to N$ as in Lemma \ref{lem:outerdynamics}.
\end{prop}

\begin{proof}
    Proposition \ref{prop:graphtransform} implies that or each $n\geq n_*$ 
\[
 \mathcal H^{(n)}_{i;j}=\Phi^{(n)}_{i\to j}(\mathcal V_{i;j}^{(n)}) \qquad\qquad \mathcal V^{(n)}_{i;j}\subset \mathcal Q^i_\delta \cap (\Phi^{(n)}_{i\to j})^{-1}(\mathcal Q_\delta^j).
\]
Moreover, by choosing $\delta>0$ small enough it is not difficult to see that if we let 
\[
\tau_\pm(w)
=\min\{r\ge1:F^{\pm r}(w)\in\mathcal Q_\delta\}.
\]
then
\[
 \mathcal V^{(n)}_{i;j}\subset \{w\in\mathcal Q_\delta^{i}\colon \tau_+(\omega)=n+n_++n_-, F^{n+n_-+n_+}(w)\in \mathcal Q^j_\delta\}
\]
and
\[
 \mathcal H^{(n)}_{i;j}\subset \{w\in\mathcal Q_\delta^{j}\colon \tau_-(\omega)=n+n_++n_-, F^{-n-n_--n_+}(w)\in \mathcal Q^i_\delta\}
\]
In particular, this implies item $(1)$. Items $(2), (3)$ and $(4)$ follow from the very definition of the sets $\{\mathcal H_{i;j}^{(n)},\mathcal V_{i;j}^{(n)}\}$ and the estimates in Proposition \ref{prop:globalmap}.
\end{proof}
\medskip

\subsection{The normally hyperbolic lamination}

Finally, it follows from the estimates in Propositions \ref{prop:globalmap} and Proposition \ref{prop:markovproperty} that the set 
\begin{equation}\label{eq:lamination}
\mathcal X=\bigcap_{n\in\mathbb Z} \mathcal R^n(\mathcal Q_\delta)
\end{equation}
is a  normally hyperbolic  lamination (see \cite{MR501173}). Moreover, from the third item in Proposition \ref{prop:markovproperty} we can connect the scattering map dynamics along the homoclinic channels $\{\Gamma^i\}_i$ to the dynamics on the lamination $\mathcal X$. More precisely, we have the following result, which is an expanded version of Proposition \ref{prop:mainnhim}.

\begin{prop}\label{prop:nhimexpanded}
 Let $F_\varepsilon:X\to X$ be a parametric family of $C^r$ maps ($r\geq 2$) of a Riemannian manifold $X$  which satisfy assumptions \ref{it:H1}-\ref{it:H2}. Let $A=\{1,\dots,m\}$ and let $\Sigma=(\mathbb N\times A)^\mathbb Z$. Then, there exists 
    \[
    C(F_0)>0, \qquad\vartheta(F_0)<1\quad\text{and}\quad \varepsilon_0(F_0)>0
    \]
    such that for any $0\leq \varepsilon\leq \varepsilon_0$ there exists a homemorphism $\Psi:\Sigma \times N\to \mathcal X$ such that for any 
    $\bs\omega\in \Sigma $ the leaf $N_{\bs\omega}=\Psi(\{\bs\omega\}\times N)\subset X$ is a $C^1$ submanifold. Moreover $\Psi$ conjugates $\mathcal R|_{\mathcal X}$ to a skew-product of the form
    \begin{equation}\label{eq:skewprodpropnhimproof}
    \begin{aligned}
    \mathcal F: (\mathbb N\times A)^\mathbb Z\times N&\to (\mathbb N\times A)^\mathbb Z\times N\\
    (\bs \omega,z)&\mapsto (\sigma(\bs\omega), F_{\bs\omega}(z))
    \end{aligned}
    \end{equation}
    where $\sigma$ is the right shift and, at each $\bs\omega=(\omega,\iota)\in (\mathbb N\times A)^\mathbb Z$ and $z\in \mathbb N$
    \begin{equation}\label{eq:expressioncenterdynskewprod}
   F_{\bs\omega}(z)=T^{\omega_0}\circ(S_{\iota_0}+R_{\bs\omega})(z),\qquad\qquad |R_{\bs\omega}|_{C^1}\leq C \vartheta^{ \min\{\omega_0,\omega_{-1}\}}.
    \end{equation}
    Moreover, for any $\bs\omega,\bs\omega'\in(\mathbb N\times A)^\mathbb Z$ and $n\geq 0$ with: 
    \begin{itemize}
        \item  $\bs\omega_k=\bs\omega_k'$ for $k<n$, $\iota_n=\iota_n'$ and $\omega_n\leq \omega_n'$  we have 
\begin{equation}\label{eq:comparisonskewprod}
    |R_{\bs\omega}-R_{\bs\omega'}|\leq C  \vartheta^{\omega_0+\dots+\omega_{n}}.
    \end{equation}
    \item $\bs\omega_k=\bs\omega_k'$ for $k\geq -n$ and $\omega_{-n-1}\leq \omega'_{-n-1}$ we have
\begin{equation}\label{eq:comparisonskewprod2}
    |R_{\bs\omega}-R_{\bs\omega'}|\leq C  \vartheta^{\omega_{-1}+\dots+\omega_{-n-1}}.
    \end{equation}
\end{itemize}
\end{prop}
\begin{rem}
    Inequalities \eqref{eq:comparisonskewprod}-\eqref{eq:comparisonskewprod2} will be important in Section \ref{sec:skewprod}. They allow us to compare the dynamics of points in the same local unstable/stable manifold.
\end{rem}

\begin{rem}
 If we further assume that both  $r$ and  the exponent $\kappa$ in \ref{it:H1}-\ref{it:H1+1/2} are large enough but universal (i.e. independent of the family $F_\varepsilon$) then the leaves $\{N_{\bs\omega}\}_{\bs\omega}$ can be assumed to be $C^2$ and the estimates \eqref{eq:expressioncenterdynskewprod}, \eqref{eq:comparisonskewprod} and \eqref{eq:comparisonskewprod2} hold also in $C^2$ norm (with a possibly larger $\theta\in (0,1)$).
\end{rem}

\begin{proof}
 Throughout the proof, give two functions $f(\bs\omega,z)$ and $h(\bs\omega,z)$ we write $f\lesssim h$ if there exists $C>0$ such that uniformly on $(\bs\omega,z)$ we have $f(\bs\omega,z)\leq Ch(\bs\omega,z)$. 
 \medskip
 
For any  $\bs\omega=(\omega,\iota)\in (\mathbb N\times A)^\mathbb Z$ introduce the sets
  \[
   \mathcal W_{\bs\omega}^u=\bigcap_{k\geq 0} \mathcal R^k(\mathcal H_{\iota_{-k-1};\iota_{-k}}^{(\omega_{-k-1})}),\qquad\qquad (\text{resp. }\mathcal W_{\bs\omega}^s=\bigcap_{k\geq 0}\mathcal R^{-k}(\mathcal V_{\iota_{k};\iota_{k+1}}^{(\omega_{k} )} ))
  \]
  By construction, the sets $\mathcal W^u_{\bs\omega}$ (resp. $,\mathcal W^s_{\bs\omega}$) are horizontal (resp. vertical) submanifolds as in Definition \ref{defn:verticalsubmanifold} and, in local coordinates $(q,p,z)$ as in \eqref{eq:localcoordinateshomchannel},  can be parametrized as 
  \[
  \mathcal W^u_{\bs\omega}=\{h_{\bs\omega}(p,z),p,z)\colon z\in N,p\in[0,\delta]\}\qquad\qquad (\text{resp. }\mathcal W^s_{\bs\omega}=\{(q,v_{\bs\omega}(q,z),z)\colon z\in N,q\in[0,\delta]\})
  \]
  in terms of $C^1$ functions $h_{\bs\omega}$ (resp. $v_{\bs\omega}$) satisfying
  \[
   |h_{\bs\omega}|_{C^1}\lesssim \vartheta^{\omega_{-1}},\qquad\qquad
  (\text{resp. } |v_{\bs\omega}|_{C^1}\lesssim \vartheta^{ \omega_{0}}).
  \]
 Hence,
   \[
   N_{\bs \omega}= \mathcal W^u_{\bs\omega}\cap \mathcal W^s_{\bs\omega}
   \]
 is a $C^1$ submanifold diffeomorphic to $N$ and, in local coordinates $(q,p,z)$ as in \eqref{eq:localcoordinateshomchannel}, can be parametrized as 
      \begin{equation}\label{eq:parametrizationNomega}
      N_{\bs\omega}=\{(q_{\bs\omega}(z),p_{\bs\omega}(z),z)\colon z\in N\}
      \end{equation}
      for some $C^1$ functions $q_{\bs\omega},p_{\bs\omega}$ which solve the system of equations
      \[
      q_{\bs \omega}(z)=h_{\bs\omega}(p_{\bs\omega}(z),z)\quad\qquad p_{\bs \omega}(z)=v_{\bs\omega}(q_{\bs\omega}(z),z)
      \]
      and satisfy
      \begin{equation}\label{eq:localizationNomega}
      |q_{\bs\omega}|_{C^1}\lesssim \vartheta^{\omega_{-1}}\qquad\qquad |p_{\bs\omega}|_{C^1}\lesssim \vartheta^{ \omega_{0}}.
      \end{equation}
      Moreover, 
      \[
      \mathcal R^k (N_{\bs\omega})\subset \mathcal H^{(\omega_{k-1})}_{\iota_{k-1};\iota_{k}}\cap \mathcal V^{(\omega_k)}_{\iota_{k};\iota_{k+1}}
     \]
 In particular, 
\[
\mathcal R(N_{\bs\omega})=N_{\sigma\bs\omega},\qquad\qquad (\sigma\bs\omega)_k=\bs\omega_{k+1}.
\]
 We deduce that 
\[
\mathcal X=\bigcup_{\bs\omega\in \Sigma} N_{\bs\omega}
\]
and the dynamics on $\mathcal X$ is driven by the skew-product map \eqref{eq:skewprodpropnhimproof}. To show that $\mathcal X$ is normally hyperbolic it is enough to observe that at any $w\in \mathcal Q_\delta$, 
\[
\mathcal C^{uu}_w=\{v\in T_w \mathcal Q_\delta\colon |v_p|\geq \max\{|v_q|,|v_z|\}\}\qquad  (\text{resp. } \mathcal C^{ss}_w=\{v\in T_w \mathcal Q_\delta\colon |v_q|\geq \max\{|v_p|,|v_z|\}\}) 
\]
constitute a family of invariant expanding (resp. contracting) cones and that, for $w\in N_{\bs\omega}$ and $v\in \mathcal C^{uu}_w$ (resp. $v\in \mathcal C^{ss}_w$)
\[
\lVert D\mathcal R(w) v\rVert\gtrsim \vartheta^{-\omega_0}\lVert v\rVert \qquad\qquad(\text{resp.} \lVert D\mathcal R^{-1}(w) v\rVert\gtrsim \vartheta^{-\omega_{-1}}\lVert v\rVert).
\]
It now remains to establish the estimates for 
\[
R_{\bs\omega}=\widetilde S_{ \iota_0}|_{N_{\bs\omega}}
\]
given in \eqref{eq:expressioncenterdynskewprod} and \eqref{eq:comparisonskewprod}-\eqref{eq:comparisonskewprod2}. Inequality \eqref{eq:expressioncenterdynskewprod} follows from \eqref{eq:propertiesfunctionsoutermap} and  \eqref{eq:localizationNomega}. To deduce inequality \eqref{eq:comparisonskewprod} we observe that for such $\bs\omega,\bs\omega'$ we have (recall the expression of the parametrization for $N_{\bs\omega}$ in \eqref{eq:parametrizationNomega})
\[
|q_{\bs\omega}-q_{\bs\omega'}|\lesssim \vartheta^{\mathrm{min}\{\omega_{n},\omega'_{n}\}}\prod_{k=-1}^{n-1}\vartheta^{\omega_k}\qquad\qquad |p_{\bs\omega}-p_{\bs\omega'}|\lesssim \vartheta^{\mathrm{min}\{\omega_{n},\omega'_{n}\}}\prod_{k=0}^{n-1}\vartheta^{\omega_k}.
\]
Inequality \eqref{eq:comparisonskewprod2} is obtained by a similar reasoning.
\end{proof}

\section{Symbolic blenders for IFS on high dimensional annuli}\label{sec:blenderIFS}
In this section we give a proof of Theorem \ref{thm:mainIFS}. The proof is divided in a number of steps:
\begin{enumerate}
    \item First, in Section \ref{sec:linearapproxIFS} we replace the maps $\{T^n_\varepsilon\circ S_\varepsilon\}_n$ by their local affine approximation and show that the corresponding matrix is hyperbolic.

    \item In Section \ref{sec:uniformIFS} we show that it is possible to construct a range of $\mathcal N_{\varepsilon,\kappa,\chi}\subset\mathbb N$ for which, at the same time, the sequence $\{[n\beta]\}_{n\in\mathcal N_{\varepsilon,\kappa,\chi}}$ is sufficiently well-distributed, while, the linearization of corresponding maps $\{T^n_\varepsilon\circ S_\varepsilon\}_{n\in\mathcal N_{\varepsilon,\kappa,\chi}}$, can be  simultaneously diagonalized (up to sufficiently small errors).

    \item In Section \ref{sec:coveringandequidistributionIFs} we verify  that the set of maps $\{T^n_\varepsilon\circ S_\varepsilon\}_{n\in\mathcal N_{\varepsilon,\kappa,\chi}}$ satisfy the so-called covering property and has a set of well distributed hyperbolic fixed points (see \cite{NassiriPujalsTransitivity} or \cite{BeyondUH}).
    \item In Section \ref{sec:csblenderIFS} we verify the existence of a symbolic cs-blender by exploiting the covering and equidistribution properties established in Section \ref{sec:coveringandequidistributionIFs}.
    \item In Section \ref{sec:doubleblenderIFS}, a symbolic double blender is constructed making use of symmetry considerations.
    \item Finally, in Section \ref{sec:completionproofmainifs}  we conclude the proof of Theorem \ref{thm:mainIFS}.
\end{enumerate}
\medskip

With the intention of alleviating the notation as much as possible, throughout this section we will simply write $T,S$ instead of $T_\varepsilon,S_\varepsilon$. We will also assume without mentioning that the maps $T,S$ satisfy assumptions \ref{it:itemB1}-\ref{it:itemB5}.

\subsection{Linear approximation}\label{sec:linearapproxIFS}
\begin{lem}\label{lem:C1control}
    Fix any $C>0$. Then, there exists $\ell>0$ independent of $\varepsilon$ such that for any $n\leq C/\varepsilon$ the map 
    \begin{equation}\label{eq:form Fn}
    F_n:=T^n\circ S: \mathbb A^d\to \mathbb A^d
    \end{equation}
    can be expressed as
    \begin{equation}\label{eq:localaffineapprox}
    F_n:\binom{\varphi}{J}=\underbrace{\binom{[n\beta]}{0}}_{\bs b_n}+\underbrace{\begin{pmatrix}
    \mathrm{id}+n\varepsilon AB&nA\\
    \varepsilon B&\mathrm{id}
\end{pmatrix}}_{A^{(n)}}\binom{\varphi}{J}+\mathcal E(\varphi,J)
    \end{equation}
   with $\mathcal E=(\mathcal E_\varphi,\mathcal E_J)$ such that at any $(\varphi,J)\in \{|\varphi|\leq 1/2,\ |J|\leq \ell \varepsilon\}$
    \begin{equation}\label{eq:errorestimatesIFSapprox}
    \mathcal E_\varphi=O(n\varepsilon |\varphi|^2,\varepsilon,\varepsilon|\varphi|)\qquad\qquad \mathcal E_J=O(\varepsilon|\varphi|^2,\varepsilon^2).
    \end{equation}
    Moreover, for these values of $(\varphi,J)$
    \begin{equation}\label{eq:errorestimatesIFSapproxC1}
    D\mathcal E(\varphi,J)=\begin{pmatrix}O(n\varepsilon|\varphi|,\varepsilon)&O(1)\\ O(\varepsilon|\varphi|,\varepsilon^2)&O(\varepsilon) \end{pmatrix}.
    \end{equation}
\end{lem}

The proof of this result is obtained in the very same way as Lemma 5.2. in \cite{guardia2025partiallyhyperbolicdynamics3body} and we do not reproduce it here. We now study the dynamical features of the local affine approximation \eqref{eq:localaffineapprox} for the family of maps $\{F_n\}_n$ for a suitable range of $n$. To that end we recall that, in view of assumption \ref{it:itemB3}, the matrix $AB$ has $d$ different eigenvalues $\{\alpha_i\}_{i=1}^d\subset\mathbb R\setminus \{0\}$.
We now exploit this fact to diagonalize the matrix $A^{(n)}$ in \eqref{eq:localaffineapprox}. To do so we let $Q\in GL(\mathbb R,d)$ be such that 
\begin{equation}\label{eq:diagonalizationAB}
Q^{-1}AB  Q=\mathrm{diag}(\alpha_1,\dots,\alpha_d).
\end{equation}

\begin{lem}\label{lem:hyperbolicityiteratesIFS}
    There exists $N_0\gg1$ (depending only on $AB$) such that, for any $n\in\mathbb N$ satisfying 
    \[
    n\varepsilon \geq N_0
    \]
    the symplectic matrix $A^{(n)}$ is hyperbolic with simple spectrum.   Its eigenvalues $\{(\lambda_i^{(n)},1/\lambda_i^{(n)})\}_{i=1}^d\subset\mathbb R\setminus\{-1,1\}$ are given  by the branch of 
    \[
\lambda_i^{(n)}=\frac{\alpha_i^{(n)}+2\pm\sqrt{(\alpha_i^{(n)})^2+4\alpha_i^{(n)}} }{2}\qquad\qquad \alpha_i^{(n)}=n\varepsilon\alpha_i 
    \]
    for which $|\lambda_i^{(n)}|<1$
    and 
    \begin{equation}\label{eq:defndiagonalmatrix}
   Q_n^{-1}A^{(n)}Q_n :=\mathcal D_n=\begin{pmatrix}
        D_n&0\\
        0&D_n^{-1}
    \end{pmatrix}=\begin{pmatrix}
        \mathrm{diag}(\lambda_1^{(n)},\dots,\lambda_d^{(n)})&0\\
        0&\ \mathrm{diag}(1/\lambda_1^{(n)} ,\dots, 1/\lambda_d^{(n)} )
    \end{pmatrix}
    \end{equation}
    with 
    \[
    Q_n=\begin{pmatrix}
        Q&Q\\
      \varepsilon BQ(D_n-\mathrm{id})^{-1}&  \varepsilon BQ(D^{-1}_n-\mathrm{id})^{-1}
        
    \end{pmatrix}.
    \]
\end{lem}

\begin{proof}
Define the matrix 
\[
\bar A^{(n)}= \underbrace{\begin{pmatrix}
    \mathrm{id}&0\\
    0 & \varepsilon^{-1} B^{-1}
\end{pmatrix}}_{C^{-1}} A^{(n)} \underbrace{\begin{pmatrix}
    \mathrm{id}&0\\
    0 & \varepsilon B
\end{pmatrix}}_C=\begin{pmatrix}
    \mathrm{id}+n\varepsilon AB &n\varepsilon AB\\\mathrm{id}&\mathrm{id}
\end{pmatrix}
\]
By assumption, the matrix $M_n=n\varepsilon AB$ has $d$ different eigenvalues $\{\alpha_i^{(n)}\}_{i=1}^d\subset\mathbb R\setminus\{0\}$ where $\alpha_i^{(n)}=n\varepsilon\alpha_i$. We now relate the eigenvalues of $\bar A^{(n)}$ to those of $M_n$. Let $\{v_i\}_{i=1}^d\in\mathbb R^d$ be  the eigenvectors of $M_n$. Then, it is a simple exercise to check that the pair $(\lambda_i^{(n)},v_i^{(n)})\in \mathbb R\times\mathbb R^{2d}$ with 
\begin{equation}\label{eq:eigenvaluesequation}
(\lambda_i^{(n)})^2-(2+\alpha_i^{(n)})\lambda_i^{(n)}+1=0
\end{equation}
and
\[
v_i^{(n)}=(v_1,\dots,v_d, v_1/(\lambda_i^{(n)}-1),\dots, v_d/(\lambda_i^{(n)}-1))
\]
constitutes and eigenpair for the matrix $M_n$. Provided 
\begin{equation}\label{eq:suffcondhyperbolicity}
\alpha_i^{(n)}\notin[-4,0]
\end{equation}
the equation \eqref{eq:eigenvaluesequation} posseses two real solutions (one being the multiplicative inverse of the other one). Clearly, since $\alpha_i^{(n)}=n\varepsilon \alpha_i$ and $\alpha_i\neq 0$, provided $n\varepsilon$ is large enough we can guarantee that \eqref{eq:suffcondhyperbolicity} holds for all $i=1,\dots,d$. The conclusion follows after a straightforward algebraic manipulation.
\end{proof}

\subsection{Diophantine approximation and simultaneous diagonalization}\label{sec:uniformIFS}

Our next result, key for the construction of a symbolic cs-blender, shows that for short enough ranges $n\in \{N_0/\varepsilon,\dots,N_0/\varepsilon+ N_*\}$ with $N_0$ as in Lemma \ref{lem:hyperbolicityiteratesIFS}  the matrices $A^{(n)}$ can be simultaneously diagonalized up to a small error. 

\begin{lem}\label{lem:simultaneouslinearization}
   Let $N_0$ be as in Lemma \ref{lem:hyperbolicityiteratesIFS}, let $N_*\in\mathbb N$ and define 
   \begin{equation}\label{eq:delta}
   N_\varepsilon=N_0/\varepsilon\qquad \qquad \delta(\varepsilon,N_*)=N_*/N_\varepsilon= \varepsilon N_*/N_0
   \end{equation}
   Then, for any $n\in \{N_\varepsilon,\dots, N_*\}$ we have 
    \[
   Q_{N_\varepsilon}^{-1} A^{(n)} Q_{N_\varepsilon}=\mathcal D_n+O(\delta)
    \]
    with $\mathcal D_n$ being the diagonal matrix \eqref{eq:defndiagonalmatrix}.
\end{lem}

\begin{proof}
Let 
 \[
    V_N=\varepsilon BQ(D_N-\mathrm{id})^{-1}\qquad W_N=\varepsilon BQ(D^{-1}_N-\mathrm{id})^{-1}.
    \]
Then,
\[
Q_n=\begin{pmatrix}
    Q&Q\\
    V_n&W_n
\end{pmatrix}
\qquad\qquad Q_N^{-1}=\begin{pmatrix}
    Q^{-1}+(W_N-V_N)^{-1}V_NQ^{-1}& -(W_N-V_N)^{-1}\\
     -(W_N-V_N)^{-1}V_NQ^{-1}& (W_N-V_N)^{-1}
\end{pmatrix}
\]
Hence, 
\[
A^{(n)}Q_N= \begin{pmatrix}
    (\mathrm{id}+n\varepsilon AB)Q+nA V_N& (\mathrm{id}+n\varepsilon AB)Q+nA W_N\\
    \varepsilon BQ+V_N& \varepsilon BQ+W_N
\end{pmatrix}
\]
and (here the term $(W_N-V_N)^{-1}$ is understood to multiply each block of the matrix)
\[
\mathcal A^{(n)}=D_n+(n-N)(W_N-V_N)^{-1}\begin{pmatrix}
         W_N (\varepsilon D_n+Q^{-1}AV_N)& V_N(\varepsilon D_n+Q^{-1} A W_N)\\
        -V_N(\varepsilon D_n+Q^{-1}AV_N)& -W_n(\varepsilon D_n+Q^{-1}AW_N)
    \end{pmatrix}
\]
Finally, notice that
\[
(V_N)_{i,j}, (W_N)_{i,j}=O(\varepsilon)
\]
and
\[
\left((W_N-V_N)^{-1}\right)_{i,j}= \varepsilon^{-1}  \left(\tilde D_n(BQ)^{-1}\right)_{i,j}=O(\varepsilon^{-1})
\]
with $\tilde D_n=\mathrm{diag}(\dots,(\lambda_i^{(n)}-1)(1/\lambda_i^{(n)}-1)/(\lambda_i^{(n)}-1/\lambda_i^{(n)}),\dots)$ so the conclusion follows.
\end{proof}

The key idea now is that, provided the Diophantine constant $\gamma$ of the rotation number $\beta\in\mathbb R^d$ is large enough compared to $\varepsilon$ it is possible to choose $N_*\in\mathbb N$ such that, at the same time:
\begin{enumerate}
    \item the orbit segment $\{[n\beta]\}_{n=N_\varepsilon}^{N_\varepsilon+N_*}$ is well distributed over $\mathbb T^d$,
    \item the quantifier $\delta$ in \eqref{eq:delta} is small enough, so all the maps $F_n$ with $n\in \{N_\varepsilon,\dots,N_*\}$ can be simultaneously diagonalized up to a small error.
\end{enumerate}
Indeed, we have the following result. 
\begin{thm}[Theorem VI in Chapter 5 of \cite{Cassels72}]\label{thm:Diophantineapprox}
  Let $\gamma>0$ be the Diophantine constant of the rotation vector $\beta\in\mathbb R^d$. Then, for any $N\in\mathbb N$ large enough
  \[
  \max_{x\in\mathbb T^d}\mathrm{dist}([n\beta]_{n\leq N}, x)\leq \frac{1}{\gamma N^{1/d}}.
  \]
\end{thm}
\medskip

\subsubsection*{Blender geometry and quantifiers}\label{sec:geometryandquantifiersIFS}

Based on the results  in Lemma \ref{lem:simultaneouslinearization} and Theorem \ref{thm:Diophantineapprox}, we now describe the region $Q^{cs}\subset\mathbb A^d$ in which we will establish the existence of a symbolic cs-blender. To do so, it will be necessary to introduce auxilary constants $\kappa,\chi>0$ which, as will become clear soon, determine the geometry of the set $Q^{cs}$. We then define the linear maps
\[
(\text{Diagonalization})\qquad\qquad \psi_\varepsilon:\binom{\tilde \xi}{\tilde \eta}\mapsto \binom{\varphi}{J}=Q_{N_\varepsilon} \binom{\tilde\xi}{\tilde\eta}
\]
and
\begin{equation}\label{eq:scalingIFS}
(\text{Scaling+Anisotropy})\qquad\qquad \psi_{\kappa,\chi}:\binom{\xi}{ \eta}\mapsto \binom{\tilde\xi}{\tilde\eta}=\underbrace{\begin{pmatrix}
    \kappa \ \mathrm{id}&0\\
    0& (\kappa/\chi)\ \mathrm{id}
\end{pmatrix}}_{S_{\kappa,\chi}}.
\end{equation}
Denote their composition by \begin{equation}\label{eq:coordinatesforIFS}
    \Psi_{\varepsilon,\kappa,\chi}=\psi_\varepsilon\circ\psi_{\kappa,\chi}:[-1,1]^{2d}\to \mathbb A^d
\end{equation}
and define the rectangle 
\begin{equation}\label{eq:csblenderdomain}
    Q^{cs}=\Psi_{\varepsilon,\kappa,\chi}([-1,1]^{2d}).
\end{equation}
Before continuing, let us briefly comment on the nature of this transformation. On one hand, the introduction of a scaling (measured by $\kappa$) will allow us to guarantee that $Q^{cs}$ is localized close to $\{J=0=\varphi\}$ so we can make use of the linear approximation in Lemma \ref{lem:C1control}. On the other hand, the transformation $\psi_\varepsilon$ diagonalizes the linear dynamics. Finally, the role of the quantifier $\chi$ (which will also be taken small) is to introduce a strong anisotropy between the sizes of the different faces of $Q^{cs}$. By doing so we can guarantee that, in the new coordinate system, the vector of affine translations, given by $\Psi_{\varepsilon,\kappa,\chi}^{-1}(\bs b_n)$, is well aligned with the contracting direction (see Figure \ref{fig:fig6}). This last property will be of importance to show that the the maps $\{F_n\}_n$ satisfy the covering property with respect to $Q^{cs}$.

\begin{figure}
    \centering
    \includegraphics[scale=0.65]{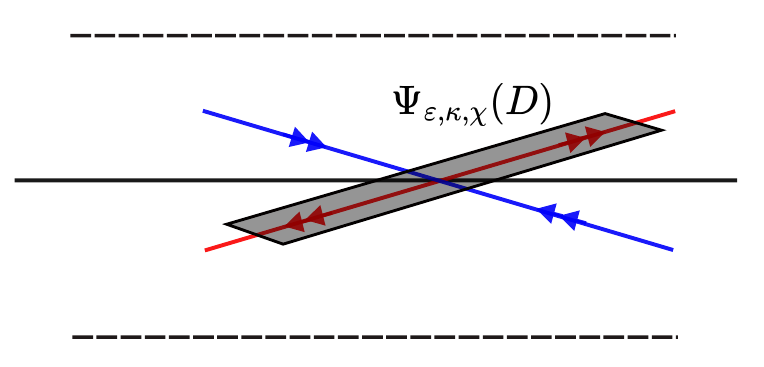}
    \caption{Image of $D=[-1,1]^{2d}$ under the coordinate change $\Psi_{\varepsilon,\kappa,\chi}$. In red and blue we also indicate the expanding and contracting directions.}
    \label{fig:fig6}
\end{figure}

\begin{prop}\label{prop:uniformcoordinatesIFS}
    Fix any $0<\chi\ll 1$ sufficiently small (depending only on the matrices $A,B$). Then, there exist $\kappa_0(\chi)>0$ and  $\varepsilon_0(\gamma,\kappa,\chi)>0$ such that for any $0<\kappa\leq \kappa_0(\chi)$ and any $0<\varepsilon\leq \varepsilon_0(\gamma,\kappa,\chi)$ there exist a constant $C>0$ (independent of $\chi,\kappa$ and $\varepsilon$) and a subset $\mathcal N_{\varepsilon,\kappa,\chi}\subset\mathbb N$ such that, for any $n\in\mathcal N_{\varepsilon,\kappa,\chi}$ the map
    \[
    \mathcal F_n:=\Psi_{\varepsilon,\kappa,\chi}^{-1}\circ F_n\circ\Psi_{\varepsilon,\kappa,\chi}:[-1,1]^{2d}\to \mathbb R^{2d}
    \]
    is of the form
    \begin{equation}\label{eq:approxdynamicsIFS}
    \mathcal F_n:\binom{\xi}{\eta}\mapsto \tilde{\bs b}_n+\mathcal D_n\binom{\xi}{\eta}+\widetilde{\mathcal{E}}(\xi,\eta),
    \end{equation}
    with $\tilde{\bs b}_n=(\tilde{\bs b}_{n,\xi},\tilde{\bs b}_{n,\xi})$ such that
    \[
   \max_{\xi\in [-1,1]^{d}} \min_{n\in\mathcal N_{\varepsilon,\kappa,\chi}}\ \mathrm{dist}(\tilde {\bs b}_{n,\xi},\xi)\leq C\chi\qquad\qquad 
        |\tilde{\bs b}_{n,\eta}|\leq C\chi |\tilde{\bs b}_{n,\xi}|
        \]
    and $\widetilde{\mathcal E}$ verifying 
    \[
   \widetilde {\mathcal E}=\binom{O(\kappa^2)}{O(\varepsilon \kappa^2)} \qquad\qquad D\mathcal{\widetilde E}=\begin{pmatrix}
    O(\kappa)&O(\kappa/\chi)\\
    O(\kappa\chi)&O(\kappa)
\end{pmatrix}.
    \]
\end{prop}
\begin{proof}
    From its definition 
\begin{align*}
\mathcal F_n(\xi,\eta)=&\Psi_{\varepsilon,\kappa,\chi}^{-1}\circ F_n\circ\Psi_{\varepsilon,\kappa,\chi}(\xi,\eta)\\
=&S_{\kappa,\chi}^{-1}Q_{N_\varepsilon}^{-1} \bs b_n+ S_{\kappa,\chi}^{-1}Q_{N_\varepsilon}^{-1} A^{(n)}Q_{N_\varepsilon}S_{\kappa,\chi} \binom{\xi}{\eta}+\Psi^{-1}_{\varepsilon,\kappa,\chi} (\mathcal E\circ \Psi_{\varepsilon,\kappa,\chi}(\xi,\eta))
\end{align*}
where $Q_{N_\varepsilon}$ is the matrix in Lemma \ref{lem:hyperbolicityiteratesIFS} (particularized for $n=N_\varepsilon=N_0/\varepsilon$) and $S_{\kappa,\chi}$ is the scaling matrix in \eqref{eq:scalingIFS}.  We study the three terms separately for $n\in\{N_\varepsilon,\dots,N_\varepsilon+ N_*\}$ where 
\[
N_*=\left(\frac{1}{\gamma\chi\kappa}\right)^d.
\]
Recall that we first fix $0<\chi\ll 1$, then we take $0<\kappa\ll 1$ as small as needed (depending on $\chi$) and finally, let $0<\varepsilon\ll 1$ be as small as desired (depending on $\chi,\kappa$).
\medskip

\noindent\textit{Translation vector}: For the first term, it follows from the definition of $Q_{N_\varepsilon}$ that 
    \[
    \tilde {\bs b}_n=S_{\kappa,\chi}^{-1}Q_{N_\varepsilon}^{-1}\bs b_n=\begin{pmatrix}
        \kappa^{-1}\mathrm{id}&0\\
        0&\kappa^{-1}\chi\ \mathrm{id} 
    \end{pmatrix}\binom{(W_N-V_N)^{-1}W_NQ^{-1}[n\beta]}{-(W_N-V_N)^{-1}V_NQ^{-1}[n\beta]}
    \]
    where
 \[
    V_N=\varepsilon BQ(D_N-\mathrm{id})^{-1}\qquad W_N=\varepsilon BQ(D^{-1}_N-\mathrm{id})^{-1}.
    \]
    An easy computation shows that 
    \begin{equation}\label{eq:techestimateschangevariables}
    \begin{aligned}
   (W_N-V_N)^{-1}W_N=&\mathrm{diag}(\dots,(\lambda_i^{(N)}-1)/(\lambda_i^{(N)}-1/\lambda_i^{(N)}),\dots)\\
    (W_N-V_N)^{-1}V_N=&\mathrm{diag}(\dots,(1/\lambda_i^{(N)}-1)/(\lambda_i^{(N)}-1/\lambda_i^{(N)}),\dots)\\
    =&\underbrace{\mathrm{diag}(\dots,(1/\lambda_i^{(N)}-1)/(\lambda_i^{(N)}-1),\dots)}_{\tilde D_N}\  (W_N-V_N)^{-1}W_N
    \end{aligned}
    \end{equation}
    Hence,
    \[
    \tilde{\bs b}_{n,\xi}=\kappa^{-1}\mathrm{diag}(\dots,(\lambda_i^{(N)}-1)/(\lambda_i^{(N)}-1/\lambda_i^{(N)}),\dots) Q^{-1}[n\beta]\qquad\qquad \tilde{\bs b}_{n,\eta}=\chi \tilde D_N \ \tilde {\bs b}_{n,\chi}.
    \]
    so:
    \begin{itemize}
        \item In view of Theorem \ref{thm:Diophantineapprox} there exists $C>0$ (independent of $\chi,\kappa$ and $\varepsilon$) such that the sequence $\{\tilde {\bs b}_{n,\xi}\}_{N_\varepsilon\leq n\leq N_\varepsilon+N_*}$ is $C\chi$-dense on $\mathbb T^d$. In particular, there exists $\mathcal N_{\varepsilon,\kappa,\chi}\subset\{N_\varepsilon,\dots,N_\varepsilon+ N_*\}$ such that 
        \[
        \max_{\xi\in [-1,1]^{d}} \min_{n\in\mathcal N_{\varepsilon,\kappa,\chi}}\ \mathrm{dist}(\tilde {\bs b}_{n,\xi},\xi)\leq C\chi.
        \]
        \item There exists $C>0$ such that, for all $n\in \mathcal N_{\varepsilon,\kappa,\chi}$ we have 
        \[
        |\tilde{\bs b}_{n,\eta}|\leq C\chi |\tilde{\bs b}_{n,\xi}|.
        \]
    \end{itemize}
    \medskip

    \noindent\textit{Linear term:} For the second term we have already seen in Lemma \ref{lem:hyperbolicityiteratesIFS} that (recall our choice of $N_*$ above)
    \[
    S_{\kappa,\chi}^{-1}Q_{N_\varepsilon}^{-1} A^{(n)}Q_{N_\varepsilon}S_{\kappa,\chi} =\mathcal D_n+O(\varepsilon/(\gamma\chi\kappa)^d).
    \]
\medskip

\noindent\textit{Nonlinear term:} Finally for the nonlinear term we start by observing that, from the estimate \eqref{eq:errorestimatesIFSapprox} we obtain $
\mathcal E\circ\Psi_{\varepsilon,\kappa,\chi}=(O(\kappa^2),\ O(\varepsilon \kappa^2))$. 
Then, 
\[
\widetilde{\mathcal E}=\Psi^{-1}_{\varepsilon,\kappa,\chi} (\mathcal E\circ \Psi_{\varepsilon,\kappa,\chi}(\xi,\eta))=\begin{pmatrix}
 \kappa^{-1}\mathrm{id}&0\\
        0&\kappa^{-1}\chi  \mathrm{id}
    \end{pmatrix}\begin{pmatrix}(W_N-V_N)^{-1}W_NQ^{-1} &- (W_N-V_N)^{-1}\\
    -(W_N-V_N)^{-1}V_NQ^{-1}& (W_N-V_N)^{-1} \end{pmatrix} \binom{O(\kappa^2)}{O(\varepsilon\kappa^2)}
\]
so, making use of the equalities \eqref{eq:techestimateschangevariables} we obtain
\[
\widetilde {\mathcal E}=(O(\kappa),O(\chi \kappa)).
\]
On the other hand, it is easy to check from \eqref{eq:errorestimatesIFSapproxC1} that  
\[
D\mathcal{\widetilde E}=\begin{pmatrix}
    O(\kappa)&O(\kappa/\chi)\\
    O(\kappa\chi)&O(\kappa)
\end{pmatrix}.\qedhere
\]
\end{proof}
For the rest of this section we will assume without mentioning that the quantifiers $\chi,\kappa$ and $\varepsilon$ are taken sufficienly small while respecting the hierarchy introduced in Proposition \ref{prop:uniformcoordinatesIFS}.

\medskip

\subsection{Covering and well distributed periodic orbits}\label{sec:coveringandequidistributionIFs}

We now study the distribution of the images of $D=[-1,1]^{2d}$ under the family of maps $\{\mathcal F_n\}_{n\in\mathcal N_{\varepsilon,\kappa,\chi}}$.

\begin{prop}[Covering property]\label{prop:covering}
Consider the setting of Proposition \ref{prop:uniformcoordinatesIFS} and let $\mathcal N_{\varepsilon,\kappa,\chi}\subset \mathbb N$ the subset constructed therein. Then, if we denote by $D=[-1,1]^{2d}$ we have that 
\[
D\subset \bigcup_{n\in\mathcal N_{\varepsilon,\kappa,\chi}} \mathcal F_n(D).
\]
    Moreover, given $z_*=(\xi_*,\eta_*)\in D$ denote by 
    \[
    B_{r,\xi}(z_*)=\{(\xi,\eta)\in \mathbb R^{2d}\colon -r<\xi_i-\xi_{i,*}<r,\ \eta=\eta_*,\ i=1,\dots,d\}
    \]
    the horizontal rectangle centered at $z_*\in D$ of radius $r$. Then, there exists a number $\underline a>0$ independent of $\chi,\kappa$ and $\eta$ such that 
\begin{equation}\label{eq:definitionkappa}
a:=\min\{r\in\mathbb R_+\colon   \text{ there exists } z\in D\text{ such that } B_{r,\xi}(z)\subset D\text{ and } B_{r,\xi}(z)\not\subset \mathcal F_n(D) \text{ for any } n\in\mathcal N_{\varepsilon,\kappa,\chi}\}\geq \underline a.
\end{equation}

\end{prop}

\begin{proof}
   In view of the approximation \eqref{eq:approxdynamicsIFS} it is enough to show that for $K>0$ large enough (independent of $\chi,\kappa$ and $\varepsilon$) \begin{equation}\label{eq:approxcovering}
\{z\in\mathbb R^{2d}\colon \mathrm{dist}(z,D)\leq K\chi\}\subset   \bigcup_{n\in \mathcal N_{\varepsilon,\kappa,\chi}} \mathtt F_n(D)
\end{equation}
where $\mathtt F_n$ stands for the affine approximation
\[
\mathtt F_n:(\xi,\eta)\mapsto \tilde{\bs b}_n+ \mathcal D_n\binom{\xi}{\eta}.
\]
However, it is easy to observe that \eqref{eq:approxcovering} holds.  Indeed, from their definition in Lemma \ref{lem:hyperbolicityiteratesIFS} we notice that there exist $0<\underline \lambda<\bar\lambda< 1$ independent of $\chi,\kappa$ and $\varepsilon$ such that 
\[
\underline\lambda<\min\{|\lambda_i^{(n)}|\colon i=1,\dots,d,\ n\in\mathcal N_{\varepsilon,\kappa,\chi}\}\qquad\qquad \bar\lambda>\max\{|\lambda_i^{(n)}|\colon i=1,\dots,d,\ n\in\mathcal N_{\varepsilon,\kappa,\chi}\}.
\]
Then 
\[
\tilde{\bs b}_n+[-\underline\lambda,\underline\lambda]^d\times[-\bar\lambda^{-1},\bar\lambda^{-1}]^d\subset \mathtt F_n(D),
\]
and the desired  conclusion follows since $\{\tilde{\bs b}_{n,\xi}\}_{n\in\mathcal N_{\varepsilon,\kappa,\chi}}$ is (for some $C>0$ independent of $\chi,\kappa$ and $\varepsilon)$ $C\chi$-dense in $[-1,1]^d$ and $|\tilde{\bs b}_{n,\eta}|\leq C\chi|\tilde{\bs b}_{n,\xi}|\leq C\chi$ for $n\in\mathcal N_{\varepsilon,\kappa,\chi}$.
\medskip

The second part of the statement follows from a rather similar argument. Indeed, for any $z_*=(\xi_*,\eta_*)\in D$ there exists $n_*\in\mathcal N_{\varepsilon,\kappa,\chi}$ such that $\mathrm{dist}(\tilde{\bs b}_{n_*,\xi},\xi_*)\leq C\chi$. Let now $r<\underline\lambda/2$. Then, 
\[
B_{r,\xi}(z_*)\times[-1,1]^d\subset \tilde{\bs b}_n+[-\underline\lambda,\underline\lambda]^d\times[-\bar\lambda^{-1},\bar\lambda^{-1}]^d\subset \mathtt F_n(D)
\]
and the claim follows after shrinking, if necessary, the value of $\chi$.
\end{proof}

We now introduce a suitable class of submanifolds which are well aligned with the contracting direction of the maps $\{\mathcal F_n\}_{n\in\mathcal N_{\varepsilon,\kappa,\chi}}$.
\begin{defn}\label{defn:cstrip}
    We say that $\Delta\subset Q^{cs}$ is a \textit{cs-strip} if, in the coordinate system $(\xi,\eta)\in[-1,1]^{2d}$ introduced in \eqref{eq:coordinatesforIFS} admits a parametrization of the form 
    \[
    \Delta=\{(\xi,h(\xi))\in[-1,1]^{2d}\colon \xi\in U \}
    \]
    where $U\subset [-1,1]^d$ is an open subset and $h:U\to [-1,1]^d$ is a $C^1$ function satisfying $|\partial_\xi h|\leq 1$. We define the \textit{width} of $\Delta$ as 
    \[
    \mathrm{width}(\Delta)=\sup\{r\in \mathbb R\colon \xi+[-r,r]^d\subset U, \ \xi\in U\}.
    \]
\end{defn}

Since cs-strips are well aligned with the contracting direction they are expanded under the application of the inverse map.

\begin{lem}\label{lem:graphtransofrmIFS}
There exists $0<\bar\lambda<1$ (independent of $\chi,\kappa$ and $\varepsilon$) such that the following holds.  Let $\Delta\subset D$ be a cs-strip and suppose that there exists $n\in\mathcal N_{\varepsilon,\kappa,\chi}$ such that $\Delta\subset \mathcal F_n (D)$. Then $\Delta_n=\mathcal F_n^{-1}(\Delta)$ is a cs-strip and 
\begin{equation}\label{eq:widthincrease} 
\mathrm{width}(\Delta_n)\geq (1/\bar\lambda) \  \mathrm{width}(\Delta).
\end{equation}
\end{lem}
\begin{proof}
 It is trivial to see that the inverse map $\mathcal F_n^{-1}$ is given by 
 \[
 \mathcal F_n^{-1}:\binom{\bar \xi}{\bar\eta}\mapsto \hat{\bs b}_n+\mathcal D_n^{-1}\binom{\bar \xi}{\bar\eta}+\widehat{\mathcal E}(\bar\xi,\bar \eta)
 \]
 for some $\hat {\bs b}_n$ and $\widehat{\mathcal E}$ satisfying similar properties to $\tilde{\bs b}_n$ and $\widetilde{\mathcal E}$ in Proposition \ref{prop:uniformcoordinatesIFS}. Let $\Delta$ be a cs-strip and let $\bar U\subset[-1,1]^d$ and $\bar h:\bar U\to [-1,1]^d$ be as in Definition \ref{defn:cstrip}. Then, the equation
 \[
 \xi=\hat{\bs b}_{n,\xi}+ D_n^{-1} \ \bar \xi+\widehat{\mathcal E}_{\xi}(\bar\xi,\bar h(\bar\xi))
 \]
 admits a unique $C^1$ solution $\bar\xi=\bar\xi(\xi)$ defined for all $\xi\in  U\subset[-1,1]^d$ where $U=(1/\bar\lambda) \bar U$ and $\bar\lambda=\max\{|\lambda_i^{(n)}|\colon i=1,\dots,d,\ n\in\mathcal N_{\varepsilon,\kappa,\chi}\}$. Now we define $h:U\to [-1,1]^d$ by the implicit equation
 \[
 h(\xi)=\hat {\bs b}_{n,\eta}+D_n\  \bar h(\bar\xi(\xi))+\widehat{\mathcal E}_\eta(\bar\xi(\xi),\bar h(\bar\xi(\xi))).
 \]
Clearly, the solution $h$ to this equation satisfies that 
\[
|\partial_\xi h|\leq \overline \lambda |\partial_\xi \bar h|+O(\chi)<1.
\]
The proof is then completed after noticing that \eqref{eq:widthincrease} follows from the definition of $U$.
\end{proof}

Together with the Covering Property in Proposition \ref{prop:covering}  (which will allow us to make use of Lemma \ref{lem:graphtransofrmIFS} to expand cs-strips iteratively), the following is the second main ingredient needed to construct a symbolic cs-blender.

\begin{prop}[Well distributed fixed points]\label{prop:welldistributed}
    There exists a subset $\mathcal N_{\varepsilon,\kappa,\chi,\mathrm{per}}\subset\mathcal N_{\varepsilon,\kappa,\chi}$ such that, for each $n\in \mathcal N_{\varepsilon,\kappa,\chi,\mathrm{per}}$ the map $\mathcal F_n$ has a hyperbolic fixed point $z_n=(\xi_n,\eta_n)\in D$ and such that:
    \begin{enumerate}
        \item There exists a constant $C>0$ (independent of $\chi,\kappa$ and $\varepsilon$) such that 
        \[
        \max_{\xi\in [-1,1]^d}\min_{n\in\mathcal N_{\varepsilon,\kappa,\chi,\mathrm{per}}}\mathrm{dist}(\xi_n,\xi)\leq C\chi.
        \]
        \item Its unstable manifold is a fully crossing vertical submanifold that is, a $d$-dimensional submanifold which admits  a graph parametrization of the form 
\begin{equation}\label{eq:parametrizationunstableIFS}
W^u(z_n;\mathcal F_n)=\{(f_n(\eta),\eta)\colon \eta\in[-1,1]^d\}
\end{equation}
for some differentiable function $f_n:[-1,1]^d\to [-1,1]^d$. Moreover, the function $f_n$ satisfies $|\partial_\eta f_n|\lesssim \chi$.
\item Its stable manifold is a fully crossing horizontal submanifold, that is a $d$-dimensional submanifold admitting a parametrization of the form 
\[
W^s(z_n;\mathcal F_n)=\{(\xi,\tilde f_n(\xi))\colon \xi\in[-1,1]^d\}\]
for some differentiable function $\tilde f_n:[-1,1]^d\to [-1,1]^d$. Moreover, the function $\tilde f_n$ satisfies $|\partial_\xi \tilde f_n|\lesssim \chi$.
    \end{enumerate}
\end{prop}
The proof of this result follows plainly from the stable manifold theorem. Moreover, notice that part of the argument in the proof of that theorem, boils down to the study of the graph transform which we already analyzed in the proof of Lemma \ref{lem:graphtransofrmIFS}. We refer the interested reader to \cite{PalisMelo,MR1326374}. 

\medskip

\subsection{Existence of a symbolic cs-blender}\label{sec:csblenderIFS}
Building now on the results in Propositions \ref{prop:covering} and \ref{prop:welldistributed} we construct a symbolic cs-blender for the family of maps $\{\mathcal F_n\}_n$ (what implies the existence of a symbolic cs-blender for the IFS generated by $\{T,S\}$). To do so we  study the action of the family of maps $\{\mathcal F_n\}_{n\in\mathcal N_{\varepsilon,\kappa,\chi}}$ on cs-strips. 

\begin{prop}\label{prop:csblender1}
  Let $\mathcal N_{\varepsilon,\kappa,\chi,\mathrm{per}}\subset\mathcal N_{\varepsilon,\kappa,\chi}\subset \mathbb N$ be the sets constructed in Propositions \ref{prop:uniformcoordinatesIFS} and \ref{prop:welldistributed} and, for $n\in\mathcal N_{\varepsilon,\kappa,\chi,\mathrm{per}}$, let $W^u(z_n;\mathcal F_n)\subset Q^{cs}$ be the unstable manifolds of the hyperbolic fixed points in Proposition \ref{prop:welldistributed}. Let $\Delta\subset Q^{cs}$ be a cs-strip. Then, there exists $n_*\in \mathcal N_{\varepsilon,\kappa,\chi,\mathrm{per}}$, $L\in\mathbb N$ and $\omega\in\mathcal N_{\varepsilon,\kappa,\chi}^L$ such that 
  \[
  \mathcal F_{\omega_{L-1}}^{-1}\circ\cdots \circ \mathcal F_{\omega_0}^{-1}(\Delta) \pitchfork W^u(z_{n_*};\mathcal F_{n_*})\neq\emptyset.
  \]
  
\end{prop}

\begin{proof}
    Let $a$ be the number introduced in \eqref{eq:definitionkappa} and recall that there exists $\underline a>0$ independent of $\chi,\kappa$ and $\varepsilon$ such that $a\geq \underline a$. Introduce 
    \[
    b=\max_{(\xi,\eta)\in D} \min_{n\in\mathcal N_{\varepsilon,\kappa,\chi,\mathrm{per}}} |\xi-f_n(\eta)|
    \]
    where $f_n:[-1,1]^d\to [-1,1]^d$ is the function in the parametrization \eqref{eq:parametrizationunstableIFS}. In view of Proposition \ref{prop:welldistributed} there exists $C>0$ independent of $\chi,\kappa$ and $\varepsilon$ such that
    \[
   b<C\chi<\underline a/2<\underline a\leq a
    \]
    Now, given $\Delta\subset Q^{cs}$ we distinguish two scenarios:
    \begin{enumerate}
        \item If $\mathrm{width}(\Delta)>\underline a/2$ then, since $\underline a/2>b$, there exists $n_*\in\mathcal N_{\varepsilon,\kappa,\chi,\mathrm{per}}$ such that $\Delta \pitchfork W^u(z_{n_*};\mathcal F_{n_*})\neq\emptyset$ so we are done.
        \item If $\mathrm{width}(\Delta)<\underline a/2$, since $\underline a<a$, we  can find $\widetilde \Delta\subset \Delta$ with 
        \[
        \mathrm{width}(\widetilde \Delta)=\mathrm{width}(\Delta)
        \]
         and $n\in\mathcal N_{\varepsilon,\kappa,\chi}$ such that $\widetilde\Delta\subset \mathcal F_n(D)$. Hence, thanks to Lemma \ref{lem:graphtransofrmIFS} the submanifold $\Delta_1=\mathcal F_n^{-1}(\widetilde\Delta)$ is a cs-strip and satisfies $\mathrm{width}(\Delta_1)\geq \underline \chi \mathrm{width}(\Delta)$. If $\mathrm{width}(\Delta_1)>\underline a/2$ then we arrive to the first scenario. Otherwise we repeat this argument. Since $\underline \chi>1$ is uniform, after a finite number of steps we must arrive to the first scenario.
    \end{enumerate}
    Finally, the proof is concluded making use of the lambda-lemma \cite{PalisMelo} after noticing that for any two  $z_n,z_{n_*}$ with $n,n_*\in\mathcal N_{\varepsilon,\kappa,\chi,\mathrm{per}}$ we have 
    \[
    W^u(z_{n};\mathcal F_{n_*})\pitchfork W^s(z_{n_*};\mathcal F_{n_*})\neq\emptyset,\qquad\qquad W^s(z_{n};\mathcal F_{n_*})\pitchfork W^u(z_{n_*};\mathcal F_{n_*})\neq\emptyset.\qedhere
    \]
\end{proof}

The following is a reformulation of Proposition \ref{prop:csblender1}.

\begin{prop}\label{prop:csblender2}
    Let $Q^{cs}$ as in \eqref{eq:csblenderdomain} and let $P^{cs}\in Q^{cs}$ be given by the image under \eqref{eq:coordinatesforIFS} of the point $z_{n_*}\in[-1,1]^{2d}$ in Proposition \ref{prop:csblender1}. Then, the pair $(P^{cs},Q^{cs})$ forms a symbolic cs-blender for the IFS generated by $\{T,S\}$.
\end{prop}

\medskip

\subsection{Existence of a symbolic double blender}\label{sec:doubleblenderIFS}
We now make use of a symmetry consideration to deduce the existence of a symbolic double blender, hence completing the first part of the proof of Theorem \ref{thm:mainIFS}.
To do so, for $n\in\mathbb N$ we define the map
\begin{equation}\label{eq:mapsforcublender}
\widetilde F_n:=S\circ T^n:\mathbb A^d\to \mathbb A^d.
\end{equation}
Then, after some algebraic manipulations, it is not difficult to observe that (compare \eqref{eq:localaffineapprox})
\begin{equation}\label{eq:expressionreversibility}
\widetilde F_n^{-1}:\binom{\varphi}{J}\mapsto {\bs b}_n+\begin{pmatrix}
    \mathrm{id}+n\varepsilon AB &-nA\\
    -\varepsilon B & \mathrm{id}
\end{pmatrix}\binom{\varphi}{J}+\widetilde{\mathcal E} (\varphi,J)
\end{equation}
with $\bs b_n$  as in \eqref{eq:localaffineapprox} and $\widetilde {\mathcal E}(\varphi,J)$ satisfying the very same estimates as $\mathcal E$ in Lemma \ref{lem:C1control}. For the involution
\begin{equation}\label{def:involution}
\psi_R:\begin{pmatrix}\varphi\\J\end{pmatrix}\mapsto \begin{pmatrix}-\varphi\\J\end{pmatrix}
\end{equation}
one can check that 
\[
\widetilde F_n^{-1}=\psi_R\circ (T^n\circ S)\circ \psi_R+ \overline{\mathcal E}
\]
with $\overline{\mathcal E}$ satisfying the same estimates as $\mathcal E$ in Lemma \ref{lem:C1control}.
In other words, $\widetilde F_n^{-1}=(S\circ T^n)^{-1}$ is conjugate (up to small errors) to $F_n=T^n\circ S$.
Therefore, verbatim repetition of the discussion in Sections \ref{sec:linearapproxIFS}-\ref{sec:csblenderIFS} shows that there exists a point 
\[
    P^{cu}\sim \psi_R (P^{cs})
    \]
which is a hyperbolic fixed point for the map $\widetilde F_{N_l}$ and such that the pair $(P^{cu},Q^{cu})$ where $Q^{cu}=\psi_R(Q^{cs})$
is a $cu$-blender for the IFS generated by $\{T,S\}$ (see Figure \ref{fig:fig8}). We now show that $P^{cu}$ and $P^{cs}$ are homoclinically related to conclude the existence of a  symbolic double blender. To alleviate the notation we denote by $W^{u}(P^{cs})=W^{u}(P^{cs};F_{N_l})$ and by $W^{s}(P^{cu})=W^{s}(P^{cu};\widetilde F_{N_l})$.

\begin{figure}
    \centering
    \includegraphics[scale=0.75]{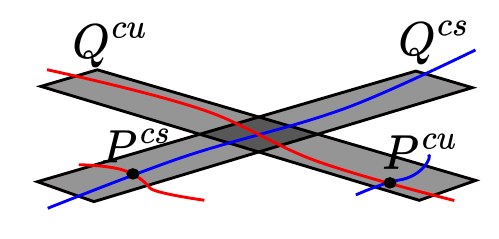}
    \caption{A sketch of the relative position between the sets $Q^{cs}$ and $Q^{cu}$.}
    \label{fig:fig8}
\end{figure}
\begin{lem}\label{lem:homrelatedblenders}
    Let  and $P^{cs}\in Q^{cs}$ be the point in Proposition \ref{prop:csblender2} and let $P^{cu}\in Q^{cu}$ be as above. Then, we have that  $W^{u}(P^{cs})\pitchfork W^s(P^{cu})\neq 0$.
\end{lem}
\begin{proof}
On the coordinate system $(\xi,\eta)\in [-1,1]^{2d}$ on $Q^{cs}$ given by the linear map  in \eqref{eq:coordinatesforIFS}, we have shown on Proposition \ref{prop:welldistributed} that the local unstable manifold $W^u(P^{cs})$ is given by a $C^1$ submanifold which is almost vertical and fully crosses the rectangle $[-1,1]^{2d}$. On the other hand, on the coordinate system $(\tilde \xi,\tilde \eta)\in[-1,1]^{2d}$ on $Q^{cs}$ given by the linear map $\phi\circ \psi_R$, it follows by construction that $W^s(P^{cu})$ is given by a $C^1$ submanifold which is almost vertical and fully crosses the rectangle $[-1,1]^{2d}$. A straightforward computation shows that on $Q^{cu}\cap Q^{cs}$ the transition map between the two coordinate charts is given by a rotation by $90$ degrees so the proof follows (see  Figure \ref{fig:fig8}).
\end{proof}

Combining Proposition \ref{prop:csblender2} and Lemma \ref{lem:homrelatedblenders} we have proven the existence of a symbolic double blender.

\begin{prop}\label{prop:symbolicdoubleblender}
    The pairs $(P^{cu},Q^{cu})$ and $(P^{cs},Q^{cs})$ constructed above form a symbolic double blender for the IFS generated by the maps $\{T,S\}$.
\end{prop}

\subsection{Proof of Theorem \ref{thm:mainIFS}}\label{sec:completionproofmainifs} We now make use of the symbolic double blender in Proposition \ref{prop:symbolicdoubleblender} to conclude the proof of Theorem \ref{thm:mainIFS}. The idea is exactly the same as in Section 3.6 of \cite{guardia2025partiallyhyperbolicdynamics3body} but adapted to the higher dimensional context. We reproduce it here for the sake of self-completeness.
\medskip

Given $c>0$ (independent of $\varepsilon,\gamma$) we denote by 
\begin{equation}\label{eq:stripsaroundKAM}
\mathbb A_{c\varepsilon}^d=\mathbb T^d\times [-c\varepsilon,c\varepsilon]^d\qquad\qquad {\mathbb A}^d_{\sqrt{\varepsilon}}=\mathbb T^d\times[-\sqrt\varepsilon, \sqrt\varepsilon]^d.
\end{equation}
We divide the proof in three steps:
\begin{itemize}
    \item We first notice that if $B\subset Q^{cu}$ and $B'\subset Q^{cs}$ then, the conclusion follows from the symbolic double blender dynamics.
    \item Second, we show that if $B,B'\in\mathbb A_{c\varepsilon}^d$ with $c>0$ sufficiently small (independent of $\varepsilon,\gamma$), there exists $n_f,n_b\in\mathbb N$ such that $T^{n_f}(B)\cap Q^{cu}\neq\emptyset$ and $T^{-n_b}(B')\cap Q^{cs}\neq\emptyset$.
    \item Finally, we complete the proof by showing that for any $B,B'\in { \mathbb A}_{\sqrt\varepsilon}^d $ there exist $n_f',n_b'\in\mathbb N$ such that $T^{n_f'}(B)\cap \mathbb A^d_{c\varepsilon} \neq\emptyset $ and $T^{-n_b'}(B')\cap \mathbb A^d_{c\varepsilon} \neq\emptyset$.
\end{itemize}

\noindent\textbf{Step 1:}
$B$ is open so it contains a $cu$-strip $\Delta_B$. Hence, since $(P^{cu},Q^{cu})$ is a symbolic $cu$-blender, the unstable manifold $W^s(P^{cu})$ intersects $\Delta_B$ transversally.  On the other hand, $B'$ is open so it contains a  $cs$-strip $\Delta_{B'}$ and the fact that $(P^{cs},Q^{cs})$ is a symbolic $cs$-blender implies that $W^u(P^{cs})$ intersects $\Delta_{B'}$ transversally. The conclusion now follows from a direct application of the lambda-lemma (see \cite{PalisMelo}) and the fact that $P^{cs}$ and $P^{cu}$ are homoclinically related.

\noindent \textbf{Step 2:} We only deal with the existence of $n_f$, the existence of $n_b$ being deduced from the same argument. Let $(\varphi,J)\in \mathbb A_{c\varepsilon}^d$ and recall that 
\[
Q^{cu}=\psi_R\circ \phi([-1,1]^{2d}),
\]
with $\Psi:[-1,1]^{2d}\to\mathbb A^d$ as in \eqref{eq:coordinatesforIFS} and $\psi_R:(\varphi,J)\mapsto (-\varphi,J)$. We show that, provided $c>0$ is small enough,  there exists $L\in\mathbb N$ such that $T^{L}(\varphi,J)\in  Q^{cu}$.
This is done in two steps:
\begin{itemize}
    \item First, we notice that there exists $\ell>0$ independent of $\varepsilon$ such that, for any $c>0$ small and for any $J_*\in[-c\varepsilon,c\varepsilon]^d$ there exists  intervals $\{I_{J,i}\}_{i=1}^d\subset\mathbb T$ of length $|I_{J,i}|\geq \ell$ such that 
    \[
   Q_{J}:= \bigg(\prod_{i=1}^d I_{J_*,i}\times [J_{i}-\varepsilon^2,J_{i}-\varepsilon^2]\bigg)\subset \psi_R\circ\phi([-1/2,1/2]^2)\subset Q^{cu}.
    \]
    That is, at any height $J\in[-c\varepsilon,c\varepsilon]^d$ we can embed a flat  (the vertical size is $O(\varepsilon^2)$) cube $Q_{J}$ inside $Q^{cu}$ with horizontal sides of length independent of $\varepsilon$.
\item Second, for any $C>0$,  any $n\leq C/\gamma^d$, and any
$(\varphi,J)\in \mathbb A_{c\varepsilon}^d$, we have
\begin{equation}\label{eq:dense}
T^n(\varphi,J)=\begin{pmatrix}\varphi+[n\beta]+O(C\gamma^{-d}c\varepsilon)\\ J+O(C\gamma^{-d}c^3\varepsilon^3)\end{pmatrix}.
\end{equation}
By Theorem \ref{thm:Diophantineapprox} the sequence $\{[n\beta]\}_{n=1,\dots,[C/\gamma^d]}$ is $\frac{1}{C}$-dense in $\mathbb T^d$ so, for $C\geq  10\ell^{-1}$,  in view of \eqref{eq:dense} and the fact that $|I_{J,i}|\geq \ell$ (for all $i=1,\dots,d$) we deduce that that, at any
$(\varphi,J)\in \mathbb A_{c\varepsilon}^d$,
\[
\{T_0^n(\varphi,J)\}_{n=1,\dots,[C/\gamma^d]}\cap Q_J\neq \emptyset.
\]
\end{itemize}

\noindent \textbf{Step 3:} We only deal with the existence of $n_f'$, the existence of $n_b'$ being deduced from the same argument. Recall that the map $S$ is of the form
\[
S:\binom{\varphi}{J}\mapsto \binom{\varphi+\varepsilon P_{\varepsilon,\varphi}(\varphi,J)}{J+\varepsilon P_{\varepsilon,J}(\varphi,J)}
\]
with $P_{0,J}(0,0)=0$ and $
B=D_{\varphi}P_{0,J}(0,0)$ 
being a nondegenerate matrix. In particular, for any $|J|\leq \sqrt\varepsilon$ there exists a smooth curve
$\varphi_*=\varphi_*(J)$ such that $
P_{\varepsilon,J}(\varphi_*(J),J)=0$ and $\widetilde B_{\varepsilon}(J):=D_\varphi P_{\varepsilon,J}(\varphi_*(J),J)$ is non-degenerate.

Let now $C>0$ be any large constant independent of $\varepsilon$.  In view of Theorem \ref{thm:Diophantineapprox},   if we let $
    \widetilde C(\gamma,\varepsilon)=C\left(\frac{\log\varepsilon}{\gamma}\right)^d$, the set $\{[n\beta]\}_{n\leq \widetilde C(\gamma,\varepsilon)}$ is $1/C|\log\varepsilon|$-dense in $\mathbb T^d$. Moverover, $(\varphi,J)\in {\mathbb A}_{\sqrt\varepsilon}^d$ and $n\leq \widetilde C(\gamma,\varepsilon)$
   \begin{equation}\label{eq:adjustingangle}
T^n:\binom{\varphi}{J}\mapsto \begin{pmatrix}\varphi+[n\beta]+O(n \sqrt{\varepsilon})\\ J+O(n \varepsilon^{3/2}) \end{pmatrix}=\begin{pmatrix}\varphi+[n\beta]+O(\gamma^{-d}|\log\varepsilon|^d\sqrt\varepsilon)\\ J+O(\gamma^{-d}|\log\varepsilon|^d\varepsilon^{3/2}) \end{pmatrix},
   \end{equation}
Given any choice of $\iota\in\{+,-\}^d$, we choose $n_\iota(\varphi,J)\in\mathbb N$ such that (this is possible because $\widetilde B_\varepsilon$ is invertible and the sequence $\{[n\beta]\}_{n\leq \widetilde C(\gamma,\varepsilon)}$ is $1/C|\log\varepsilon|$-dense in $\mathbb T^d$)
\[
\widetilde B_\varepsilon(J) (\varphi+[n_+\beta]-\varphi_*(J))\in \prod_{i=1}^d (\min\{\iota_i/\log\varepsilon,4\iota_i/\log\varepsilon\},\max\{\iota_i/\log\varepsilon,4\iota_i/\log\varepsilon\})
\]
Then, after writing 
\[
\pi_J S(\varphi,J)=\varepsilon \widetilde B_\varepsilon(J)(\varphi-\varphi_*(J))+O(\varepsilon|\varphi-\varphi_*(J)|^2),
\]
it follows from \eqref{eq:adjustingangle} that 
\begin{align*}
\Delta_\iota J:=\pi_J (S\circ T^{n_\iota})(\varphi,J)-J&\in \prod_{i=1}^d (\min\{2\iota_i \varepsilon/\log\varepsilon,3\iota_i\varepsilon/\log\varepsilon\},\max\{2\iota_i\varepsilon/\log\varepsilon,3\iota_i\varepsilon/\log\varepsilon\})
\end{align*}
If $J+\Delta_\iota J\in[-c\varepsilon,c\varepsilon]^d$ (for $c$,  as in Step 2) we are done. If not, we repeat the argument a finite number of times. The proof of Theorem \ref{thm:mainIFS} is completed.
\medskip

\section{Transport in   IFS generated by close-to-identity maps}\label{sec:transporIFS}

The goal of this section is to give a proof of Proposition \ref{prop:maintransportIFS}. Namely, to show that under assumption \ref{it:B6}, the system of maps $\{S_{i,\varepsilon}\}_{i=1}^m:N\to N$ has the $\varepsilon$-reachability property.  However, this result is not  strong enough to later be of use for the proof of Proposition \ref{prop:skewprodmain}. Namely, in the skew-product setting constructed in Proposition \ref{prop:mainnhim}, we are forced to consider a rather particular subclass of orbits of the  iterated function system composed by the inner map $T_\varepsilon:N\to N$  and the system of maps $\{\{S_{i,\varepsilon}\}_{i=1}^m\}:N\to N$. In particular, we are only allowed to choose orbits composed by segments of the form $T^n\circ S_i$ (instead of arbitrary choices between the maps $\{T_\varepsilon,\{S_{i,\varepsilon}\}_{i=1}^m\}\}$). Hence, instead of Proposition \ref{prop:maintransportIFS} we prove Proposition \ref{prop:epsaccesibility}, which is tailored for this subclass of orbits.
\medskip

\begin{rem}\label{rem:lessymsymbol}
    Given two functions $f(\varepsilon,z)$ and $h(\varepsilon,z)$, we write $f\lesssim h$ if there exists $C>0$ such that uniformly on $(\varepsilon,z)$ we have $f(\varepsilon,z)\leq C h(\varepsilon,z)$. The same notation will be used in Section \ref{sec:skewprod}.
\end{rem}

The first step in the proof of Proposition \ref{prop:epsaccesibility} can be seen as a discrete time analogue of the classical Ball-Box theorem (see \cite{MR1867362}).

\begin{prop}\label{prop:Chow}
    Let $m\geq 2$ and let  $\{S_{i,\varepsilon}\}_{i=1}^m:N\to N$ satisfy assumption \ref{it:B6}. Then, there exists $K>0$ such that for any two points $z,z_*\in N$ there exists $L\in \mathbb N$  and a pair
    \[
    \iota\in\{1,\dots, m\}^{L},\qquad \varsigma\in\{+1,-1\}^{L}
    \]
    such that 
    \begin{equation}\label{eq:orbitIFSChow}
\mathrm{dist}(S_{\iota_{L-1}}^{\varsigma_{L-1}}\circ\cdots S_{\iota_{0}}^{\varsigma_{0}}(z),\  z_*)\leq K\varepsilon.
    \end{equation}
\end{prop}

\begin{proof}
Let $2d=\mathrm{dim}(N)$ and let $\psi:[-2,2]^{2d}\subset\mathbb R^{2d}\to N$ be a chart of $N$ as in assumption \ref{it:B6}. For $i=1,\dots, m$, the expression of the scattering map $S_i$ in this coordinates reads (we abuse notation and write $S_{i,\varepsilon}$ instead of $\psi^{-1}\circ S_{i,\varepsilon}\circ\psi$)
\[
S_{i,\varepsilon}(z)=z+\varepsilon X_i(z)+\varepsilon^2 R_{i,\varepsilon}(z)\qquad\qquad |R_{i,\varepsilon}|_{C^r}\lesssim 1.
\]
We start by introducing some notation. Given two vector fields $X,Y$ we denote by $[X,Y]$ their Lie bracket. Abusing notation, for diffeomorphisms $\phi,\psi$ we define their commtator
\[
[\phi,\psi]=\phi^{-1}\circ\psi^{-1}\circ\phi\circ\psi.
\]
Now, for any $n\geq 2$ and $\iota\in\{1,\dots,m\}^n$ we define 
\[
 Y_\iota=[X_{\iota_{n-1}},[\dots[X_{\iota_2},[X_{\iota_1},X_{\iota_0}]]\dots]]\qquad\qquad \mathcal S_{\iota,\varepsilon}=[S_{\iota_{n-1},\varepsilon},[\dots[S_{\iota_2,\varepsilon},[S_{\iota_1,\varepsilon},S_{\iota_0,\varepsilon}]]\dots].
\]
In the case $n=1$ we write
\[
 Y_\iota=X_{\iota_0}\qquad\qquad \mathcal S_{\iota,\varepsilon}=S_{\iota_0,\varepsilon}.
\]
We claim that for any $n\leq \tilde r$ (recall that $\tilde r\geq 0$ enters in \ref{it:B6} as the order of the distribution which generates $\mathbb R^{2d}$) and any $\iota\in\{1,2\}^n$ we have 
\begin{equation}\label{eq:inductionconmutators}
\mathcal S_{\iota,\varepsilon}(z)=z+\varepsilon^n Y_\iota(z)+\varepsilon^{n+1}\mathcal R_{\iota,\varepsilon}(z)\qquad\qquad|\mathcal R_{\iota,\varepsilon}|_{C^{\tilde r-n+1}}\lesssim 1.
\end{equation}
We now show how to conclude the proof of the proposition and verify the claim  afterwards. To do so, fix $K>0$ large enough (independent of $\varepsilon$) and  cover $[-2,2]^d$ by the family of boxes $\{B_{\bs a}\}_{\bs a}$
\[
B_{\bs a}=\bs a+[-2K\varepsilon, 2K\varepsilon]^{2d}
\]
where
\[
\bs a =(a_1 K\varepsilon, a_1 K\varepsilon,\dots, a_d K\varepsilon), \qquad\qquad (a_1,\dots,a_d)\in\{-[2/(K\varepsilon)],\dots,0,\dots,[2/(K\varepsilon)]\}^d\subset\mathbb N^d.
\]
Given any $\bs a,\bs a'$ there exists a finite chain $\{\bs a^{(l)}\}_{l=1}^L$ such that $\bs a^{(1)}=\bs a$ and $\bs a^{(L)}=\bs a'$. We will show that given any $z^{(l)}\in B_{\bs a^{(l)}}$ there exists an orbit of the IFS generated by $\{S_{i,\varepsilon}\}_{i=1}^m$ and their inverses which moves $z^{(l)}$ to $B_{\bs a^{(l+1)}}$.

\medskip

We start by noticing that we can find $\iota^{(1)},\dots \iota^{( 2d)}\subset \{1,\dots,m\}^{\tilde r}$ such that $\{Y_{\iota^{(j)}}\}_{j=1}^{ 2d}\subset \mathrm{Lie}^{\tilde r}\{X_1,\dots, X_m\}$   satisfying at each $z\in B_{\bs a}$
\[
\mathrm{det} \underbrace{\left(Y_{\iota^{(1)}}|\dots|Y_{\iota^{( 2d)}}\right)(z)}_{G(z)}\neq 0.
\]
Then, we introduce the linear coordinate transformation on $B_{\bs a^{(i)}}$ given by 
\[
\psi_l:y\mapsto z=\bs a^{(l)}+G(\bs a^{(l)}) y.
\]
By construction, in the new coordinate system we have (observe that, by continuity, for any $\tilde z\in B_{\bs a^{(l)}}$ and any $j=1,\dots, d$ we have the upper bound $|Y_{\iota^{(j)}}(\tilde z)-Y_{\iota^{(j)}}(\bs a^{(l)})|\lesssim \varepsilon$), 
\[
\psi^{-1}_l\circ\mathcal S_{\iota^{(j)},\varepsilon}\circ\psi_l:y\mapsto y+\varepsilon^{w_j}e_j+O(\varepsilon^{w_j+1})\qquad\qquad w_j=\mathrm{card}\{i\in\mathbb N\colon \iota^{(j)}_i\neq 0\}.
\]
and $e_j\in \mathbb R^d$ is the $j$-th element of the canonical  basis of $\mathbb R^d$. This concludes the proof of Proposition \ref{prop:Chow} modulo verification of the claim \eqref{eq:inductionconmutators}.

\medskip

To prove the validity of this claim we first observe that
\eqref{eq:inductionconmutators} trivially  holds for $n=1$. Now suppose that \eqref{eq:inductionconmutators}  is true for some $n\in\mathbb N$ and a given $\iota\in\{1,\dots,m\}^n$. Let $i_{n}\in\{1,\dots,m\}$ and define  $\iota'=(\iota_0,\dots, i_{n})$.  Then, 
\[
S_{\iota_{n},\varepsilon}\circ \mathcal S_{\iota,\varepsilon}(z)=S_{i_{n},\varepsilon}(z)+\varepsilon^n Y_\iota(z)+\varepsilon^{n+1}(\mathcal R_{\iota,\varepsilon}(z)+DX_{i_{n}}\ Y_\iota(z))+\varepsilon^{n+2}\ \widetilde{R}_{\iota',\varepsilon}(z),\quad\quad |\widetilde R_{\iota',\varepsilon}|_{C^{\tilde r-n}}\lesssim 1.
\]
Using this formula and the fact that 
\[
\mathcal S_{\iota,\varepsilon}^{-1}(z)=z-\varepsilon^n Y_\iota-\varepsilon^{n+1} \mathcal R_{\iota,\varepsilon}(z)+\varepsilon^{2n} DY_\iota Y_\iota+ \varepsilon^{\mathrm{min}\{2n,n+1\}+1} \widehat R_{\iota,\varepsilon}(z),\qquad\qquad |\widehat R_{\iota',\varepsilon}|_{C^{\tilde r-n}}\lesssim 1,
\]
is easy to check that 
\[
\mathcal S_{\iota,\varepsilon}^{-1}\circ S_{\iota_{n},\varepsilon}\circ \mathcal S_{\iota,\varepsilon}(z)=S_{i_{n},\varepsilon}(z)+\varepsilon^{n+1}[X_{i_{n}},Y_\iota](z)+\varepsilon^{n+2} \bar R_{\iota',\varepsilon}(z)\qquad\qquad|\bar R_{\iota'}|_{C^{\tilde r-n}}\lesssim 1,
\]
Since
\[
\mathcal S_{\iota',\varepsilon}=S_{i_{n},\varepsilon}^{-1}\circ\mathcal S_{\iota,\varepsilon}^{-1}\circ S_{\iota_{n},\varepsilon}\circ \mathcal S_{\iota,\varepsilon},
\]
the claim now follows by a simple computation. This concludes the proof.
\end{proof}
\medskip

With the intention of applying it to the skew-product setting in Proposition \ref{prop:mainnhim} we now make use of Poincar\'e recurrence to upgrade the  result in Proposition \ref{prop:Chow} to forward orbits of the IFS generated by
\[
\{\{T_\varepsilon^n\circ S_{1,\varepsilon}\}_{n=n_*}^\infty,\dots, \{T^n\circ S_{m,\varepsilon}\}_{n=n_*}^\infty\}
\]
where $n_*\in\mathbb N$ is some large natural number, $T_\varepsilon:N\to N$ is a diffeomorphism which a smooth measure  on $N $ and $\{S_{i,\varepsilon}\}_{i=1}^M:N\to N$ preserve the same measure and satisfy \ref{it:B6}  (later we will think of these maps as those driving the skew-product in Proposition \ref{prop:mainnhim}). 
\begin{rem}
In order to  alleviate the notation, as we did in Section \ref{sec:blenderIFS}, for the rest of this section we simply write $T$ and $\{S_i\}_{i=1}^m$ instead of $T_\varepsilon$ and $\{S_{i,\varepsilon}\}_{i=1}^m$.
\end{rem}

We introduce the sets
\[
\mathcal R_0=\{z\in N\colon z\text{ is recurrent for } T\},\qquad\qquad \mathcal R_{i}=\{z\in N\colon z\text{ is recurrent for } S_i\},\qquad i=1,\dots,m
\]
let $G=\langle T,S_1,\dots,S_m\rangle$ be the
group generated by $T$ and $\{S_i\}_{i=1}^m$,
and define
\begin{equation}\label{eq:recurrentset}
\mathcal R=
\bigcap_{g\in G}
g\left(\mathcal R_0\cap\mathcal R_1
\cap\dots\cap\mathcal R_m\right).
\end{equation}
i.e. the set of points which are recurrent for $T$, are recurrent for  $S_i$ with $i=1,\dots,m$, and which, moreover, is invariant under any word in $G$.
\begin{lem}\label{lem:recurrentset}
    Suppose that $T_\varepsilon:N\to N$ and  $\{S_{i,\varepsilon}\}_{i=1}^M:N\to N$ preserve a common smooth  measure on $N$. Then the  set $\mathcal R\subset N$  is of full measure.
\end{lem}
\begin{proof}
    It is a straightforward consequence of the fact that all the maps preserve a common smooth measure. Indeed, $\mathcal R$ is a countable union of full measure subsets. Hence, it is of full measure.
\end{proof}

The following is the main result of this section. It will allow us to establish $\varepsilon$-reachability for the skew-product dynamics.

\begin{prop}\label{prop:epsaccesibility}
    Let $T_\varepsilon:N\to N$ be a diffeomorphism which preserves a smooth measure and let  $\{S_{i,\varepsilon}\}_{i=1}^m:N\to N$ be diffeomorphisms preserving the same measure  and satisfying \ref{it:B6}. Let $\mathcal R_\varepsilon$ be the recurrent set in \eqref{eq:recurrentset}.  Then, there exists $K>0$ and $\varepsilon_0>0$ such that for any $0<\varepsilon\leq \varepsilon_0$ and any two points $z,z_*\in N$ with $z\in \mathcal R$ there exists $L\in \mathbb N$,   $\iota\in\{1,\dots, m\}^L$ such that for any sequence $\{n_r\}_{r=1}^{L-1}$ satisfying $ n_1\gg \cdots \gg n_{L-1}\gg 1$ there exists  $\omega\in \mathbb N^L$ satisfying 
    \begin{equation}\label{eq:increasingtimes}
\omega_{r-1}-\omega_{r}\geq n_r \qquad\qquad \text{for all }r=1,\dots, L-1\quad\text{and } \omega_{L-1}\geq n_*
    \end{equation}
    and such that 
    \[
    \mathrm{dist}(T^{\omega_{L-1}}\circ S_{\iota_{L-1}}\circ\cdots\circ T^{ \omega_0}\circ S_{\iota_0}(z),\ z_*)\leq K\varepsilon.
    \]
\end{prop}
\medskip

The observation that $\omega$ can be chosen satisfying \eqref{eq:increasingtimes} will prove crucial in Section \ref{sec:skewprod} to implement  a shadowing argument and substitute the IFS dynamics by the true skew-product dynamics.
\begin{proof}
  Throughout the proof we omit the dependence on $\varepsilon$.
By Proposition \ref{prop:Chow}, there exist $L_0\in\mathbb N$,
$\iota^{(0)}\in\{1,\dots,m\}^{L_0}$ and
$\varsigma\in\{+1,-1\}^{L_0}$ such that, setting
\[
z_0=z,\qquad
z_r=S_{\iota^{(0)}_{r-1}}^{\varsigma_{r-1}}\circ\cdots\circ
S_{\iota^{(0)}_0}^{\varsigma_0}(z),
\qquad r=1,\dots,L_0,
\]
we have $\mathrm{dist}(z_{L_0},z_*)\leq K_0\varepsilon$,
where $K_0$ is independent of $\varepsilon,z,z_*$. We want to now show that the very same conclusion can be obtained using only words compatible with the skew-product dynamics.
\medskip

First we replace this signed word by a word containing only positive
powers of the maps $S_i$. Let
\[
\mathcal C(z)=\overline{\left\{
S_{a_{\ell-1}}\circ\cdots\circ S_{a_0}(z):
\ell\geq0,\ (a_0,\dots,a_{\ell-1})\in\{1,\dots,m\}^{\ell}
\right\}},
\]
where $\ell=0$ denotes the empty word. The set $\mathcal C(z)$ is
forward invariant under every $S_i$. We claim that $z_r\in\mathcal C(z)$
for all $r=0,\dots,L_0$. Indeed, this is clear for $r=0$, and the
inductive step is immediate when $\varsigma_r=+1$. If
$\varsigma_r=-1$, recurrence of $z_r$ for $S_{\iota^{(0)}_r}$ gives
$k_j\to\infty$ such that
\[
S_{\iota^{(0)}_r}^{k_j}(z_r)\longrightarrow z_r.
\]
Hence, by forward invariance and closedness of $\mathcal C(z)$,
\[
z_{r+1}=S_{\iota^{(0)}_r}^{-1}(z_r)
=\lim_{j\to\infty}S_{\iota^{(0)}_r}^{k_j-1}(z_r)
\in\mathcal C(z).
\]
Here we used the invariance of $\mathcal R$ under (all) the scattering maps,
so that each $z_r\in \mathcal R$.
Consequently, there exist $L\geq1$ and
$\iota\in\{1,\dots,m\}^L$ such that, writing
\[
y_0=z,\qquad y_{r+1}=S_{\iota_r}(y_r),
\quad r=0,\dots,L-1,
\]
we have
\begin{equation}\label{eq:positivewordtransportIFS}
\mathrm{dist}(y_L,z_*)\leq(K_0+1)\varepsilon.
\end{equation}
(If an empty word suffices, recurrence of $S_1$ at $z$ allows us
to choose a nonempty one.) Fix $L$ and $\iota$ from now on.
In particular, they do not depend on the prescribed gaps.
\medskip

Let $n_1,\dots,n_{L-1}\in\mathbb N$ be arbitrary. We select
$\omega_{L-1},\dots,\omega_0$ in this order. Set $P_L=\mathrm{id}$.
Suppose that $\omega_{r+1},\dots,\omega_{L-1}$ have been selected,
and put
\[
P_{r+1}=
T^{\omega_{L-1}}\circ S_{\iota_{L-1}}\circ\cdots\circ
T^{\omega_{r+1}}\circ S_{\iota_{r+1}},
\]
with the convention $P_L=\mathrm{id}$. Since $y_{r+1}\in\mathcal R$
is recurrent for $T$ and $P_{r+1}$ is continuous, there are
arbitrarily large $\omega_r\in\mathbb N$ such that
\begin{equation}\label{eq:recurrenceinsertiontransportIFS}
\mathrm{dist}\big(
P_{r+1}(T^{\omega_r}(y_{r+1})),
P_{r+1}(y_{r+1})
\big)<\frac{\varepsilon}{2^{r+1}}.
\end{equation}
Choose such an $\omega_r$ satisfying also
$\omega_{L-1}\geq n_*$ if $r=L-1$, and
$\omega_r-\omega_{r+1}\geq n_{r+1}$ otherwise. Define
\[
P_r=P_{r+1}\circ T^{\omega_r}\circ S_{\iota_r}.
\]
As $S_{\iota_r}(y_r)=y_{r+1}$, inequality
\eqref{eq:recurrenceinsertiontransportIFS} implies
\[
\mathrm{dist}\big(P_r(y_r),P_{r+1}(y_{r+1})\big)
<\frac{\varepsilon}{2^{r+1}}.
\]
Summing these inequalities, we get
\[
\mathrm{dist}\big(P_0(z),y_L\big)
\leq\sum_{r=0}^{L-1}
\mathrm{dist}\big(P_r(y_r),P_{r+1}(y_{r+1})\big)
<\varepsilon.
\]
Together with \eqref{eq:positivewordtransportIFS}, this gives
\[
\mathrm{dist}\big(
T^{\omega_{L-1}}\circ S_{\iota_{L-1}}\circ\cdots\circ
T^{\omega_0}\circ S_{\iota_0}(z),z_*
\big)\leq(K_0+2)\varepsilon.
\]
Taking $K=K_0+2$ completes the proof.
\end{proof}

\medskip

\section{Robustly transitive skew-products over the shift}\label{sec:skewprod}

In this section we prove Proposition \ref{prop:skewprodmain}. The proof builds on the results obtained in Sections \ref{sec:blenderIFS}-\ref{sec:transporIFS}. The fact that our shadowing argument is two-sided introduces some technicalities (not present in the one sided case). We  divide the proof in a number of steps. 

\begin{enumerate}
    \item In Section \ref{sec:compactification} we introduce an extension of the skew-product dynamics to a compactified space. The reason is that to shadow the orbits of the IFSs which were constructed in Sections \ref{sec:blenderIFS}-\ref{sec:transporIFS} we will need to approach arbitrarily close to the invariant sets $W^{u,s}(\Lambda)$. The purpose of this compactification  is to give an extension of the center dynamics for orbits in $W^{u,s}(\Lambda)$. 

    \item In Section \ref{sec:transportskewprod} we show that the skew-product dynamics also displays the $\varepsilon$-reachability property established in Section \ref{sec:transporIFS} for the IFS (see Lemma \ref{lem:transportskewprod}). In order to establish this result we will find convenient to make use of the compactification in Section \ref{sec:compactification}. 

    \item In Section \ref{sec:blenderskewprod} we make use of the results in Section \ref{sec:blenderIFS} (i.e. the existence of a symbolic blender for the IFS) to conclude the proof of Proposition \ref{prop:skewprodmain}.
\end{enumerate}

\subsection{Compactification}\label{sec:compactification}
In  Proposition \ref{prop:nhimexpanded} we showed that the dynamics on the lamination $\mathcal X$ in \eqref{eq:lamination} is driven by the skew-product 
  \begin{align*}
    \mathcal F_\varepsilon: (\mathbb N\times\{1,\dots, m\})^\mathbb Z\times N&\to (\mathbb N\times\{1,\dots,m\})^\mathbb Z\times N\\
    (\bs \omega,z)&\mapsto (\sigma(\bs\omega), F_{\bs\omega,\varepsilon}(z))
    \end{align*}
   where 
   \[
   F_{\bs\omega,\varepsilon}(z)=T_\varepsilon^{\omega_0}\circ (S_{\iota_0,\varepsilon}+R_{\bs\omega,\varepsilon})(z).
   \]
   with $T_\varepsilon:N\to N$ being the inner map \eqref{eq:innerdynamicsdefn} and $\{S_{i,\varepsilon}\}_{i=1}^m:N\to N$ being the system of scattering maps in \eqref{eq:scattmapdynamicsdefn}. Following Moser \cite{MR442980} we introduce the compactification of  \[
 \Sigma=(\mathbb N\times\{1,\dots,m\})^\mathbb Z
 \]
 obtained by including also sequences of the form $\bs\omega=(\omega,\iota)$ with 
    \begin{enumerate}
        
         \item (right-finite) $\omega=(\dots\omega_{-l},\dots;\omega_0,\dots,\omega_{r-1},\infty)$ and $\iota=(\dots,-\iota_{-l},\dots;\iota_0,\dots,\iota_r)$ for some $r\in\mathbb N$,
          \item (left-finite) $\omega=(\infty,\omega_{-l+1},\dots;\omega_0,\dots,\omega_{r},\dots)$ and $\iota=(-\iota_{-l},\dots;\iota_0,\dots,\iota_{r},\dots)$ for some $l\in\mathbb N$,
         \item (finite) $\omega=(\infty,\omega_{-l+1},\dots;\omega_0,\dots,\omega_{r-1},\infty)$ and $\iota=(-\iota_{-l},\dots;\iota_0,\dots,\iota_r)$ for some $l,r\in\mathbb N$.
    \end{enumerate}
  We denote the corresponding space by $\Sigma_c$ and introduce as neighbourhoods of an element $\bs\omega^*$ as in (1), the sets of $\bs \omega$ such that 
  \[
  \omega^*_k=\omega_k\qquad\text{for }-n\leq k\leq r,\qquad\qquad \omega_{r+1}\geq n
  \]
  and
  \[
  \iota^*_k=\iota_k\qquad\text{for }-n\leq k\leq r
  \]
  for $n\in\mathbb N$. Similarly we introduce the set of neighbourhoods of an element $\bs\omega^*$ as in (2) or (3). In this way we obtain a topology on $\Sigma_c$ extending that of $\Sigma$ and for which $\Sigma_c$ is compact.

  The key point in introducing this extension is that, if we fix any $\bs\omega^*$ as in (1), for any sequence $\{\bs\omega^{(k)}\}_k\in \Sigma_c$ such that $\bs\omega^{(k)}\to \bs\omega^*$ the sequence of functions
  \[
R_{\bs\omega^{(k)},\varepsilon}:N\to N
  \]
  converges uniformly to a function
  \[
  R_{\bs\omega^*,\varepsilon}:N\to N.
  \]
Similar (suitably modified) statements hold sequences $\bs\omega^*$ of types (2) or (3). The skew-product dynamics is only defined on  $\{(\bs\omega,z)\in \Sigma_c\times N\colon \omega_0<\infty\}$. However, for $\bs\omega^*\in \Sigma$ such that $\omega_0^*=\infty$ we can define a map
\begin{align*}
F_{\bs\omega^*,\varepsilon}(z):N&\to N\\z&\mapsto (S_{\iota_0,\varepsilon}+R_{\bs\omega^*,\varepsilon})(z).
\end{align*}

\subsection{Transport in the skew-product system}\label{sec:transportskewprod}
In this section we show that the orbits constructed in Section \ref{sec:transporIFS} (see in particular Proposition \ref{prop:epsaccesibility}) can be shadowed by the actual skew-product dynamics.
\begin{lem}\label{lem:transportskewprod}
   Consider the skew-product map $\mathcal F_\varepsilon:(\bs\omega,z)\mapsto (\sigma(\bs\omega),F_{\bs\omega,\varepsilon}(z))$ in Proposition \ref{prop:nhimexpanded} and suppose that for every $\varepsilon\geq 0$ small enough the maps  $T_\varepsilon:N\to N$  and $\{S_{i,\varepsilon}\}_{i=1}^m:N\to N$ involved in its definition preserve a smooth measure  and satisfy \ref{it:B6}. There exists $K>0$  and $\varepsilon_0>0$ such that for any $0<\varepsilon\leq \varepsilon_0$, given any pair of  open balls $B, B_*\subset N$, any $\bs\omega\in \Sigma$ and any $n\in\mathbb N$ there exists $L\in \mathbb N$ and $\tilde{\bs\omega}=(\tilde\omega,\tilde\iota)\in (\mathbb N\times\{1,\dots,m\})^L$ such that for any $\bs\omega'=(\omega',\iota')\in\Sigma$ with
    \begin{equation}\label{eq:modifiedomegaprimetransport}
    \omega'=(\dots,\omega_{-1};\omega_0,\dots,\omega_{n-1},\tilde\omega_0,\dots,\tilde\omega_{L-1},\omega'_{n+L},\dots)
    \end{equation}
    and
    \begin{equation}\label{eq:modifiediotaprimetransport}
     \iota'=(\dots,\iota_{-1};\iota_0,\dots,\iota_{n-1},\tilde\iota_0,\dots,\tilde\iota_{L-1},\iota'_{n+L},\dots)
    \end{equation}
we have 
\[
\inf_{z\in B,z_*\in B_*}\mathrm{dist}(F^{n+L}_{\bs\omega',\varepsilon}(z),z_*)\leq K \varepsilon.
\]
\end{lem}

\begin{proof}
Throughout the proof we omit the dependence on $\varepsilon$ from
our notation and use $\lesssim$ as in Remark \ref{rem:lessymsymbol}.
We choose the decay rate $\vartheta$ in Proposition
\ref{prop:nhimexpanded} so that
\[
|DT|_{C^0}\vartheta<1,
\]
Consider the right-finite sequence $\bs\omega^*$ which agrees
with $\bs\omega$ at every index $k<n$ and has
$\omega_n^*=\infty,\iota_n^*=\iota_n$. Set
\[
\widetilde B(\bs\omega^*)=
(S_{\iota_n}+R_{\sigma^n\bs\omega^*})\circ
F_{\bs\omega^*}^{n}(B).
\]
Choose
$\tilde z=\tilde z(\bs\omega^*)\in
\widetilde B(\bs\omega^*)\cap\mathcal R$ and $z\in B$ such that
\[
(S_{\iota_n}+R_{\sigma^n\bs\omega^*})
\circ F_{\bs\omega^*}^{n}(z)=\tilde z.
\]
By the comparison estimates \eqref{eq:comparisonskewprod},
for every $\bs\omega'$ which agrees with $\bs\omega$ at $k<n$,
has $\iota'_n=\iota_n$, and has $\omega'_n=q$, we have
\begin{equation}\label{eq:keycontroltransportproof}
\mathrm{dist}\left(
(S_{\iota_n}+R_{\sigma^n\bs\omega'})\circ
F_{\bs\omega'}^n(z),\,\tilde z\right)
\leq C_n\vartheta^q.
\end{equation}
for some $C_n$ depends on the fixed initial itinerary, but not on
$q$ or on the continuation of $\bs\omega'$.
\medskip

By Proposition
\ref{prop:epsaccesibility}, there exist $\ell\geq1$ and
$\hat\iota=(\hat\iota_0,\dots,\hat\iota_{\ell-1})$ such that,
setting $y_0=\tilde z$ and $y_{r+1}=S_{\hat\iota_r}(y_r)$,
we have
\begin{equation}\label{eq:closenessproof}
\mathrm{dist}(y_\ell,z_*)\lesssim\varepsilon
\end{equation}
for some fixed $z_*\in B_*$. All the points $y_r$ belong to
$\mathcal R$. We now choose $\hat\omega_{\ell-1},\dots,\hat\omega_0$ as follows. Let
\[
P_\ell=\mathrm{id},\qquad
P_{r+1}=T^{\hat\omega_{\ell-1}}\circ S_{\hat\iota_{\ell-1}}
\circ\cdots\circ T^{\hat\omega_{r+1}}\circ S_{\hat\iota_{r+1}}
\]
(with the empty composition when $r=\ell-1$), and set
\[
K_{r+1}=\max\{1,|P_{r+1}|_{C^1}\},
\qquad K_\ell=1.
\]
Once the times with indices greater than $r$ have been fixed,
recurrence of $y_{r+1}$ for $T$ allows us to choose
$\hat\omega_r$ arbitrarily large so that
\begin{equation}\label{eq:recurrenceinsertionskewprod}
\mathrm{dist}\big(P_{r+1}T^{\hat\omega_r}(y_{r+1}),
P_{r+1}(y_{r+1})\big)<\varepsilon2^{-r-1},
\qquad
C K_{r+1}(|DT|_{C^0}\bar\lambda)^{\hat\omega_r}<\frac{\varepsilon}{\ell+1}.
\end{equation}
We also require
$\hat\omega_{r}-\hat\omega_{r+1}\geq1$ if $r<\ell-1$
and $\hat\omega_{\ell-1}\geq n_*$. Put
$P_r=P_{r+1}\circ T^{\hat\omega_r}\circ S_{\hat\iota_r}$.
As in the proof of Proposition \ref{prop:epsaccesibility},
summing the recurrence errors gives
\begin{equation}\label{eq:idealtransportskewprod}
\mathrm{dist}\big(P_0(\tilde z),z_*\big)\lesssim\varepsilon.
\end{equation}
Fix these times and put
\[
K_0=\max\{1,|P_0|_{C^1}\}.
\] By recurrence of $\tilde z$ for $T$,
we can choose $\tilde\omega_0>\hat\omega_0$ so large that
\begin{equation}\label{eq:firstinsertiontransportskewprod}
K_0\,\mathrm{dist}\big(T^{\tilde\omega_0}(\tilde z),
\tilde z\big)<\varepsilon,
\qquad C_nK_0(|DT|_{C^0}\vartheta)^{\tilde\omega_0}<\varepsilon.
\end{equation}
Define
\[
\tilde\omega=(\tilde\omega_0,\hat\omega_0,\dots,
\hat\omega_{\ell-1}),\qquad
\tilde\iota=(\iota_n,\hat\iota_0,\dots,
\hat\iota_{\ell-1}),\qquad L=\ell+1.
\]
We claim that this word has the desired property. Indeed, let
$\bs\omega'$ satisfy \eqref{eq:modifiedomegaprimetransport}--
\eqref{eq:modifiediotaprimetransport}. For $0\leq r<\ell$ set
\[
G_r=T^{\hat\omega_r}\circ S_{\hat\iota_r},\qquad
E_r=F_{\sigma^{n+r+1}\bs\omega'}.
\]
Since $\tilde\omega_0>\hat\omega_0>\cdots>
\hat\omega_{\ell-1}$, the estimate in
\eqref{eq:expressioncenterdynskewprod} gives, uniformly in the
unspecified continuation of $\bs\omega'$,
\[
\sup_{x\in N}\mathrm{dist}(E_r(x),G_r(x))
\leq C |DT|_{C^0}^{\hat\omega_r}\vartheta^{\hat\omega_r}
=C(|DT|_{C^0}\vartheta)^{\hat\omega_r}.
\]
Using the bounds $K_{r+1}$ on the Lipschitz constants of the
remaining ideal compositions and \eqref{eq:recurrenceinsertionskewprod},
a telescoping argument yields
\begin{equation}\label{eq:shadowingtransportskewprod}
\sup_{x\in N}\mathrm{dist}\big(
E_{\ell-1}\circ\cdots\circ E_0(x),P_0(x)\big)
\leq\sum_{r=0}^{\ell-1}CK_{r+1}(|DT|_{C^0}\vartheta)^{\hat\omega_r}
<\varepsilon.
\end{equation}
Finally, set
\[
x'=(S_{\iota_n}+R_{\sigma^n\bs\omega'})
\circ F_{\bs\omega'}^n(z).
\]
By \eqref{eq:keycontroltransportproof} and
\eqref{eq:firstinsertiontransportskewprod},
\[
\mathrm{dist}\big(P_0T^{\tilde\omega_0}(x'),
P_0T^{\tilde\omega_0}(\tilde z)\big)
\leq C_n K_0(|DT|_{C^0}\vartheta)^{\tilde\omega_0}<\varepsilon.
\]
As
\[
F_{\bs\omega'}^{n+L}(z)
=E_{\ell-1}\circ\cdots\circ E_0
\circ T^{\tilde\omega_0}(x'),
\]
combining this with \eqref{eq:shadowingtransportskewprod},
\eqref{eq:firstinsertiontransportskewprod}, and
\eqref{eq:idealtransportskewprod}, we conclude that
\[
\mathrm{dist}\big(F_{\bs\omega'}^{n+L}(z),z_*\big)
\lesssim\varepsilon,
\]
This proves the lemma.
\end{proof}

\subsection{Blender dynamics in the skew-product system}\label{sec:blenderskewprod}

We now translate the result in Theorem \ref{thm:mainIFS} (local topological mixing of the IFS) to the skew-product setting. We start with the following local result.

\begin{lem}\label{lem:hyperbolic points in blender}
Consider the skew-product map $\mathcal F_\varepsilon:(\bs\omega,z)\mapsto (\sigma(\bs\omega),F_{\bs\omega,\varepsilon}(z))$ in Proposition \ref{prop:nhimexpanded} and suppose that for every $\varepsilon\geq 0$ small enough the maps  $T_\varepsilon:N\to N$  and $\{S_{i,\varepsilon}\}_{i=1}^m:N\to N$ involved in its definition  satisfy the following. There exists a subset $\mathcal A\subset N$ diffeomorphic to the $2d$-dimensional annulus $\mathbb A^d=\mathbb T^d\times[-1,1]^d$ (here $2d=\mathrm{dim}(N)$) on which the maps $\{T_\varepsilon, S_{1,\varepsilon}\}$ are of the form \eqref{eq:innermapsymbolic} and \eqref{eq:scattmapsymbolic} and verify hypothesis \ref{it:itemB1}-\ref{it:itemB5} with $\gamma=\gamma(F_0)$.  Let 
\[
\mathbb A^d_{\sqrt\varepsilon}=\mathbb T^d\times [-\sqrt\varepsilon,\sqrt\varepsilon]^d \subset \mathbb A^d
\]
     Then, there exists $\bs\xi=(\xi,a)\in \mathbb N\times\{1,\dots,m\}$ and $\bs\xi'=(\xi',a')\in  \mathbb N\times\{1,\dots,m\}$ such that for 
   \[
   \bs\omega=(\dots,\bs\xi,\dots;\bs\xi,\dots,\bs\xi,\dots),\qquad\qquad \bs\omega'=(\dots,\bs\xi',\dots;\bs\xi',\dots,\bs\xi',\dots)
   \]
   there exists points
   \[
   P\in N_{\bs\omega}, \qquad\qquad Q\in N_{\bs\omega'}
   \]
   such that:
   \begin{enumerate}
       \item $P,P'$ are hyperbolic fixed points for $\mathcal F_\varepsilon$, 
       \item $W^u(P;\mathcal F_\varepsilon)\pitchfork W^s(P';\mathcal F_\varepsilon)\neq 0$.
       \item We have
       \[
       \{\bs\omega\}\times \mathbb A^d_{\sqrt\varepsilon}\subset \overline{W^s(P;\mathcal F_\varepsilon)}\qquad\qquad \{\bs\omega'\}\times \mathbb A^d_{\sqrt\varepsilon}\subset \overline{W^u(P';\mathcal F_\varepsilon)}
       \]
   \end{enumerate}
   Moreover, given any $c\in \mathbb N$,  $\bs\omega_0,\bs\omega_0'$ can be chosen so that 
   \begin{equation}\label{eq: gcd}
   \mathrm{gcd}(c+\omega_0,c+\omega_0')=1.
   \end{equation}
\end{lem}
The last observation will be important for the proof of Theorem \ref{thm:corollaryMain}.
\begin{proof}
     The proof of this result is an straightforward consequence of the construction in Section \ref{sec:blenderIFS}. Indeed, for a constant sequence $\bs\omega$ the map $\mathcal F_{\bs\omega}:N\to N$ is of the form \eqref{eq:form Fn}. Hence, the existence of $\bs\omega_0,\bs\omega_0'$ and $P,Q$ as in the statement follows from Proposition \ref{prop:symbolicdoubleblender}. The fact that $\bs\omega_0,\bs\omega_0'$ can be chosen to satisfy  \eqref{eq: gcd} can be deduced easily making use of the strong irrationality of the rotation vector $\beta$\footnote{ Indeed, take any two points $x_1,x_2\in\mathbb T$, $\beta\in\mathcal B_\gamma$ and any $\delta>0$. We want to show that there exists $L\in\mathbb N$ large enough (depending only on $\delta,\gamma$) such that for any $\varepsilon>0$ small enough we can find $n_1,n_2\in \{[1/\varepsilon],\dots,[1/\varepsilon]+L\}$ for which $\mathrm{gcd}(n_1,n_2)=1$ and, moreover, $|x_i-[n_i\beta]|\leq \delta$ for $i=1,2$. To do so, we proceed as follows. On one hand, for $L_0\in\mathbb N$ large enough (depending only on $\delta,\gamma$) we  select $n_0\in \{1,\dots,L_0\}$ such that $|(x_2-x_1)-[n_0\beta]|\leq \delta/2$. On the other hand it is possible to find $L_0\in \mathbb N$ (depending only on $\delta,\gamma$)  and $k\in \{[1/n_0\varepsilon]+1,\dots,[1/n_0\varepsilon]+L_0+1\}$ for which $n_1=1+k n_0$ satisfies $|x_1-[n_1\beta]|\leq \delta/2$. Then $n_2=n_1+n_0$ satisfies the desired conclusion.
    }.
\end{proof}

\subsection{Proof of Proposition \ref{prop:skewprodmain}}

Combining the argument in the proof of Lemmas \ref{lem:transportskewprod}-\ref{lem:hyperbolic points in blender} together those in Section \ref{sec:blenderIFS} we obtain the following result, which is equivalent to Proposition \ref{prop:skewprodmain}.

\begin{prop}
    Consider the skew-product map $\mathcal F_\varepsilon:(\bs\omega,z)\mapsto (\sigma(\bs\omega),F_{\bs\omega,\varepsilon}(z))$ in Proposition \ref{prop:nhimexpanded} and let   $T_\varepsilon:N\to N$  and $\{S_{i,\varepsilon}\}_{i=1}^m:N\to N$ be the maps involved in its definition. Suppose that there exists $\varepsilon_0>0$ such that for all $0<\varepsilon\leq \varepsilon_0$:
\begin{enumerate}
\item The maps $T_\varepsilon$ and $\{S_i\}^m_{i=1}$ preserve a common smooth  measure on $N$,
    \item There exists a subset $\mathcal A\subset N$ diffeomorphic to the $2d$-dimensional annulus $\mathbb A^d=\mathbb T^d\times[-1,1]^d$ (here $2d=\mathrm{dim}(N)$) on which the maps $\{T_\varepsilon, S_{1,\varepsilon}\}$ are of the form \eqref{eq:innermapsymbolic} and \eqref{eq:scattmapsymbolic} and verify hypothesis \ref{it:itemB1}-\ref{it:itemB5} with $\gamma=\gamma(F_0)$.   \item The scattering maps $\{S_{i,\varepsilon}\}_{i=1}^m$ satisfy assumption \ref{it:B6}. 
    \end{enumerate}
    Let $P,P'$ be as in Lemma \ref{lem:hyperbolic points in blender}. Then, 
    \[
    \overline{W^s(P;\mathcal F_\varepsilon)}=\Sigma\times N =\overline{W^u(P';\mathcal F_\varepsilon)}
    \]
    \end{prop}
    
    \begin{proof}
We only establish the stable case, the unstable case being analogous. Let 
\[
\bs\omega=(\dots,\bs\xi,\dots;\bs\xi,\dots,\bs\xi,\dots)
\]
be as in Lemma \ref{lem:hyperbolic points in blender}. We need to show that given any open ball $B\subset N$ any $\bs{\tilde\omega}\in \Sigma$ and any $n\in \mathbb N$ there exists $L\in \mathbb N$ and   $\hat{\bs\omega}\in (\mathbb N\times\{1,\dots,m\})^L$ such that for any $\bs{\breve\omega}=(\omega',\iota')\in \Sigma$ with\begin{equation}\label{eq:modifiedomegaprimetransport}
    \omega'=(\dots,\omega_{-1};\omega_0,\dots,\omega_{n-1},\hat\omega_0,\dots,\hat\omega_{L-1},\omega'_{n+L},\xi,\dots,\xi,\dots)
    \end{equation}
    and
    \begin{equation}\label{eq:modifiediotaprimetransport}
      \iota'=(\dots,\iota_{-1};\iota_0,\dots,\iota_{n-1},\hat\iota_0,\dots,\hat\iota_{L-1},\iota'_{n+L},a,\dots,a,\dots)
    \end{equation}
we have 
\[
F^{n+L}_{\bs\omega',\varepsilon}(B)\cap W^s(P;\mathcal F_\varepsilon)\neq\emptyset.
\]
However, the existence of such $L$ and $\hat{\bs\omega}$ follows plainly from the arguments in Section \ref{sec:blenderIFS}-\ref{sec:transporIFS} combined with that in the proof of Lemma \ref{lem:transportskewprod}, which allows us to bypass non-stationarity (i.e. the fact that the skew-product is not locally constant).
\end{proof}

\section{Proof of Theorem \ref{thm:corollaryMain}}\label{sec:proof top mix symplectomorphisms}

Fix any $\theta\in \mathfrak B^\omega_{r_0}(M\times N)$ for which Theorem \ref{thm:mainham} holds.  We claim that we can find a pair of hyperbolic periodic orbits $P,Q\subset \mathcal X_\varepsilon$ for which:
 \begin{enumerate}
     \item their periods $p,q\in \mathbb N$ (for the map $F_{\varepsilon,\theta}$) are coprime, i.e.
     \[
     \mathrm{gcd}(p,p')=1.
     \]
     \item their stable/unstable manifolds intersect transversally, i.e.  $W^u(P)\pitchfork W^s(P')\neq 0$
     \item their stable/unstable manifolds are dense in the whole manifold, i.e. $\overline{W^s(P)}=M\times N=\overline{W^u(P')}$.
 \end{enumerate}
From here, topological mixing of the map $F_{\varepsilon,\theta}$ is straightforward. Indeed, a straighforward consequence of the $\lambda$-lemma (see \cite{PalisMelo}) shows that given any open sets $U,V\in M\times N$ there exists $n_*\in \mathbb N$ such that for any 
\begin{equation}\label{eq:large n decomposition}
n=n_*+a p+b p',\qquad\qquad a\in \mathbb N,\ b\in \mathbb N
\end{equation}
we must have $F_{\varepsilon,\theta}^n(U)\cap B\neq \emptyset$. The conclusion follows from the fact that since $p,p'$ are coprime, any $N\in\mathbb N$ sufficiently large is of the form \eqref{eq:large n decomposition}.

\medskip

 We now verify our claim. To do so we observe that it follows from Lemma \ref{lem:hyperbolic points in blender} that we can find two hyperbolic fixed points $P,P'\in \mathcal X_\varepsilon$ of the map $\mathcal R_\varepsilon|_{\mathcal X_\varepsilon}:\mathcal X_\varepsilon\to \mathcal X_\varepsilon$ for which 
 \[
 W^u(P;\mathcal R_\varepsilon|_{\mathcal X_\varepsilon})\pitchfork W^s(P';\mathcal R_\varepsilon|_{\mathcal X_\varepsilon})\neq 0
 \]
 and 
 \[
 \overline{W^s(P;\mathcal R_\varepsilon|_{\mathcal X_\varepsilon})}=\mathcal X_\varepsilon=\overline {W^u(P';\mathcal R_\varepsilon|_{\mathcal X_\varepsilon})}.
 \]
Moreover, it follows from \eqref{eq: gcd} that for the original map $F_{\varepsilon, \theta}$ the points $P,Q$ are hyperbolic periodic points and can be chosen with coprime periods. Clearly their stable/unstable manifolds (for the map $F_{\varepsilon,\theta}$) intersect transversally within $M\times N$.  

In order to deduce the third item it is enough to show that for any pair of open sets $U,V\in M\times N$ we have $W^{u}(\mathcal X_\varepsilon)\cap U\neq \emptyset$ and $W^{s}(\mathcal X_\varepsilon)\cap V\neq \emptyset$. However, the manifold $W^{u,s}(\mathcal X_0)$ is dense in $M\times N$ (since in the unperturbed case $\varepsilon=0$, the dynamics is a direct product in which the first factor is, by assumption, a transitive Anosov map). By the classical perturbation theory for normally hyperbolic submanifolds (see \cite{MR501173}), for any $\theta\in \mathfrak B_{r_0}^\omega(M\times N)$  and $0\leq \varepsilon\leq \varepsilon_0$, compact pieces of $W^{u,s}(\mathcal X_\varepsilon)$ are $O(\varepsilon)$-close to $W^{u,s}(\mathcal X_0)$. In particular, there exists $L>0$ independent of $\varepsilon$ and $\theta$ such that $W^u(\mathcal X_\varepsilon)$ (resp.  $W^s(\mathcal X_\varepsilon)$) intersects any disk tangent to $E^s$ (resp. $E^u$) and which has radius larger than $L\varepsilon$. Hence, our claim follows from the fact that any disk tangent to $E^s$ (resp. $E^u$) grows under backwards (forward) iteration and that $W^u(\mathcal X_\varepsilon)$ (resp.  $W^s(\mathcal X_\varepsilon)$) is invariant.
\medskip

\section{Genericity of assumptions}\label{sec:genericity}

Throughout this section we we write $2d=\mathrm{dim}(N)$, assume that $r_0=r_0(d)<\infty$ and $\kappa_0=\kappa_0(d)<\infty$ are positive constants which we may take as large as we need, and give a proof of Proposition \ref{prop:genericityofassumptions}.

\subsection{Standing assumptions}
We start by recalling some basic consequences of the assumptions \ref{it:Assumption1}-\ref{it:Assumption2} in Proposition \ref{prop:genericityofassumptions}. Consider a symplectomorphism  
\begin{align*}
F_0:M\times N&\to M\times N\\
(x_M,x_N)&\mapsto (F_{M}(x_M),F_N(x_N))
\end{align*}
satisfying \ref{it:Assumption1}-\ref{it:Assumption2}. On one hand, by assumption \ref{it:Assumption2}, the map  $F_M:M\to M$ has a hyperbolic fixed point $a\in M$ such that, for $F_0$, the submanifold \begin{equation}\label{eq:unperturbednhim}
  \Lambda_0=\{a\}\times N
  \end{equation}
is a real-analytic $\kappa$-normally hyperbolic invariant manifold  with $\kappa>\kappa_0$ . Obviously, the dynamics restricted to this invariant manifold is given by the map 
\begin{equation}\label{eq:unperturbedinnerdynam}
T_0=F_N:N\to N
\end{equation}
which, by assumption \ref{it:Assumption1}, displays a non-degenerate elliptic equilibrium.

On the other hand, assumption \ref{it:Assumption2} also guarantees that the hyperbolic fixed point $a$ of the map $F_M$ displays a transverse homoclinic point. In particular, for any $m\in\mathbb N$ we may assume that there exist $m$ different transverse homoclinic points $\{\gamma_i\}_{i=1}^m\in W^{u}(a)\pitchfork W^s(a)$ such that 
    \[
    \left(\bigcup_{n\in\mathbb Z}F_M^n(\gamma_i)\right)\cap \left(\bigcup_{n\in\mathbb Z}F_M^n(\gamma_j)\right)=\emptyset
    \]
for every pair $i\neq j$. In particular, for $i=1,\dots m$ 
\begin{equation}\label{eq:unperturbedchannels}
\Gamma^i_0=\{\gamma_i\}\times N\subset W^{s}(\Lambda_0)\pitchfork W^u(\Lambda_0)
\end{equation}
are disjoint full homoclinic manifolds as in Assumption \ref{it:H2}. From the construction, since the map $F_0$ is a direct product, the corresponding scattering maps $\{S_{i,0}\}_i:N\to N$ are both given by the identity map.
\medskip

Now, given a perturbation $f\in \mathcal B^\omega_{r_0}(M\times N)$ let 
 \begin{equation}\label{eq:unfoldingproof}
 F_{\varepsilon,f}:M\times N\to M\times N
 \end{equation}
 be the corresponding one-parameter family of symplectic maps unfolding $F_{0}$ according to \eqref{eq:unfoldingforsymplectic}.

 By regular perturbation theory (see \cite{MR501173} or Section 4.2 in \cite{DelaLLaveScattmap}, in particular Theorem 23) there exists $\varepsilon_0(F_0)$ such that, for any $f\in \mathcal B^\omega_{r_0}(M\times N)$ and for any $0\leq \varepsilon\leq \varepsilon_0(F_0)$:
\begin{enumerate}
    \item \label{it:P1} the submanifold \eqref{eq:unperturbednhim} admits a unique continuation $\Lambda_{\varepsilon}(f)$ as a $C^{\kappa+1}$ normally hyperbolic invariant submanifold for $F_{\varepsilon,f}$ which satisfies Assumption \ref{it:H1} with $\kappa\geq \kappa_0$. Moreover, $\Lambda_{\varepsilon}(f)$ can be parametrized by means of a $C^{\kappa+1}$ embedding 
    \begin{equation}\label{eq:embeddinggenericity}
    \psi_{\varepsilon,f}:N\to \Lambda_{\varepsilon}(f)\subset M\times N
    \end{equation}
    which depends smoothly on $\varepsilon$ and satisfies $\psi_{0,f}:z\in N\to (a,z)\in \Lambda_0$. We denote by 
    \begin{equation}\label{eq:innermapproof}
    T_{\varepsilon}(f)=\psi_{\varepsilon,f}^{-1}\circ F_{\varepsilon,f}|_{\Lambda_{\varepsilon}(f)} \circ\psi_ {\varepsilon,f}:N\to N
    \end{equation}
    the corresponding inner map and note that it is a symplectomorphism with respect to the induced symplectic form on $N$ (see Theorem 26 in \cite{DelaLLaveScattmap}). Moreover, the map \eqref{eq:innermapproof} displays a non-degenerate elliptic equilibrium.

    \item \label{it:P2} the  submanifolds  \eqref{eq:unperturbedchannels} admit a unique continuation  $\{\Gamma_{\alpha,\varepsilon}^i\}_i$ as $C^{\kappa}$  disjoint full homoclinic manifolds to $\Lambda_{\varepsilon}(f)$ as in \ref{it:H2}. This follows from the fact that $W^{u,s}(\Lambda_{\varepsilon}(f))$ are $C^\kappa$ close to $W^{u,s}(\Lambda_0)$. We denote by $\{S_{\Gamma_{\varepsilon}^i(f)}\}_i:\Lambda_{\varepsilon}(f)\to \Lambda_{\varepsilon}(f)$ the corresponding scattering maps and write 
    \begin{equation}\label{eq:scattmapproof}
    S_{i,\varepsilon}(f)=\psi_{\varepsilon,f}^{-1}\circ S_{\Gamma_{\varepsilon}^i(f)}\circ\psi_{\varepsilon,f}: N\to N
    \end{equation}
    which, by Theorem 8 in \cite{DelaLLaveScattmap}, are also symplectic with respect to the induced symplectic form on $N$. 
\end{enumerate}
In particular, for any $f\in \mathcal B^\omega_{r_0}(M\times N)$ the one-parametric family of real-analytic diffeomorphisms \eqref{eq:unfoldingproof} satisfies assumptions \ref{it:H1}-\ref{it:H2} for $0\leq \varepsilon\leq \varepsilon_0(F_0)$. Consequently, the first item in Proposition \ref{prop:genericityofassumptions} holds. Moreover, these maps preserve a common smooth  measure (as they preserve a smooth symplectic structure). To conclude the proof of Proposition \ref{prop:genericityofassumptions} we now show that, for $f$ in an open and dense subset of $\mathcal B^\omega_{r_0}(M\times N)$ the maps \eqref{eq:innermapproof} and \eqref{eq:scattmapproof} satisfy assumptions \ref{it:itemB1}-\ref{it:B6}.

\subsection{Inner dynamics}
Our first observation is the following.

\begin{lem}\label{lem:Weinsteinlemma}
  Given $\gamma>0$ let $\mathcal B_\gamma\subset \mathbb R^d$ be the set of constant type vectors  in \eqref{eq:constanttype} There exists $\gamma>0$  such that the symplectic  map $T_0:N\to N$ in \eqref{eq:unperturbedinnerdynam}  displays a real-analytic Lagrangian invariant torus $\mathcal T\subset N$ on which the dynamics is conjugated to a rigid rotation with vector $\beta\in \mathcal B_\gamma$. Moreover, there exists a neighbourhood $\mathcal A$ of $\mathcal T$ and a real-analytic  map 
$\psi_{\mathrm{pol}}:(\varphi,J)\in\mathbb A^d\to z\in \mathcal A $ which   transforms $\psi^*_{\mathrm{pol}}\varpi_N=\mathrm dJ\wedge\mathrm d\varphi$ and (locally) conjugates $T_0$ to 
\[
\widetilde T_0:=\psi_\mathrm{pol}^{-1}\circ T_0\circ\psi_\mathrm{pol}:\binom{\varphi}{J}\mapsto \binom{\varphi+\beta +R_{0,\varphi}(\varphi,J)}{J+R_{0,J}(\varphi,J)}
\]
with $R_{0,\varphi}(\varphi,0)=0$ and  $\partial^n_JR_{0,J}(\varphi,0)=0$ for $n=0,1$.
\end{lem}

\begin{proof}
    By assumption, the map $T_0$ displays a non-degenerate elliptic equilibrium $z_0\in N$. After introducing symplectic polar coordinates on a neighbourhood of $z_0$ the proof follows by the classical KAM theorem  (see for instance \cite{MR442980,PoschelKAM}).
\end{proof}

\begin{prop}\label{prop:innerdynamicsassumption}
   Fix $l>0$ and assume that $\kappa_0$ and $r_0$ are sufficently large.   Given any $f\in \mathcal B^{\omega}_{r_0}(M\times N)$, let $T_{\varepsilon}(f):N\to N$ be the symplectic map defined in \eqref{eq:innermapproof}. Then, there exists $\varepsilon_0(F_0)>0$ such that for any  $0\leq \varepsilon\leq \varepsilon_0$ there exists a coordinate chart  $\psi_{\mathrm{BNF}}:(\varphi,J)\in\mathbb A^d\mapsto z\in N$ which transforms $\psi_{\mathrm{BNF}}^*\varpi_N=\mathrm dJ\wedge\mathrm d\varphi$ and  such that the map 
    \begin{equation}\label{eq:KAMmap}
        \widetilde T_{\varepsilon}(f)=\psi_{\mathrm{BNF}}^{-1} \circ T_{\varepsilon}(f) \circ  \psi_{\mathrm{BNF}}
    \end{equation}
    is $C^l$, is of the form \begin{equation}\label{eq:BNFform}
    \widetilde T_\varepsilon(f):\binom{\varphi}{J}\mapsto \binom{\varphi+\beta+A_\varepsilon(f)J+R_{\varepsilon,\varphi}(f)(\varphi,J)}{J+ R_{\varepsilon,J}(f)(\varphi,J)}
    \end{equation}
    with\footnote{We denote by $\mathrm{Symm}(d,\mathbb R)$ the set of symmetric square matrices of dimension $d$ and real entries.} 
    $A_\varepsilon(f)\in GL(d,\mathbb R)\cap\mathrm{Symm}(d,\mathbb R)$, 
    and satisfies assumptions \ref{it:itemB1}-\ref{it:itemB2} and \ref{it:itemB5}. Moreover, the map $\psi_\mathrm{pol}^{-1}\circ\psi_{\mathrm{BNF}}$ preserves the symplectic form $\mathrm dJ\wedge\mathrm d\varphi$ and is $O(\varepsilon)$-close to the identity in the $C^l$ topology.
\end{prop}
\begin{proof}

By Lemma \ref{lem:Weinsteinlemma}, the real-analytic symplectic map $T_0:N\to N$ displays a real-analytic KAM torus $\mathcal T\subset N$ on which the dynamics is conjugated to a Diophantine rotation with vector $\beta\in\mathcal B_\gamma>0$. Since for any $f\in \mathcal B_{r_0}^\omega(M\times N)$ the map $T_{\varepsilon}(f):N\to N$ in \eqref{eq:innermapproof} is $O(\varepsilon)$-close to the map $T_0$ in the $C^{\min\{\kappa+1,r_0\}}$ topology, assuming $\kappa_0,r_0$ are large enough (depending only on $d$ and $F_0$) the main result in  \cite{MR1790659} (see also \cite{MR595866,MR668410})  shows that $\mathcal T$ admits a continuation $\mathcal T_\varepsilon(f)$ as a Lagrangian invariant torus of the map $T_{\varepsilon}(f)$ and, on which, the dynamics is conjugated to a rigid rotation with frequency vector $\beta$. Moreover,  there exists a neighbourhood $\mathcal A$ of $\mathcal T_\varepsilon(f)$ in $N$ and a  symplectomorphism $
\hat\psi:(\varphi,J)\in\mathbb A^d\to z\in \mathcal A $ which   transforms $\hat\psi^*\varpi_N=\mathrm dJ\wedge\mathrm d\varphi$ and conjugates
\[
\widehat T_{\theta,\varepsilon}:=\hat\psi^{-1}\circ T_{\theta,\varepsilon}\circ \hat\psi:\binom{\varphi}{J}\mapsto \binom{\varphi+\beta+A_\varepsilon(\varphi)J+\widehat R_{\varepsilon,\varphi}(\varphi,J)}{J+\widehat R_{\varepsilon,J}(\varphi,J)}
\]
with $A_\varepsilon:\mathbb T^d\to GL(d,\mathbb R)\cap \mathrm {Symm}(d,\mathbb R)$ and $\widehat R_{\varepsilon,\varphi},\widehat R_{\varepsilon,J}$  satisfying $ \partial_{J}^n \widehat R_{\varepsilon,\star}(\varphi,0)=0$ for $n\leq 1$. For any fixed $\tilde l>0$,  we may assume (by increasing $r_0,\kappa_0$ if necessary) that $A_\varepsilon,\widehat R_{\varepsilon,\varphi}$ and $\widehat R_{\varepsilon,\varphi}$ are of class $C^{\tilde l}$.
\medskip

The conclusion now follows after one step of the classical  Birkhoff normal form iteration. We  look for a generating function 
\[
\mathcal G(\varphi,I)=\varphi I+G(\varphi,I)
\] 
such that the associated change of coordinates $\tilde\psi:(\theta,I)\mapsto (\varphi,J)$
\[
\theta=\varphi+\partial_IG(\varphi,I)\qquad\qquad J=I+\partial_\varphi G(\varphi,I),
\]
eliminates the term 
\[
[A_\varepsilon](\varphi)=A_\varepsilon(\varphi)-\int_{\mathbb T^d} A_\varepsilon(\varphi)\mathrm d\varphi
\]
(by symplectic symmetry this transformation will also eliminate the term of order $2$ in $R_{\varepsilon,J}$). To do so, define 
\[
G(\varphi, I)=g(\varphi) I^{2},
\]
with 
\[
g(\varphi)=\frac{1}{2}\sum_{l\in\mathbb Z^d\setminus\{0\}} \frac{{A_\varepsilon}^{[l]}}{1-e^{il\beta}}e^{il\varphi}
\]
Since $\beta\in \mathcal B_\gamma$ a standard computation shows that
\[
 \partial_{I}  G(\varphi+\beta,I)- \partial_{I} G(\varphi, I)= [A_\varepsilon](\varphi,I)
\]
and $
|G|_{C^{\tilde l-2d-1}}\lesssim |A_\varepsilon|_{C^{\tilde l}}$. Again we may assume that  the change of coordinates $\tilde\psi:(\theta,I)\mapsto (\varphi,J)$ is as regular as we want and conjugates $\widehat T_{\theta,\varepsilon}$ to a map of the form \eqref{eq:innermapsymbolic} which satisfies assumptions \ref{it:itemB1}-\ref{it:itemB2} and \ref{it:itemB5}.
\end{proof}

\subsection{Deformation theory for the scattering maps}

It now remains to show that for a typical $f$, at least one of the scattering maps $\{S_{i,\varepsilon}(f)\}_i:N\to N$ in \eqref{eq:scattmapproof} verifies the hypothesis \ref{it:itemB5} and that, together, they satisfy  assumption \ref{it:B6}. To do so we will make use of the perturbation theory of the scattering map developed in \cite{DelaLLaveScattmap}. 

\begin{lem}\label{lem:deformationoriginaldyn}
Let $H_0$ and $f$ be $C^2$ time-dependent Hamiltonians
on a smooth compact symplectic manifold $X$.
Then, for any $\varepsilon\geq 0$, the time-one map
$F_{\varepsilon,f}=\Phi_{H_{\varepsilon,f}}$,
where $H_{\varepsilon,f}=H_0+\varepsilon f$,
satisfies the initial value problem
\[
\begin{cases}
\displaystyle
\frac{\mathrm d}{\mathrm d\varepsilon}F_{\varepsilon,f}
=Z_{\varepsilon,f}\circ F_{\varepsilon,f},\\
F_{0,f}=F_0,
\end{cases}
\]
where $Z_{\varepsilon,f}$ is the Hamiltonian vector
field associated to the Hamiltonian function
$K_\varepsilon(f):X\to\mathbb R$ given by
\begin{equation}\label{eq:deformationHam final}
K_\varepsilon(f)(\cdot)
=\int_0^1
f(\phi^{1,u}_{H_{\varepsilon,f}}(\cdot),u)
\,\mathrm du.
\end{equation}
\end{lem}

\begin{proof}
Differentiating the equation satisfied by the time-$t$
flow, we obtain that $
Y^t_{\varepsilon,f}
:=\frac{\mathrm d}{\mathrm d\varepsilon}
\phi^t_{H_{\varepsilon,f}}$ 
satisfies the linear ODE
\[
\dot Y^t_{\varepsilon,f}
=
DX_{H_{\varepsilon,f}}
(\phi^t_{H_{\varepsilon,f}})
Y^t_{\varepsilon,f}
+
X_{f}(\phi^t_{H_{\varepsilon,f}})
\]
with initial condition $Y^0_{\varepsilon,f}=0$. By the variation-of-constants formula, its solution
at time $t=1$ is
\begin{align*}
Y_{\varepsilon,f}:=Y^1_{\varepsilon,f}
=&D\phi^1_{H_{\varepsilon,f}}
\int_0^1
(D\phi^u_{H_{\varepsilon,f}})^{-1}
X_{f}(\phi^u_{H_{\varepsilon,f}})
\ \mathrm du\\
=&\left(
\int_0^1
(\phi^{u,1}_{H_{\varepsilon,f}})_*
X_{f}
\ \mathrm du
\right)\circ F_{\varepsilon,f}\\
=&X_{K_\varepsilon(f)}\circ F_{\varepsilon,f},
\end{align*}
where in the last equality we used that
$\phi^{u,1}_{H_{\varepsilon,f}}$ is symplectic and
\[
(\phi^{u,1}_{H_{\varepsilon,f}})_*
X_{f}
=
X_{f(\phi^{1,u}_{H_{\varepsilon,f}}(\cdot),u)}.
\]
 Hence, the desired conclusion follows.
\end{proof}

\subsubsection*{Delshams-de la Llave-Seara theory of the scattering map}
We will exploit the following result, which can be easily read from the (more general) statement in Theorem 31 of that paper.

\begin{thm}[After Theorem 31 in \cite{DelaLLaveScattmap}]\label{thm:deformationtheoryscattmap}
  Given any $i=\{1,\dots,m\}$ and $f\in \mathcal B^\omega_{r_0}(M\times N)$ let $S_{i,\varepsilon}(f):N\to N$ be the scattering map in \eqref{eq:scattmapproof} and let $ K_\varepsilon(f):M\times N\to \mathbb R$  be the Hamiltonian in \eqref{eq:deformationHam final}. Then, $ S_{i,\varepsilon}(f)$  solves the initial value problem 
  \[
 \begin{cases}
     \frac{\mathrm d}{\mathrm d\varepsilon} S_{i,\varepsilon}(f)=X_{s_\varepsilon^i(f)}\circ S_{i,\varepsilon}(f)\\
 S_{i,\varepsilon}(f)|_{\varepsilon=0}=S_{i,0}
 \end{cases} \qquad\qquad\qquad \varpi_N(\cdot,X_{s_\varepsilon^i(f)})=\mathrm ds_\varepsilon^i(f)(\cdot)
  \]
where $s_\varepsilon^i(f):N\to \mathbb R$ is given by the absolutely convergent series \begin{equation}\label{eq:deformationscattmap}
    \begin{aligned}
      s_{\varepsilon}^i(f)=&  \lim_{N\to \infty} \left(\sum_{j=0}^{N-1} \left(K_{\varepsilon}(f)\circ F_{{\varepsilon,f}}^{-j}\circ(\Omega_{i,\varepsilon}^{u}(f))^{-1}\circ ( S_{ \Gamma^i_\varepsilon(f)})^{-1}\circ \psi_\varepsilon-K_{\varepsilon}(f)\circ F_{{\varepsilon,f}}^{-j}\circ ( S_{ \Gamma^i_\varepsilon(f)})^{-1}\circ \psi_\varepsilon \right)\right.\\
      &\left.\qquad\quad +\sum_{j=1}^N \left( K_{\varepsilon}(f)\circ F_{{\varepsilon,f}}^{j}\circ(\Omega_{i,\varepsilon}^{s}(f))^{-1}\circ \psi_\varepsilon-K_{\varepsilon}(f)\circ F_{{\varepsilon,f}}^{j}\circ \psi_\varepsilon\right) \right).
      \end{aligned}
    \end{equation}
    
\end{thm}

\begin{rem}
    Since in  Theorem \ref{thm:mainham}  we  only consider deformations of direct products all the scattering maps for $\varepsilon=0$ (i.e. those along the unperturbed homoclinic channels \eqref{eq:unperturbedchannels})  are given by the identity map. 
\end{rem}

 Appealing  to the regular dependence on parameters of all the objects involved (see  Section 4.2 in \cite{DelaLLaveScattmap}, in particular Theorem 23), we  obtain the following perturbative result.

\begin{prop}\label{lem:Melnikovformula}
 Given any $i=\{1,\dots,m\}$ and $f\in \mathcal B^\omega_{r_0}(M\times N)$, let $K_\varepsilon(f)$ be given by \eqref{eq:deformationHam final} and let  $s^i_{\varepsilon}(f)$ be as in \eqref{eq:deformationscattmap}. Fix $l>0$ and assume that $\kappa_0$ and $r_0$ are sufficently large. Then, \begin{equation}\label{eq:deformationscattmapasymp}
    \begin{aligned}
      s_{\varepsilon}^i(f)= \mathcal L_i(f)+O_{C^{l}}(\varepsilon ) 
      \end{aligned}
    \end{equation}
    where the function $\mathcal L_i(f)$ is given by the absolutely convergent series
    \begin{equation}\label{eq:Melnikovpotential}
    \begin{aligned}
    \mathcal L_i(f)=&  \lim_{N\to \infty} \left(\sum_{j=0}^{N-1} \left(K_{0}(f)\circ F_{{0}}^{-j}\circ(\Omega_{i,0}^{u})^{-1}\circ ( S_{ \Gamma^i_0})^{-1}\circ \psi_0-K_{0}(f)\circ F_{{0}}^{-j}\circ ( S_{ \Gamma^i_0})^{-1}\circ \psi_0 \right)\right.\\
      &\left.\qquad\quad +\sum_{j=1}^N \left( K_{0}(f)\circ F_{{0}}^{j}\circ(\Omega_{i,0}^{s})^{-1}\circ \psi_0-K_{0}(f)\circ F_{{0}}^{j}\circ \psi_0\right) \right).
    \end{aligned}
    \end{equation}
In particular, the map $S_{i,\varepsilon}(f):N\to N$ in \eqref{eq:scattmapproof} is of the form
\begin{equation}\label{eq:expansionscattmap}
    S_{i,\varepsilon}(f)=\mathrm{id}+\varepsilon X_{\mathcal L_i(f)}+O_{C^{l}}(\varepsilon^2)
\end{equation}
where the vector field $X_{\mathcal L_i(f)}:N\to TN$ is defined implicitly by 
\begin{equation}\label{eq:vfieldsscattmaps}
\varpi_N(\cdot, X_{\mathcal L_i(f)})= \mathrm d\mathcal L_i(f).
\end{equation}
\end{prop}

The approximation \eqref{eq:deformationscattmapasymp} in terms of the so-called \textit{Melnikov potential} \eqref{eq:Melnikovpotential} is key to verify assumptions \ref{it:itemB3} and \ref{it:B6}. Indeed, \eqref{eq:Melnikovpotential} depends linearly on the generator $f$ of the perturbation, so it is very easy to understand how it behaves as we move $f$.

\subsubsection*{Verification of assumption \ref{it:itemB3}}

We start by analyzing the coordinate representation of the scattering maps in the coordinate system built in Proposition \ref{prop:innerdynamicsassumption}.

\begin{lem}\label{lem:assumptionB4}
  Given any $i=\{1,\dots,m\}$ and $f\in \mathcal B^\omega_{r_0}(M\times N)$, let $S_{i,\varepsilon}(f):N\to N$ be the scattering map in \eqref{eq:scattmapproof} and let $\mathcal L_i(f)$ be given by \eqref{eq:Melnikovpotential}. Let $\psi_{\mathrm{pol}}:(\varphi,J)\in\mathbb A^d\to z\in \mathcal A\subset N $ be the coordinate chart in Lemma \ref{lem:Weinsteinlemma} and define the function
  \begin{equation}\label{eq:restrictedMelnikov}
 \widetilde {\mathcal L}_i(f)=\psi_{\mathrm{pol}}^*\  \mathcal L_i(f)=\mathcal L_i (f)\circ\psi_{\mathrm{pol}}. 
  \end{equation}
  Fix $l>0$ and let $\psi_{\mathrm{BNF}}:\mathbb A^d\to N$ be the coordinate chart in Proposition \ref{prop:innerdynamicsassumption}. Then,  the map
   \begin{equation}\label{eq:scattmapsinBNFcoords}
   \widetilde S_{i,\varepsilon}(f)=\psi_{\mathrm{BNF}}^{-1}\circ S_{i,\varepsilon}(f)\circ \psi_{\mathrm{BNF}}
   \end{equation}
   is $C^l$ and of the form 
   \[
   \widetilde S_{i,\varepsilon}(f):\binom{\varphi}{J}\mapsto \binom{\varphi+\varepsilon\partial_J ( \widetilde{\mathcal L}_i(f))+O_{C^l}(\varepsilon^2)}{J-\varepsilon\partial_{\varphi}( \widetilde{\mathcal L}_i(f))+O_{C^l}(\varepsilon^2)}.
   \]
\end{lem}

\begin{proof}
Let $\psi_{\mathrm{pol}}:(\varphi,J)\in \mathbb A^d\to z\in N$ be the coordinate patch in Lemma \ref{lem:Weinsteinlemma}. Then, 
\begin{equation}\label{eq:scattmapinpolarvars}
\psi_{\mathrm{pol}}^{-1}\circ S_{i,\varepsilon}(f)\circ \psi_{\mathrm{pol}}=\mathrm{id}+\varepsilon X_{\widetilde {\mathcal L}_i(f)}+O_{C^l}(\varepsilon^2) 
\end{equation}
where the vector field $X_{\widetilde{\mathcal L}_i(f)}:\mathbb A^d\to T\mathbb A^{d}$ is defined implicitly by 
\[
(\mathrm dJ\wedge\mathrm d\varphi)(\cdot, X_{\widetilde{\mathcal L}_i(f)})=\mathrm d \widetilde{\mathcal L}_i(f)(\cdot).
\]
This already shows that the  map \eqref{eq:scattmapinpolarvars} is of the form  of the form \eqref{eq:scattmapsymbolic}. In particular, 
\[
\psi_{\mathrm{pol}}^{-1}\circ S_{i,\varepsilon}(f)\circ \psi_{\mathrm{pol}}:\binom{\varphi}{J}\mapsto \binom{\varphi+\varepsilon\partial_J ( \widetilde{\mathcal L}_i(f))+O_{C^l}(\varepsilon^2)}{J-\varepsilon\partial_{\varphi}( \widetilde{\mathcal L}_i(f))+O_{C^l}(\varepsilon^2)}.
\]
Finally, we note that the very same conclusion holds for the map \eqref{eq:scattmapsinBNFcoords} since the map $\psi_{\mathrm{pol}}^{-1}\circ\psi_{\mathrm{BNF}}$ is $O(\varepsilon)$-close to identity in the $C^l$ topology (see Proposition \ref{prop:innerdynamicsassumption}). 
\end{proof}

Let now $\varphi_*\in\mathbb T^d$ be a critical point of the function 
\[
\varphi\in\mathbb T^d\mapsto \widetilde {\mathcal L}_i(f)(\varphi,0)\in\mathbb R
\]
and denote by 
\[
B=\partial_{\varphi^2}^2\widetilde{\mathcal L}_i(f)(\varphi_*,0).
\]
If $A$ is the matrix in \eqref{eq:BNFform} we now show that for a typical $f$ the product $AB$ satisfies Asssumption \ref{it:itemB3}, i.e. it has $d$ different real non-zero eigenvalues. To do so, the first step is to study    the correspondence between Hamiltonians $f\in \mathcal B^\omega_{r_0}(M\times N)$ and the jet of the associated Melnikov potential \eqref{eq:restrictedMelnikov}. Given a  function $h:N\to \mathbb R$ we denote by 
\[
\mathrm{Jet}^k(h)(z)=\{\partial^\alpha h(z)\colon \alpha\in \mathbb N^{2d},\  0\leq |\alpha|_1\leq k\}.
\]

The following lemma will prove crucial in establishing genericity of assumptions \ref{it:itemB3} and \ref{it:B6}.
\begin{lem}\label{lem:surjectivity}
  Given any $i=\{1,\dots,m\}$ and $f\in \mathcal B^\omega_{r_0}(M\times N)$ let $\mathcal L_i(f):N\to \mathbb R$ be the Melnikov potential defined in \eqref{eq:Melnikovpotential}. For any $k\in\mathbb N$ and $z\in N$ the Fréchet differential of the map
   \begin{equation}\label{eq:jetmap}
   f\in \mathcal B^\omega_{r_0}(X)\to \{\mathrm{Jet}^k ({\mathcal L}_i(f))(z)\}_i  \in \underbrace{\mathbb R^{n(d,k)}\times\dots\times \mathbb R^{n(d,k)}}_{m\text{ times}}\qquad\qquad n(d,k)=\binom{2d+k}{k}
   \end{equation}
   is surjective.
\end{lem}

\begin{proof}
Fix some $i\in\{1,\dots,m\}$. We start by observing that the map \eqref{eq:jetmap} is linear. In coordinates,
\begin{equation}\label{eq:coordinateexpressionMelnikov}
    \mathcal L_i(f)(z)=  \lim_{N\to \infty}\sum_{j=-N}^{N-1} \left(K_{0}(f)(F_{{M}}^{j}(\gamma_i),F_{N}^{j}(z))-K_{0}(f)(a, F_{{N}}^{j}(z))\right)
    \end{equation}
  where $a\in M$ is the hyperbolic fixed point of the map $F_{M}$ and $\gamma_i\in W^u(a)\pitchfork W^s(a)$ is a transverse homoclinic point.  Let  $V\in M\times N$ be a sufficiently small open set  such that (recall that the orbits of different channels are, by assumption, disjoint)
    \begin{equation}\label{eq:nointersectionsupports}
 V\cap  \Gamma_0^i\neq \emptyset\qquad\qquad  V\cap  \big(\bigcup_{j\neq 0} F^j_{0}(\Gamma_0^i)\big)=\emptyset\qquad\qquad V\cap \big( \bigcup_{k\neq i}\bigcup_{j} F^j_{0}(\Gamma_0^k) \big)=\emptyset.
    \end{equation}
We will show that for a sufficiently small open set $W\subset V$, for any $z\in N$, the map
\begin{equation}\label{eq:surjectivemapCinfty}
g\in C^\infty_c(W\times \mathbb R,\mathbb R)\to
\mathrm{Jet}^k({\mathcal L}_i(f_g))(z),
\end{equation}
where $f_g$ is defined below, is surjective and then conclude the proof by appealing to the density of $\mathcal H^\omega(M\times N)$ (recall the definition of this set in \eqref{eq:spaceofHams}). To establish surjectivity of the map \eqref{eq:surjectivemapCinfty} we argue as follows. Fix any $\delta>0$ and let $\eta\in C^\infty_c((0,1),\mathbb R)$ be such that for any $g\in C^\infty_c(W\times\mathbb R,\mathbb R)$
\[
\left|
\mathrm{Jet}^k\left(
\int_0^1\eta(u)g(\gamma_i,\cdot,u)\,\mathrm du
-g(\gamma_i,\cdot,1/2)
\right)(z)
\right|
\leq \delta |g|_{C^{k+1}}.
\]
Choose now any $c\in\mathbb R^{n(d,k)}$ and let
$g\in C^\infty_c(W\times\mathbb R,\mathbb R)$ be such that
\[
\mathrm{Jet}^k(g|_{\substack{x_M=\gamma_i,t=\frac12}})(z)=c.
\]
Then, for the function
\[
 f_g(x,u)=\eta(u)g(\phi_{H_0}^{u,1}(x),u)
\]
we have, by \eqref{eq:deformationHam final},
\[
K_0(f_g)(x)
=\int_0^1\eta(u)g(x,u)\,\mathrm du.
\]
Consequently,
\[
\left|
\mathrm{Jet}^k(K_0(f_g)(\gamma_i,\cdot))(z)-c
\right|
\leq \delta |g|_{C^{k+1}},
\]
so, by \eqref{eq:coordinateexpressionMelnikov} and the construction of the set $V$,
\[
\left|
\mathrm{Jet}^k({\mathcal L}_i(f_g))(z)-c
\right|
\leq \delta |g|_{C^{k+1}}.
\]
Moreover, for any $j\neq i$,
\[
\mathrm{Jet}^k({\mathcal L}_j(f_g))(z)=0.
\]
Since $\delta>0$ is arbitrary, the conclusion follows by linearity and density of real-analytic functions.
\end{proof}

Making use of Lemma \ref{lem:surjectivity}, establishing the genericity of assumption \ref{it:itemB3}  boils down to a local perturbation argument.

\begin{lem}\label{lem:morselemma}
    Given $f\in\mathcal B^\omega_{r_0}(M\times N)$, let $\widetilde{\mathcal L}_i(f)$ be the function defined in \eqref{eq:restrictedMelnikov}. For $f$ in an open and dense subset of $\mathcal B^\omega_{r_0}(M\times N)$ there exists a critical point $\varphi_*\in \mathbb T^d$ of the map $\varphi\mapsto \widetilde{\mathcal L}_i(f)(\varphi,0)$ such that matrix $A(f)$ in Proposition \ref{prop:innerdynamicsassumption} and the matrix 
    \[
    B(f)=\partial^2_{\varphi^2}\widetilde{\mathcal L}_i(f)(\varphi_*,0)
    \]
    are such that the matrix $A(f)B(f)$ has $d$ different non-zero real eigenvalues. In particular, assumption \ref{it:itemB3} is satisfied for $f$ in an open and dense subset of $\mathcal B^\omega_{r_0}(M\times N)$.
\end{lem}

\begin{proof}
 Fix any $\varphi_*\in\mathbb T^d$ which is a critical point of the map $\varphi\mapsto \widetilde{\mathcal L}_i(f)(\varphi,0)$. A straightforward adaptation of the proof of Lemma \ref{lem:surjectivity} shows that at any $f\in \mathcal B^\omega_{r_0}(M\times N)$ it is possible to find  $f+\Delta f\in \mathcal B^\omega_{r_0}(M\times N)$ arbitrarily close to $f$ in the real-analytic topology (see Appendix \ref{sec:realanalytictop}) such that
\[
\partial_{\varphi}\widetilde{\mathcal L}_i(f+\Delta f)(\varphi_*,0)=0,
\]
 and for which the product $A(f+\Delta f)B(f+\Delta f)$ has $d$-different non-zero eigenvalues (note that not having simple spectrum or full rank, are positive codimension phenomena in the space of square matrices). If moreover, we choose $\varphi_*$ to be a local minimum, the matrix $B(f+\Delta f)$ is positive definite so we can write (all the matrices being the ones associated to $f+\Delta f$)
\[
AB=B^{-1/2}(B^{1/2}AB^{1/2})B^{1/2}
\]
and deduce that the eigenvalues of $A(f+\Delta f)B(f+\Delta f)$ must be real (since $A(f+\Delta f)$ and $B^{1/2}(f+\Delta f)$ are real  symmetric matrices).
\end{proof}

\subsubsection*{Verification of assumption \ref{it:B6}}

Assumption \ref{it:B6} involves a global property of the system of scattering maps so we cannot rely anymore on local perturbation arguments (as we did in establishing genericity of assumption \ref{it:itemB3}). We will instead deploy the powerful machinery of transversality theory. The  following version of the parametric transversality theorem can be easily read from Theorem 19.1 in \cite{MR240836}.
\begin{thm}[Parametric transversality theorem]\label{thm:parametrictransvthm}
    Let $r>0$ and let $X,Y$ and $\mathcal S$ be $C^r$ Banach manifolds and let $G:X\times \mathcal S\to Y$ be a $C^r$  map. Let $W\subset Y$ be a submanifold (not necessarily closed) and suppose that 
    \begin{enumerate}
        \item $\mathrm{dim}(X)=n<\infty$,
        \item $\mathrm{codim}(W)=q<\infty$,
        \item $r>\max\{0,n-q\}$,
        \item the map $G$ is transverse to $W$
    \end{enumerate}
    Then, 
    \[
 \mathcal S_W=\{s\in \mathcal S\colon \ \text{the map }x\mapsto G(x,s)\text{ is transverse to }W\} 
    \]
    is a residual subset of $\mathcal S$.
\end{thm}

A straightforward corollary of Theorem \ref{thm:parametrictransvthm} is the following.
\begin{cor}\label{cor:corollarytransv}
    Under the hypothesis in Theorem \ref{thm:parametrictransvthm}, if $n-q<0$, then for any $s\in\mathcal S_W$ we have $G(x,s)\notin W$ for all $x\in X$.
\end{cor}
\begin{proof}
    If $n<q$ then, at any $(x,s)\in X\times \mathcal S$, for the differential of the map $G_s:x\mapsto G(x,s)$ we have that $
    \mathrm{dim}(\mathrm{im}D G_s)<q$. Hence, since by assumption the map $G$ is transverse to $W$ we deduce  that for all $s\in\mathcal S_W$ we must have $G(x,s)\neq y$ for all $x\in X$.
\end{proof}

We now combine Lemma \ref{lem:surjectivity} with   Corollary \ref{cor:corollarytransv} to establish genericity of assumption \ref{it:B6}.

\begin{prop}\label{prop:assumptionB6}
    Let $m\geq 4d$ where $2d=\mathrm{dim}N$. Given any $i=\{1,\dots,m\}$ and $f\in \mathcal B^\omega_{r_0}(M\times N)$ let $S_{i,\varepsilon}(f):N\to N$ be the scattering map in \eqref{eq:scattmapproof}. Then, provided $f$ belongs to an open and dense subset of $\mathcal B^\omega_{r_0}(M\times N)$ the scattering maps $\{S_{i,\varepsilon}(f)\}_i:N\to N$ satisfy assumption \ref{it:B6}.
\end{prop}

\begin{proof}
  Write 
  \[
  W=\{v=(v_1,\dots,v_m)\in\mathbb (\mathbb R^{2d})^m\colon \mathrm{rank}(v_1,\dots,v_m)<2d\}=\bigcup_{k=0}^{2d-1}W_k
  \]
  with 
  \[
  W_k=\{v=(v_1,\dots,v_m)\in\mathbb (\mathbb R^{2d})^m\colon \mathrm{rank}(v_1,\dots,v_m)=k\}
  \]
  For any $k\in\{0,\dots,2d-1\}$ it is a standard fact\footnote{To see this it is enough to note that for any $k< 2d$, $\mathrm{dim} \mathrm{Gr}(k,2d)=k(2d-k)$. Hence, the locus where  $\mathrm{rank}(v_1,\dots,v_m)=k$ is a submanifold of dimension $mk+k(2d-k)$.} that 
  $W_k$ is a submanifold of $\mathbb R^{2dm}$ of codimension 
  \[
  \mathrm{codim}(W_k)=2dm-mk-(2d-k)k.
  \]
  On the other hand, Lemma \ref{lem:surjectivity} shows that  the map 
  \[
  G:(z,f)\in N\times\mathcal B^\omega_{r_0}(M\times N)\to  \{X_{\mathcal L_i(f)}\}_{i=1}^m(z)\in \mathbb R^{2dm}
  \]
  is transverse to $W_k$. Hence, since for all $0\leq k\leq 2d-1$ and $m\geq 4d$
  \[
 \mathrm{dim}(N)=2d< 2d+1\leq 2dm-mk-(2d-k)k=\mathrm{codim}(W_k),
  \]
  a direct application of Corollary \ref{cor:corollarytransv} shows that for $f$ belonging to an open and dense subset of ${\mathcal B}^\omega_{r_0}(M\times N)$ we have $\mathrm{rank}(\{X_{\mathcal L_i(f)} (z)\}_{i=1}^m) (z)=2d$ for all $z\in N$. 
\end{proof}

\subsection{Proof of Proposition \ref{prop:genericityofassumptions}}
The proof of Proposition \ref{prop:genericityofassumptions} follows directly from Proposition \ref{prop:innerdynamicsassumption} (assumptions \ref{it:itemB1}-\ref{it:itemB2} and \ref{it:itemB5} for the inner map), Lemmas  \ref{lem:assumptionB4} and \ref{lem:morselemma} (transversality-torsion assumption \ref{it:itemB3}) and Proposition \ref{prop:assumptionB6} (assumption \ref{it:B6} for the system of scattering maps).

\subsection{The symplectic case}
 We now describe the necessary modifications needed to treat the case where $F_0\in \mathrm{Symp}^\omega(M\times N)$ is a direct product of symplectomorphisms satisfying assumptions \ref{it:Assumption1}-\ref{it:Assumption2} and we consider an unfolding $F_{\varepsilon,\theta}=F_0\circ\Phi^{\varepsilon}_\theta\in \mathrm{Symp}^\omega(M\times N)$
   with $\Phi_\theta^{\varepsilon}$ as in \eqref{eq:unfoldingforsymplectic}.

\subsubsection*{The generating Hamiltonian}
For the symplectic unfolding
\[
F_{\varepsilon,\theta}=F_0\circ\Phi_\theta^\varepsilon,
\]
the infinitesimal generator of the deformation is
\[
Z_{\varepsilon,\theta}=(F_0)_*X_\theta,
\]
and satisfies
\[
\varpi(\cdot,Z_{\varepsilon,\theta})
=(F_0^{-1})^*\theta.
\]
Since $\theta$ is closed, this one-form is closed.
Consequently, the deformation is locally Hamiltonian,
with Hamiltonian
\[
K(\theta)=k_\theta\circ F_0^{-1},
\qquad dk_\theta=\theta.
\]
Inserting this Hamiltonian into \eqref{eq:Melnikovpotential} gives us the corresponding \textit{mutivalued} Melnikov potential $\mathcal L_i(\theta)$. 
\medskip

\subsubsection*{The intersection property}
It remains to justify the intersection property
\[
S_{i,\varepsilon}(\mathcal T_\varepsilon)
\cap\mathcal T_\varepsilon\neq\varnothing,
\]
where $\mathcal T_\varepsilon\subset\Lambda_\varepsilon$ is the corresponding Lagrangian
KAM torus for the map $F_{\varepsilon,\theta}$. Suppose first that $\dim N=2$. Then
$\mathcal T_\varepsilon$ is a simple closed curve
contained in a Darboux neighbourhood of the
elliptic fixed point. Let $D_\varepsilon$ be
the disk bounded by $\mathcal T_\varepsilon$.
Since $S_{i,\varepsilon}$ preserves a smooth measure $\mu$ on $N$ which is  $O(\varepsilon)$-close to the Liouville area
\[
\mu(S_{i,\varepsilon}(D_\varepsilon))
=\mu(D_\varepsilon).
\]
Hence,  these two disks must 
intersect. Indeed, if their boundaries were disjoint,
one disk would be strictly contained in the
other, contradicting equality of their $\mu$-measure.
\medskip

In arbitrary dimension, let $B\subset N$ be
a contractible Darboux neighbourhood containing
$\mathcal T_\varepsilon$ and write
$\varpi|_B=d\lambda$ where $\varpi$ is the induced symplectic form on $\Lambda_\varepsilon$.
Since $S_{i,\varepsilon}$ is symplectic (preserves $\varpi$) and close to the identity, its restriction
to a slightly smaller contractible neighbourhood $\widetilde B$
containing $\mathcal T_\varepsilon$ is exact symplectic. Indeed, we can assume that $S_{i,\varepsilon}(\widetilde B)\subset B$ so by symplecticity
\[
d(S_{i,\varepsilon}^*\lambda-\lambda)=0,
\]
and therefore, since $B$ is simply connected
\begin{equation}\label{eq:local exactness}
S_{i,\varepsilon}^*\lambda-\lambda=dG_{i,\varepsilon}
\end{equation}
for some function $G_\varepsilon:B\to \mathbb R$.
\medskip

By the Weinstein neighbourhood theorem \cite{MR286137},
$S_{i,\varepsilon}(\mathcal T_\varepsilon)$
can be represented as the graph of a closed
one-form $\eta_\varepsilon$ over
$\mathcal T_\varepsilon$. We want to show that $\eta_\varepsilon$ is exact. To do so let $\gamma\subset \mathcal T_\varepsilon$ be a closed curve and let $\tilde \gamma _\varepsilon\gamma\subset \mathcal T_{\varepsilon}$ be the projection to  $\mathcal T_{\varepsilon}$ of $S_{i,\varepsilon}(\gamma)$. Let $C_\varepsilon\subset B$ be the cylinder joining $\tilde\gamma_\varepsilon$ and $S_{i,\varepsilon}(\gamma)$. Then, by Stokes and \eqref{eq:local exactness} 
\[
\int_{C_\varepsilon}\varpi=\int_{S_{i,\varepsilon}(\gamma)}\lambda-\int_{\tilde\gamma_\varepsilon}\lambda=0
\]
Also, by construction
\[
\int_{C_\varepsilon}\varpi=\int_{\tilde\gamma_\varepsilon}\eta_\varepsilon
\]
so we indeed must have $\int_{\tilde\gamma_\varepsilon} \eta_\varepsilon=0$. In particular, there exists $h_\varepsilon:\mathcal T_\varepsilon\to \mathbb R$ such that 
\[
\eta_\varepsilon=\mathrm d h_\varepsilon.
\]

\medskip
Finally, since $\mathcal T_\varepsilon$ is compact,
$h_\varepsilon$ has a critical point,
at which its graph intersects the zero section.
This proves the desired intersection property.

\appendix

\section{The real-analytic topology in the space of diffeomorphisms}\label{sec:realanalytictop}

A differentiable  manifold $X$ is said to be \textit{real-analytic} if given any two charts 
\[
\varphi_i: U_i\subset \mathbb R^d\to V_i\subset X,\qquad\qquad \varphi_j: U_j\subset \mathbb R^d\to V_j\subset X
\]
with $V_i\cap V_j\neq \emptyset$ then, the corresponding transition function
\[
\varphi_{i,j}:=\varphi_j^{-1}\circ \varphi_i:U_{i,j}\subset \mathbb R^d \to U_{j,i}\subset\mathbb R^d\qquad\qquad \text{where}\qquad U_{i,j}:=\varphi_i^{-1}(V_i\cap V_j)
\]
is real-analytic.  Provided $X$ is compact there exists a complexification $X_{\mathbb C}$, i.e. a complex manifold $X_\mathbb C$ containing $X$ as a totally real submanifold and, moreover, this complexification is unique up to biholomorphisms\footnote{If $X_{\mathbb C}$ and $X_{\mathbb C}$ are two such complexifications then there exists neighbourhoods $U\subset X_{\mathbb C}$ and $U'\subset X_{\mathbb C}'$ of $X$ in $X_{\mathbb C}$ and in $X_{\mathbb C}'$, and a biholomorphism $\varphi:U\to U'$.} (see Proposition 1 in \cite{MR102094}). 
\subsection{The space $\mathcal C^\omega(X)$}
Given a Riemannian compact manifold $X$ let 
\[
\mathcal C^\omega(X)=\{f:X\to \mathbb R\colon f \text{ is real-analytic}\}
\]
Then, the real-analytic topology on $\mathcal C^\omega(X)$ is the one generated by the open sets 
\begin{align*}
\mathcal U(f,K,\varepsilon)=&\{g\in \mathcal C^\omega(X)\colon g\text{ admits a holomorphic extension $\bar g$ to an open neighbourhood of }K\text{ and}\\
&\qquad \sup_{z\in K}|\bar g(z)-\bar f(z)|<\varepsilon\}
\end{align*}
where $X\subset K\subset X_\mathbb C$ is compact and $f\in \mathcal C^\omega (X)$ admits a holomorphic extension $\bar f$  to an open neighbourhood of $K$. Namely, two elements are close if they are close in some compact neighbourhood of $X$ in $X_{\mathbb C}$.
\subsection{The space $\mathrm{Diff}^\omega(X)$}
Given a Riemannian compact manifold $X$, the real-analytic topology on $\mathrm{Diff}^\omega(X)$ is the one generated by the open sets 
\begin{align*}
\mathcal U(F,K,\varepsilon)=&\{G\in \mathrm{Diff}^\omega(X)\colon G\text{ admits a holomorphic extension $\bar G$ to an open neighbourhood of }K\text{ and}\\
&\qquad \sup_{z\in K}\mathrm{dist}(\bar G(z),\bar F(z))<\varepsilon\}
\end{align*}
where $X\subset K\subset X_\mathbb C$ is compact and $F\in \mathrm{Diff}^\omega(X)$ admits a holomorphic extension $\bar F$  to an open neighbourhood of $K$. Namely, two elements are close if they are close in some compact neighbourhood of $X$ in $X_{\mathbb C}$.

\section{Orbit segments near a normally hyperbolic invariant manifold}\label{sec:prooflocaldynamics}
In this section we provide the proofs of Propositions \ref{prop:localdynamics} and \ref{prop:globalmap}.

\subsection{Proof of Proposition \ref{prop:localdynamics}} From the expression \eqref{eq:mapF}  we observe that (as long as it stays in $V_\delta$) any  $\{(q_i,p_i,z_i)\}_{i=0}^n$  which satisfies the recursion
    \begin{equation}\label{eq:recursion}
    \begin{aligned}
        q_i=&\left(\prod_{k=0}^{i-1} \lambda (z_k)\right)  q +\sum_{k=0}^{i-1} \left(\prod_{j=k+1}^{i-1} \lambda(z_j)\right) Q(q_k,p_k,z_k)\\
        p_i=&\left(\prod_{k=i}^{n-1} \mu^{-1} (z_k)\right)  \bar p -\sum_{k=i}^{n-1}\left( \prod_{j=i}^k \mu^{-1} (z_j)\right) P(q_k,p_k,z_k)\\
        z_i=&T^{i-n}(\bar z)
    \end{aligned}
    \end{equation}
    is an orbit segment of the map $F$ and satisfies $q_0= q $, $p_n=  \bar p $ and $z_n=\bar z$. We will show the existence of solutions to \eqref{eq:recursion} by means of a fixed point argument. To that  we introduce the notation
    \[
    \bs w( q ,  \bar p ,\bar z)=\{w_i( q ,  \bar p ,\bar z)\}_{i=0}^n=\{(q_i( q , \bar  p ,\bar z),p_i( q , \bar  p ,\bar z))\}_{i=0}^n
    \]
    and define the Banach space
    \[
    \mathcal X=\{\bs w :[-\delta_0,\delta_0]^2\times \Lambda \to \mathbb R^{2n}\colon \lVert \bs w\rVert<\infty\}
    \]
    where $\lVert \cdot\rVert$ stands for the weighted $C^1$ norm
    \begin{align*}
   \lVert \bs w\rVert= \max_{0\leq i\leq n} &\left( |(\lambda^{(0,i-1)} )^{-1}q_i|_{C^0}+ \delta_0|(\lambda^{(0,i-1)} )^{-1}\partial_{( q ,  \bar p )}q_i|_{C^0}  +\delta_0 \alpha^{-n}|(\lambda^{(0,i-1)} )^{-1} \partial_{\bar z} q_i|_{C^0}\right. \\
   &\qquad\left. +|(\mu^{(i,n-1)})^{-1} p_i|_{C^0}+\delta_0|(\mu^{(i,n-1)})^{-1} \partial_{( q , \bar  p )}p_i|_{C^0}+ \delta_0 \alpha^{-n}|(\mu^{(i,n-1)})^{-1} \partial_{\bar z} p_i|_{C^0} \right).
    \end{align*}
    and, for $i,k=0,\dots,n$ we have written
\[
\lambda^{(k,i)}(\bar z)=\prod_{j=k}^{i}\lambda(z_j)=\prod_{j=k}^{i} \lambda(T^{j-n}(\bar z))\qquad\qquad \mu^{(i,k)}(\bar z)=\prod_{j=i}^k \mu^{-1}(z_j)=\prod_{j=i}^k \mu^{-1}(T^{j-n}(\bar z)).
\]
Then, we define
    \[
    \mathcal F(\bs w)=\begin{pmatrix}
        \left\{\lambda^{(0,i-1)}  q  +\sum_{k=0}^{i-1} \lambda^{(k+1,i-1)}Q(\pi_k\bs w)\right\}_{i=0}^{n}\\
          \left\{\mu^{(i,n-1)}   \bar p  -\sum_{k=i}^{n-1} \mu^{(i,k)} P(\pi_k\bs w)\right\}_{i=0}^n
    \end{pmatrix}
    \]
    We will show that the operator $\mathcal F$ is well defined and contracting on a a neighbourhood of the origin and conclude the existence of a (locally) unique fixed point. This is done in a number of steps:

\medskip

\noindent\textit{Preliminary estimates:} Recall that $Q,P$ are $C^2$ and satisfy 
\[
Q(0,p,z)=0=\partial_q Q(0,0,z)\qquad\qquad P(q,0,z)=0=\partial_p P(q,0,z).
\]
In particular, 
\[
|\partial_p Q(q,p,z)|\lesssim |q|,\qquad|\partial_q Q(q,p,z)|,|\partial_p P(q,p,z)|\lesssim |q|+|p|,\qquad|\partial_q P(q,p,z)|\lesssim |p|.
\]
Then, at any $(q,p,z), (\tilde q,\tilde p,z)\in[-2\delta_0,2\delta_0]^2\times\Lambda$:
\begin{itemize}
\item Lipschitz estimates at $C^0$ level. Observe that
\begin{align*}
|Q(q,p,z)-Q(\tilde q,\tilde p,z)|\lesssim (|q|+|p|+|\tilde q|+|\tilde p|)|q-\tilde q|+ (|q|+|\tilde q|) |p-\tilde p|\\
|P(q,p,z)-P(\tilde q,\tilde p,z)|\lesssim (|p|+|\tilde p|)|q-\tilde q|+ (|q|+|p|+|\tilde q|+|\tilde p|) |p-\tilde p|
\end{align*}
In particular, if $\bs w, \tilde{\bs w}\in \mathcal X$  we deduce that 
\begin{equation}\label{eq:Lipestimlocaldynproof}
\begin{split}
\lambda^{(k+1,i-1)}(\bar z) |Q(\pi_k \bs w)-Q(\pi_k \tilde{\bs w})|\lesssim \lambda^{(0,i-1)}(\bar z)(\lambda^{(0,k-1)}(\bar z)+\mu^{(k,n-1)}(\bar z)) (\lVert \bs w \rVert+\lVert \tilde{\bs w}\rVert) \lVert \bs w-\tilde{\bs w}\rVert\\
\mu^{(i,k)}(\bar z) |P(\pi_k \bs w)-P(\pi_k \tilde{\bs w})|\lesssim \mu^{(i,n-1)}(\bar z)(\lambda^{(0,k-1)}(\bar z)+\mu^{(k,n-1)}(\bar z)) (\lVert \bs w \rVert+\lVert \tilde{\bs w}\rVert) \lVert \bs w-\tilde{\bs w}\rVert.
\end{split}
\end{equation}
\item Lipschitz $C^1$ estimates with respect to $( q ,  p )$. Start by observing that
\begin{align*}
|\partial_q Q(q,p,z)-\partial_q Q (\tilde q,\tilde p,z)|\lesssim &|q-\tilde q|+|p-\tilde p|\\
|\partial_p Q(q,p,z)-\partial_p Q(\tilde q,\tilde p,z)|\lesssim & |q-\tilde q|+f(|q|+|\tilde q|)|p-\tilde p|\\
|\partial_q P(q,p,z)-\partial_q P(\tilde q,\tilde p,z)|\lesssim& f(|p|+|\tilde p|)|q-\tilde q|+|p-\tilde p|\\
|\partial_p  P(q,p,z)-\partial_p P(\tilde q,\tilde p,z)|\lesssim &|q-\tilde q|+|p-\tilde p|
\end{align*}
for some continuous function $f$ wich satisfies $f(0)=0$.  Also, for $\bs w\in\mathcal X$ 
\[
|\partial_{( q ,  p )} \pi_{k,q} \bs w|\lesssim \lambda^{(0,k-1)}(\bar z)\lVert \bs w\rVert\qquad\qquad|\partial_{( q ,  p )} \pi_{k,p} \bs w|\lesssim \mu^{(k,n-1)}(\bar z)\lVert \bs w\rVert.
\]
Therefore, 
for $\bs w, \tilde{\bs w}\in \mathcal X$ we get 
\begin{equation}\label{eq:LipestimC1localdynproof}
\begin{split}
\lambda^{(k+1,i-1)}(\bar z) |\partial_{( q ,  p )}Q(\pi_k \bs w)-\partial_{( q ,  p )}Q(\pi_k \tilde{\bs w})|\lesssim \lambda^{(0,i-1)}(\bar z)(\lambda^{(0,k-1)}(\bar z)+\mu^{(k,n-1)}(\bar z)) (\lVert \bs w \rVert+\lVert \tilde{\bs w}\rVert) \lVert \bs w-\tilde{\bs w}\rVert\\
\mu^{(i,k)}(\bar z) |\partial_{( q ,  p )}P(\pi_k \bs w)-\partial_{( q ,  p )}P(\pi_k \tilde{\bs w})|\lesssim \mu^{(i,n-1)}(\bar z)(\lambda^{(0,k-1)}(\bar z)+\mu^{(k,n-1)}(\bar z)) (\lVert \bs w \rVert+\lVert \tilde{\bs w}\rVert) \lVert \bs w-\tilde{\bs w}\rVert.
\end{split}
\end{equation}
\item Lipschitz $C^1$ estimates with respect to $\bar z$. On one hand 
\begin{equation}\label{eq:C^1estimzetalambdalocaldynproof}
|\partial_{\bar z} \lambda^{(k,i)}(\bar z )| \lesssim \lambda^{(k,i)}(\bar z) \alpha^{n-k}\qquad\qquad |\partial_{\bar z} \mu^{(i,k)}(\bar z )| \lesssim \mu^{(i,k)}(\bar z) \alpha^{n-i}.
\end{equation}
On the other hand,
\[
|\partial_{\bar z} z_k(\bar z)|\lesssim \alpha^{n-k}
\]
Also, 
\begin{align*}
|\partial_z Q(q,p,z)-\partial_z Q(\tilde q,\tilde p,z)|\lesssim &(|q|+|\tilde q|+|p|+|\tilde p|)|q-\tilde q|+(|q|+|\tilde q|)|p-\tilde p|\\
|\partial_z P(q,p,z)-\partial_zP(\tilde q,\tilde p,z)|\lesssim& (|p|+|\tilde p|)|q-\tilde q|+(|q|+|\tilde q|+|p|+|\tilde p|)|p-\tilde p|.
\end{align*}
Hence, it is not difficult to show that for $\bs w,\tilde{\bs w}\in\mathcal X$ 
and
\begin{equation}\label{eq:LipestimC1zetalocaldynproof}
\begin{split}
\lambda^{(k+1,i-1)}(\bar z) |\partial_{\bar z}Q(\pi_k \bs w)-\partial_zQ(\pi_k \tilde{\bs w})|\lesssim \alpha^n\lambda^{(0,i-1)}(\bar z)(\lambda^{(0,k-1)}(\bar z)+\mu^{(k,n-1)}(\bar z)) (\lVert \bs w \rVert+\lVert \tilde{\bs w}\rVert) \lVert \bs w-\tilde{\bs w}\rVert\\
\mu^{(i,k)}(\bar z) |\partial_{\bar z}P(\pi_k \bs w)-\partial_zP(\pi_k \tilde{\bs w})|\lesssim \alpha^n \mu^{(i,n-1)}(\bar z)(\lambda^{(0,k-1)}(\bar z)+\mu^{(k,n-1)}(\bar z)) (\lVert \bs w \rVert+\lVert \tilde{\bs w}\rVert) \lVert \bs w-\tilde{\bs w}\rVert.
\end{split}
\end{equation}
\end{itemize}
\medskip

\noindent\textit{First iteration:} Let 
    \begin{equation}\label{eq:initialpointfixedpointlocaldyn}
    \bs w_0=
       \left( \{\lambda^{(0,i-1)}  q \}_{i=0}^n,\ 
          \{\mu^{(i,n-1)}   p \}_{i=0}^n \right)
    \end{equation}
with $(q_0,p_n)\in[-\delta_0,\delta_0]^2$. By direct computation, the definition of the norm $\lVert\cdot\rVert $ and the estimate \eqref{eq:C^1estimzetalambdalocaldynproof} we obtain
\begin{equation}\label{eq:estimatefirstiterationfixedpoint}
\lVert \bs w_0\rVert\lesssim \delta_0.
\end{equation}

\medskip
\noindent\textit{Lipschitz estimate:} Let now $C>0$ be any fixed constant and let $\bs w,\tilde{\bs w}\in B(C\delta_0)\subset \mathcal X$. Then, by the estimates \eqref{eq:Lipestimlocaldynproof},\eqref{eq:LipestimC1localdynproof} and\eqref{eq:LipestimC1zetalocaldynproof} it follows that 
\begin{equation}\label{eq:lipschitzconstantfixedpoint}
\lVert \mathcal F(\bs w)-\mathcal F(\tilde{\bs w})\rVert \lesssim \delta_0 \lVert \bs w-\tilde{\bs w}\rVert.
\end{equation}

\medskip
\noindent\textit{Conclusion:} In view of the inequailities \eqref{eq:initialpointfixedpointlocaldyn} and \eqref{eq:lipschitzconstantfixedpoint}, the Banach fixed point theorem ensures the existence of a fixed point $\bs w\in\mathcal X$ which satisfies $\lVert \bs w\rVert\lesssim \delta_0$. But then, 
\[
\lVert \bs w-\bs w_0\rVert=\lVert \mathcal F(\bs w)-\mathcal F(0)\rVert \lesssim \delta_0 \lVert \bs w\rVert\lesssim \delta_0^2
\]
and the asymptotic estimates \eqref{eq:asymptoticslocal} follow.

\subsection{Proof of Proposition \ref{prop:globalmap}}\label{sec:appendixreturnmap} Given $(q_0, p_n, z_n)\in[-\delta_0,\delta_0]^2\times N$ we denote by $\{(q_i,p_i,z_i)\}_i$,  the corresponding orbit segment in Proposition \ref{prop:localdynamics} and  define 
 $\Phi^s_{n}:(q_0, p_n,z_n)\mapsto (q_0,p_0,z_0)$
and $\Phi^u_{n}:(q_0,p_n,z_n)\mapsto (q_n,p_n,z_n)$. We divide the proof  in three steps. First, we consider the functions 
\[
(\tilde q,\tilde p,\tilde z)=\Phi_{\mathrm{out}}\circ\Psi(q,p,z).
\]
It then follows from Lemma \ref{lem:outerdynamics} that for $(x,y,z)\in[0,\delta]\times[0,K/\underline \mu^n]\times N$
\[
|\partial_q \tilde q|,|\partial_p \tilde q|,|\partial_z \tilde q|\lesssim 1, \qquad\qquad |\partial_q \tilde p|,|\partial_z \tilde p|\lesssim 1/\underline \mu^n,\qquad\qquad |\partial_q \tilde z|,|\partial_p \tilde z|,|\partial_z \tilde z|\lesssim 1,
\]
and
\[
1\lesssim |\partial_ p\tilde p|\lesssim 1.
\]
We now  consider the functions 
\[
(\hat q,\hat p,\hat z)=(\Phi^{s}_{n,\varepsilon})^{-1}\circ\Phi_{\mathrm{out}}\circ\Psi(q,p,z).
\]
In view of Proposition \ref{prop:localdynamics} 
\[
(\Phi^s_{n,\varepsilon})^{-1}:\begin{pmatrix}
    q\\p\\z
\end{pmatrix}\mapsto \begin{pmatrix}
    q\\ \mu^{(n)}(T^n(z)) p+\tilde g(x,y,z)\\ T^n(z)
\end{pmatrix}
\]
with $|\tilde g|,|\partial_q\tilde g|\lesssim 1$, $|\partial_p \tilde g|\lesssim (\bar\mu/\underline \mu)^n$, $|\partial_z \tilde z|\lesssim (\alpha/\underline \mu)^n$. Then, 
\[
\hat q=\tilde q,\qquad\qquad |\partial_q \hat p|,|\partial_z \hat p|\lesssim (\alpha \bar \alpha)^n,\qquad\qquad |\partial_q \tilde z|,|\partial_p \tilde z|,|\partial_z \tilde z|\lesssim \alpha^n
\]
and
\[
\underline\mu^n\lesssim |\partial_p \hat p|\lesssim \bar \mu^n.
\]
Finally, define 
\[
 (\bar q,\bar p,\bar z)=\Psi^{-1}\circ \Phi^u_{n,\varepsilon}\circ(\Phi^{s}_{n,\varepsilon})^{-1}\circ\Phi_{\mathrm{out}}\circ\Psi(q,p,z).
\]
Since we have seen in Proposition \ref{prop:localdynamics} that 
\[
\Phi^u_{n,\varepsilon}:\begin{pmatrix}
    q\\p\\z
\end{pmatrix}\mapsto \begin{pmatrix}
    \lambda^{(n)}(z)q+ h(q,p,z)\\ p\\ z
\end{pmatrix}
\]
with $|h|,|\partial_q h|,|\partial_p h|,|\partial_z h|\lesssim \bar \lambda^n$ we deduce that 
\[
|\partial_q \bar q|\lesssim (\bar\lambda \alpha\bar \alpha)^n,\quad |\partial_p \bar q|\lesssim (\bar\lambda \bar \mu)^n,\quad |\partial_z \bar q|\lesssim (\bar\lambda \alpha)^n,\qquad\qquad \bar p=\hat p,\qquad\qquad \bar z=\hat z.
\]

\medskip

 \section{Commutators of a generic pair of scattering maps}\label{sec:hormander}

 We present here a proof of Proposition \ref{prop:genericityhormanderwithtwomaps}. That is, we want to show that even if we take $m=2$, i.e. we just look at two homoclinic channels and their corresponding scattering maps, we can generically verify condition \ref{it:B6}. To do so, recall that, in virtue of Lemma \ref{lem:surjectivity}, at any point $(z,f)\in N\times \mathcal B_{r_0}^\omega(M\times N)$ we can freely move the jets of the vector fields $\{X_{\mathcal L_i(f)}\}_{i=1}^2$. We now upgrade this conclusion to commutators of these vector fields. For technical reasons, we will only consider a suitable class of commutators for which certain algebraic manipulations simplify considerably. More precisely we let $\tilde m\in\mathbb N$ to be chosen later and write
\[
 \mathcal I_{\tilde m}^{(1)}=\bigcup_{k=1}^{\tilde m}\{\iota\in\{1,2\}^k\colon \iota_1=2,\ \iota_2=\cdots=\iota_k=1\}.
\]
and
\[
 \mathcal I_{\tilde m}^{(2)}=\bigcup_{k=1}^{\tilde m}\{\iota\in\{1,2\}^k\colon \iota_1=1,\ \iota_2=\cdots=\iota_k=2\}.
\]
Before stating our next result we recall the notation for commutators introduced in Section \ref{sec:transporIFS}. Namely, given vector fields $\{X_i\}_{i=1}^2:N\to TN$ and  $\iota\in\{1,2\}^k$  we write 
\[
Y_{\iota} =[X_{\iota_n},[\dots[X_{\iota_3},[X_{\iota_2},X_{\iota_1}]]\dots]].
\]

\begin{lem}\label{lem:surjectivitymaptobracket}
    Let $f\in \mathcal B^\omega(M\times N)\setminus\{0\}$ and fix any $j\in\{1,2\}$. Then,  at any $z\in N$ for which 
    \[
    X_{\mathcal L_j(f)}\neq 0
    \]
     the Fréchet differential of the map 
    \[
   f\in \mathcal B^\omega (M\times N)\mapsto \{Y_{\iota}(f)\}_{\iota\in\mathcal I^{(j)} _{\tilde m}}(z)\simeq \mathbb{R}^{2d\tilde m}
    \]
    is surjective.
\end{lem}

To prove Lemma \ref{lem:surjectivitymaptobracket} we will use the following technical result.
\begin{lem}\label{lem:technicalmultilinear}
    Fix $v\in\mathbb R^{2d}\setminus\{0\}$ and let $k\in \mathbb N$. Then, for any $\ell\in (\mathbb R^{2d})^*\simeq \mathbb R^{2d} $ there exists a symmetric $k+1$ linear form $C$ such that 
    \[
    C[\underbrace{v,\dots,v}_{k\text{ times}},\cdot]=\ell(\cdot).
    \]
\end{lem}
\begin{proof}
   Let $\{e_i\}_{i=1}^{2d}$ be an orthonormal basis of $(\mathbb R^{2d})^*$. Then, given $v_1,\dots,v_k\in \mathbb R^{2d}$ we write 
   \[
   C[v_1,\dots,v_k,e_i]=\sum_{\sigma\in\{1,\dots,d\}^k} C_{i,\sigma}\  (v_1)_{\sigma_1}\dots (v_k)_{\sigma_k}
   \]
   so $C[v,\dots,v,e_i]=\sum_{\sigma} C_{i,\sigma}\  (v)_{\sigma_1}\dots (v)_{\sigma_k}$. 
   Without loss of generality we may assume that $v=(1,0,\dots,0)$ so 
   \[
   C[v,\dots,v,e_i]=C_{i,1,\dots,1}.
   \]
   Therefore, given $\ell=\sum_{i}\ell_i$ we choose $C_{i,1,\dots,1}=\ell_i$ and extend the definition so $C$ is symmetric.
\end{proof}
\medskip

\begin{proof}[Proof of Lemma \ref{lem:surjectivitymaptobracket}]
Throughout the proof we  write $X_i$ instead of $X_{\mathcal L_i(f)}$ to refer to the vector fields  \eqref{eq:vfieldsscattmaps}. We suppose for simplicity that the symplectic form $\varpi_N$ on $N$ is constant and that we can introduce coordinates $z=(\xi,\eta)$ such that $\varpi_N=\mathrm d\xi \wedge\mathrm d\eta$. We denote by $J$ the $2d$-dimensional matrix such that $\varpi_N(v_\xi,v_\eta)=(v_\xi,v_\eta)J\binom{v_\xi}{v_\eta}$. The proof in the general case can be deduced from this particular case by making use of Darboux's theorem on the local equivalence of symplectic forms.
\medskip

We fix any point $p\in N$  for which   $X_2\neq 0$  (the other case is handled with a similar argument) and write the Taylor series expansion of the vector fields $\{X_i\}_i$
\[
X_i=\sum_{k\geq 0} B_i^{(k)} [\underbrace{(z-p),\dots,(z-p)}_{k\text{ times}}]\qquad\qquad B_i^{(k)}=\frac{1}{k!} D^k X_i=\frac{1}{k!} D^{k} J \nabla \mathcal L_i(f).
\]
Since the vector fields $\{X_i\}$ are  given by the differential of a function, each $B_i^{(k)}$ can be identified with a $k+1$-linear symmetric operator. It is now an straightforward computation to observe that
\[
[X_i,X_j]=\sum_{k\geq 0} B_{ij}^{(k)} [(z-p),\dots,(z-p)]
\]
where
\[
B_{ij}^{(k)}[\cdot,\dots,\cdot]= (k+1)(B_j^{(k+1)} [B_i^{(0)},\cdot,\dots,\cdot]-B_i^{(k+1)}[B_j^{(0)},\cdot,\dots,\cdot])+R_{ij}^{(k)}[\cdot,\dots,\cdot]
\]
and $R_{ij}^{(k)}$ depends only on $\{B_i^{(m)},B_j^{(m)}\}_{m\leq k}$. For any $n\geq 2$ and $\iota\in\{1,2\}^n$ we define 
\[
Y_\iota=[X_{\iota_n},[\dots[X_{\iota_3},[X_{\iota_2},X_{\iota_1}]]\dots]].
\]
Hence, writing $\tilde \iota=(\iota_1,\dots,\iota_{n-1})$
\[
Y_\iota=\sum_{k\geq 0} B_{\iota}^{(k)} [(z-p),\dots,(z-p)]
\]
where
\[
B_{\iota}^{(k)}[\cdot,\dots,\cdot]= (k+1) (B_{i_n}^{(k+1)} [B_{\tilde\iota}^{(0)},\cdot,\dots,\cdot]-B_{\tilde\iota}^{(k+1)}[B_{i_n}^{(0)},\cdot,\dots,\cdot])+ R_{\iota}^{(k)}[\cdot,\dots,\cdot]
\]
and $R_{\iota}^{(k)}$ depends only on $\{B_i^{(m)},B_j^{(m)}\}_{m\leq k}$. In particular, it is not difficult to verify by induction that
\begin{align*}
B_\iota^{(k)}[\cdot,\dots,\cdot]
=& \frac{(k+n-1)!}{k!}
\Big(
B_{\iota_1}^{(k+n-1)}
[B_{\iota_2}^{(0)},\dots,B_{\iota_n}^{(0)},
\cdot,\dots,\cdot]-
B_{\iota_2}^{(k+n-1)}
[B_{\iota_1}^{(0)},B_{\iota_3}^{(0)},
\dots,B_{\iota_n}^{(0)},\cdot,\dots,\cdot]
\Big)\\
&+\widetilde R_\iota^{(k)}[\cdot,\dots,\cdot],
\end{align*}
where $\widetilde R_\iota^{(k)}$ depends only on $\{B_i^{(m)},B_j^{(m)}\}_{m\leq k+n-2}$. 
\medskip

We now focus our attention to  the set of covectors $B_\iota^{(0)}\in (\mathbb R^{2d})^*$ with $\iota\in\mathcal I_{\tilde m}^{(2)}$
\[
B_\iota^{(0)}
=(n-1)!\Big(
B_1^{(n-1)}
[\underbrace{B_2^{(0)},\dots,B_2^{(0)}}_{n-1}]-B_2^{(n-1)}
[B_1^{(0)},
\underbrace{B_2^{(0)},\dots,B_2^{(0)}}_{n-2}]
\Big)
+\widetilde R_\iota^{(0)}.
\]
Let $\{e_j\}_j$ be the orthonormal basis of $(\mathbb R^{2d})^*$.  Since,   by assumption 
\[
B_2^{(0)}=X_2\neq 0,
\]
 by Lemma \ref{lem:technicalmultilinear} there exists  a symmetric $n$-linear form $C_j$ such that 
\[
e_j=(n-1)!J C_j[
\underbrace{B_2^{(0)},\dots,B_2^{(0)}}_{n-1},
\cdot].
\]
We conclude the proof by observing that, in view of Lemma \ref{lem:surjectivity}, it is possible to find a function $g_j\in \mathcal H^\omega(M\times N)$ such that 
\[
\frac{\mathrm{d}}{\mathrm{d}\mu}(B_{1}^{(n-1)}(f+\mu g_j))|_{\mu=0}= J C_j\qquad\qquad \frac{\mathrm{d}}{\mathrm{d}\mu}(B_{2}^{(n-1)}(f+\mu g_j))|_{\mu=0}=0
\]
and for $0\leq k\leq n-2$ and any $i=1,2$
\[
\frac{\mathrm{d}}{\mathrm{d}\mu}(B_{i}^{(k)}(f+\mu g_j))|_{\mu=0}=0
\]
so we are done.
\end{proof}

 To conclude the proof of Proposition \ref{prop:genericityhormanderwithtwomaps} we will also use the following. 
\begin{lem}\label{lem:generic non double vanishing}
For $f $ in an open and dense subset of $\mathcal B^\omega(M\times N)$ the following holds.  There exists a finite open cover $\{V_j\}_j\subset N$ such that for any $z\in V_j$ we have 
\[
X_{\mathcal L_1(f)}\neq 0\qquad \text{or} \qquad X_{\mathcal L_2(f)}\neq 0.
\]
\end{lem}
The proof of this result follows from a standard parametric transversality argument is left to the reader.

   \begin{proof}[Proof of Proposition \ref{prop:genericityhormanderwithtwomaps}] 
 Let $\mathcal C\subset \mathcal B^\omega(M\times N)$ be the open and dense subset in Lemma \ref{lem:generic non double vanishing}.  Let $\tilde m\geq 4\mathrm{dim}(N)$.   Making use of use of Lemma \ref{lem:surjectivitymaptobracket} verbatim repetition of the argument in the proof of Proposition \ref{prop:assumptionB6} shows  that for $f$ belonging to an open and dense subset of $\mathcal C$ we have 
\[
 \mathrm{rank}(\{Y_\iota\}_{\iota\in \mathcal I_{\tilde m}}^{(1)}, \{Y_\iota\}_{\iota\in \mathcal I_{\tilde m}}^{(2)}) (z)=2d
\]
 for all $z\in N$. The proof of Proposition \ref{prop:genericityhormanderwithtwomaps} is complete.
\end{proof}
    
\bibliography{Biblio}

@misc{guardia2025partiallyhyperbolicdynamics3body,
      title={Partially hyperbolic dynamics in the 3-body problem}, 
      author={Guardia, M. and  Paradela, J.},
      year={2025},
      eprint={2512.02133},
      archivePrefix={arXiv},
      primaryClass={math.DS},
      url={https://arxiv.org/abs/2512.02133}, 
}

@article {MR4624370,
    AUTHOR = {Clarke, A. and Turaev, D.},
     TITLE = {Arnold diffusion in multidimensional convex billiards},
   JOURNAL = {Duke Math. J.},
  FJOURNAL = {Duke Mathematical Journal},
    VOLUME = {172},
      YEAR = {2023},
    NUMBER = {10},
     PAGES = {1813--1878},
      ISSN = {0012-7094,1547-7398},
   MRCLASS = {37C83 (37J40)},
  MRNUMBER = {4624370},
MRREVIEWER = {Jianlu\ Zhang},
       DOI = {10.1215/00127094-2022-0073},
       URL = {https://doi-org.proxy-um.researchport.umd.edu/10.1215/00127094-2022-0073},
}

@article {MR4495839,
    AUTHOR = {Fayad, B.},
     TITLE = {Lyapunov unstable elliptic equilibria},
   JOURNAL = {J. Amer. Math. Soc.},
  FJOURNAL = {Journal of the American Mathematical Society},
    VOLUME = {36},
      YEAR = {2023},
    NUMBER = {1},
     PAGES = {81--106},
      ISSN = {0894-0347,1088-6834},
   MRCLASS = {37J25 (37J06 37J11 37J12 37J30)},
  MRNUMBER = {4495839},
MRREVIEWER = {Diogo\ Pinheiro},
       DOI = {10.1090/jams/997},
       URL = {https://doi-org.proxy-um.researchport.umd.edu/10.1090/jams/997},
}

@article {MR3646879,
    AUTHOR = {Bernard, P. and Kaloshin, V. and Zhang, K.},
     TITLE = {Arnold diffusion in arbitrary degrees of freedom and normally
              hyperbolic invariant cylinders},
   JOURNAL = {Acta Math.},
  FJOURNAL = {Acta Mathematica},
    VOLUME = {217},
      YEAR = {2016},
    NUMBER = {1},
     PAGES = {1--79},
      ISSN = {0001-5962,1871-2509},
   MRCLASS = {37J40 (37J45 70H08 70K30)},
  MRNUMBER = {3646879},
MRREVIEWER = {Mikhail\ B.\ Sevryuk},
       DOI = {10.1007/s11511-016-0141-5},
       URL = {https://doi-org.proxy-um.researchport.umd.edu/10.1007/s11511-016-0141-5},
}

@book {MR4298716,
    AUTHOR = {Kaloshin, V. and Zhang, K.},
     TITLE = {Arnold diffusion for smooth systems of two and a half degrees
              of freedom},
    SERIES = {Annals of Mathematics Studies},
    VOLUME = {208},
 PUBLISHER = {Princeton University Press, Princeton, NJ},
      YEAR = {2020},
     PAGES = {xiii+204},
      ISBN = {978-0-691-20253-2; 978-0-691-20252-5; 978-0-691-20493-2},
   MRCLASS = {37-02 (37J40)},
  MRNUMBER = {4298716},
MRREVIEWER = {Pau\ Mart\'in},
}

@article {MR283825,
    AUTHOR = {Pugh, C. and Shub, M.},
     TITLE = {Linearization of normally hyperbolic diffeomorphisms and
              flows},
   JOURNAL = {Invent. Math.},
  FJOURNAL = {Inventiones Mathematicae},
    VOLUME = {10},
      YEAR = {1970},
     PAGES = {187--198},
      ISSN = {0020-9910,1432-1297},
   MRCLASS = {57.48},
  MRNUMBER = {283825},
MRREVIEWER = {R.\ F.\ Williams},
       DOI = {10.1007/BF01403247},
       URL = {https://doi-org.proxy-um.researchport.umd.edu/10.1007/BF01403247},
}

@article {MR102094,
    AUTHOR = {Whitney, H. and Bruhat, F.},
     TITLE = {Quelques propri\'et\'es fondamentales des ensembles
              analytiques-r\'eels},
   JOURNAL = {Comment. Math. Helv.},
  FJOURNAL = {Commentarii Mathematici Helvetici},
    VOLUME = {33},
      YEAR = {1959},
     PAGES = {132--160},
      ISSN = {0010-2571,1420-8946},
   MRCLASS = {53.00 (32.00)},
  MRNUMBER = {102094},
MRREVIEWER = {S.\ Hitotumatu},
       DOI = {10.1007/BF02565913},
       URL = {https://doi.org/10.1007/BF02565913},
}

@article {MR2163534,
    AUTHOR = {Gidea, M. and de la Llave, R.},
     TITLE = {Topological methods in the instability problem of
              {H}amiltonian systems},
   JOURNAL = {Discrete Contin. Dyn. Syst.},
  FJOURNAL = {Discrete and Continuous Dynamical Systems},
    VOLUME = {14},
      YEAR = {2006},
    NUMBER = {2},
     PAGES = {295--328},
      ISSN = {1078-0947,1553-5231},
   MRCLASS = {37J40 (34C37 37C29 37C50 70H08)},
  MRNUMBER = {2163534},
MRREVIEWER = {Ernest\ Fontich},
       DOI = {10.3934/dcds.2006.14.295},
       URL = {https://doi-org.proxy-um.researchport.umd.edu/10.3934/dcds.2006.14.295},
}

@article {MR222474,
    AUTHOR = {Hörmander, L.},
     TITLE = {Hypoelliptic second order differential equations},
   JOURNAL = {Acta Math.},
  FJOURNAL = {Acta Mathematica},
    VOLUME = {119},
      YEAR = {1967},
     PAGES = {147--171},
      ISSN = {0001-5962,1871-2509},
   MRCLASS = {35.48 (47.00)},
  MRNUMBER = {222474},
MRREVIEWER = {Joel\ Smoller},
       DOI = {10.1007/BF02392081},
       URL = {https://doi-org.proxy-um.researchport.umd.edu/10.1007/BF02392081},
}

@article {MR4521949,
    AUTHOR = {Clarke, Andrew and Fejoz, Jacques and Guardia, Marcel},
     TITLE = {Topological shadowing methods in {A}rnold diffusion: weak
              torsion and multiple time scales},
   JOURNAL = {Nonlinearity},
  FJOURNAL = {Nonlinearity},
    VOLUME = {36},
      YEAR = {2023},
    NUMBER = {1},
     PAGES = {426--457},
      ISSN = {0951-7715,1361-6544},
   MRCLASS = {37J40},
  MRNUMBER = {4521949},
MRREVIEWER = {Mar\ Giralt},
       DOI = {10.1088/1361-6544/aca5df},
       URL = {https://doi-org.proxy-um.researchport.umd.edu/10.1088/1361-6544/aca5df},
}

@article {MR3917797,
    AUTHOR = {Barrientos, P.G. and Raibekas, A.},
     TITLE = {Robustly non-hyperbolic transitive symplectic dynamics},
   JOURNAL = {Discrete Contin. Dyn. Syst.},
  FJOURNAL = {Discrete and Continuous Dynamical Systems},
    VOLUME = {38},
      YEAR = {2018},
    NUMBER = {12},
     PAGES = {5993--6013},
      ISSN = {1078-0947,1553-5231},
   MRCLASS = {37D30 (37J10 70K50)},
  MRNUMBER = {3917797},
MRREVIEWER = {Sini\v sa\ Slijep\v cevi\'c},
       DOI = {10.3934/dcds.2018259},
       URL = {https://doi-org.proxy-um.researchport.umd.edu/10.3934/dcds.2018259},
}

@book {MR3674984,
    AUTHOR = {McDuff, D. and Salamon, D.},
     TITLE = {Introduction to symplectic topology},
    SERIES = {Oxford Graduate Texts in Mathematics},
   EDITION = {Third},
 PUBLISHER = {Oxford University Press, Oxford},
      YEAR = {2017},
     PAGES = {xi+623},
      ISBN = {978-0-19-879490-5; 978-0-19-879489-9},
   MRCLASS = {53D35 (53D40 57R17 57R57 57R58)},
  MRNUMBER = {3674984},
MRREVIEWER = {Hansj\"org\ Geiges},
       DOI = {10.1093/oso/9780198794899.001.0001},
       URL = {https://doi-org.proxy-um.researchport.umd.edu/10.1093/oso/9780198794899.001.0001},
}

@article {MR4160091,
    AUTHOR = {Gidea, M. and de la Llave, R. and Seara, T.M. },
     TITLE = {A general mechanism of instability in {H}amiltonian systems:
              skipping along a normally hyperbolic invariant manifold},
   JOURNAL = {Discrete Contin. Dyn. Syst.},
  FJOURNAL = {Discrete and Continuous Dynamical Systems},
    VOLUME = {40},
      YEAR = {2020},
    NUMBER = {12},
     PAGES = {6795--6813},
      ISSN = {1078-0947,1553-5231},
   MRCLASS = {37J25 (70H08 70H33)},
  MRNUMBER = {4160091},
MRREVIEWER = {Lennard\ F.\ Bakker},
       DOI = {10.3934/dcds.2020166},
       URL = {https://doi-org.proxy-um.researchport.umd.edu/10.3934/dcds.2020166},
}

@book {MR1345386,
    AUTHOR = {Arnold, V. I.},
     TITLE = {Mathematical methods of classical mechanics},
    SERIES = {Graduate Texts in Mathematics},
    VOLUME = {60},

 PUBLISHER = {Springer-Verlag, New York},
      YEAR = {1989},
     PAGES = {xvi+516},
      ISBN = {0-387-96890-3},
   MRCLASS = {70-02 (58F05 58Fxx 70Hxx)},
  MRNUMBER = {1345386},
}

@article {MR4913967,
    AUTHOR = {Delshams, A. and Zhang, K.},
     TITLE = {Generic global diffusion for analytic a priori unstable
              systems},
   JOURNAL = {Comm. Math. Phys.},
  FJOURNAL = {Communications in Mathematical Physics},
    VOLUME = {406},
      YEAR = {2025},
    NUMBER = {6},
     PAGES = {Paper No. 146, 22},
      ISSN = {0010-3616,1432-0916},
   MRCLASS = {37J40 (70H05 70H08)},
  MRNUMBER = {4913967},
MRREVIEWER = {Qinbo\ Chen},
       DOI = {10.1007/s00220-025-05342-1},
       URL = {https://doi.org/10.1007/s00220-025-05342-1},
}

@article {MR2039999,
    AUTHOR = {Dolgopyat, D. and Wilkinson, A.},
     TITLE = {Stable accessibility is {$C^1$} dense},
   JOURNAL = {Ast\'erisque},
  FJOURNAL = {Ast\'erisque},
    NUMBER = {287},
      YEAR = {2003},
     PAGES = {33--60},
      ISSN = {0303-1179,2492-5926},
   MRCLASS = {37D30 (37C20 37J10)},
  MRNUMBER = {2039999},
MRREVIEWER = {Lorenzo\ J.\ D\'iaz},
}

@book {MR1867362,
    AUTHOR = {Montgomery, R.},
     TITLE = {A tour of subriemannian geometries, their geodesics and
              applications},
    SERIES = {Mathematical Surveys and Monographs},
    VOLUME = {91},
 PUBLISHER = {American Mathematical Society, Providence, RI},
      YEAR = {2002},
     PAGES = {xx+259},
      ISBN = {0-8218-1391-9},
   MRCLASS = {53C17 (37J99 53C60 58E10 70G45 70H05)},
  MRNUMBER = {1867362},
MRREVIEWER = {Andrey\ V.\ Sarychev},
       DOI = {10.1090/surv/091},
       URL = {https://doi-org.proxy-um.researchport.umd.edu/10.1090/surv/091},
}

@book {Cassels72,
    AUTHOR = {Cassels, J. W. S.},
     TITLE = {An introduction to {D}iophantine approximation},
    SERIES = {Cambridge Tracts in Mathematics and Mathematical Physics},
    VOLUME = {No. 45},
 PUBLISHER = {Hafner Publishing Co., New York},
      YEAR = {1972},
     PAGES = {x+169},
   MRCLASS = {10FXX},
  MRNUMBER = {349591},
}

@book {MR1326374,
    AUTHOR = {Katok, A. and Hasselblatt, B.},
     TITLE = {Introduction to the modern theory of dynamical systems},
    SERIES = {Encyclopedia of Mathematics and its Applications},
    VOLUME = {54},
 PUBLISHER = {Cambridge University Press, Cambridge},
      YEAR = {1995},
     PAGES = {xviii+802},
      ISBN = {0-521-34187-6},
   MRCLASS = {58Fxx (34Cxx 34Dxx 58-01 58F11 58F15)},
  MRNUMBER = {1326374},
MRREVIEWER = {Edoh\ Amiran},
       DOI = {10.1017/CBO9780511809187},
       URL = {https://doi-org.proxy-um.researchport.umd.edu/10.1017/CBO9780511809187},
}

@article {MR1880,
    AUTHOR = {Chow, W-L.},
     TITLE = {\"Uber {S}ysteme von linearen partiellen
              {D}ifferentialgleichungen erster {O}rdnung},
   JOURNAL = {Math. Ann.},
  FJOURNAL = {Mathematische Annalen},
    VOLUME = {117},
      YEAR = {1939},
     PAGES = {98--105},
      ISSN = {0025-5831,1432-1807},
   MRCLASS = {36.0X},
  MRNUMBER = {1880},
MRREVIEWER = {E.\ W.\ Titt},
       DOI = {10.1007/BF01450011},
       URL = {https://doi-org.proxy-um.researchport.umd.edu/10.1007/BF01450011},
}

@article {MR3893271,
    AUTHOR = {Zhang, Z.},
     TITLE = {On stable transitivity of finitely generated groups of
              volume-preserving diffeomorphisms},
   JOURNAL = {Ergodic Theory Dynam. Systems},
  FJOURNAL = {Ergodic Theory and Dynamical Systems},
    VOLUME = {39},
      YEAR = {2019},
    NUMBER = {2},
     PAGES = {554--576},
      ISSN = {0143-3857,1469-4417},
   MRCLASS = {37C85 (37D25)},
  MRNUMBER = {3893271},
MRREVIEWER = {S\'ebastien\ Alvarez},
       DOI = {10.1017/etds.2017.28},
       URL = {https://doi-org.proxy-um.researchport.umd.edu/10.1017/etds.2017.28},
}

@article {MR2719428,
    AUTHOR = {Koropecki, A. and Nassiri, M.},
     TITLE = {Transitivity of generic semigroups of area-preserving surface
              diffeomorphisms},
   JOURNAL = {Math. Z.},
  FJOURNAL = {Mathematische Zeitschrift},
    VOLUME = {266},
      YEAR = {2010},
    NUMBER = {3},
     PAGES = {707--718},
      ISSN = {0025-5874,1432-1823},
   MRCLASS = {37E30 (37C85)},
  MRNUMBER = {2719428},
MRREVIEWER = {Jo\~ao\ Lopes Dias},
       DOI = {10.1007/s00209-009-0595-7},
       URL = {https://doi-org.proxy-um.researchport.umd.edu/10.1007/s00209-009-0595-7},
}

@book {MR209436,
    AUTHOR = {Arnold, V. I. and Avez, A.},
     TITLE = {Probl\`emes ergodiques de la m\'ecanique classique},
    SERIES = {Monographies Internationales de Math\'ematiques Modernes},
    VOLUME = {9},
 PUBLISHER = {Gauthier-Villars, \'Editeur, Paris},
      YEAR = {1967},
     PAGES = {ii+243},
   MRCLASS = {28.70 (70.00)},
  MRNUMBER = {209436},
MRREVIEWER = {F.\ Hahn},
}

@book {MR1853077,
    AUTHOR = {Cannas da Silva, A.},
     TITLE = {Lectures on symplectic geometry},
    SERIES = {Lecture Notes in Mathematics},
    VOLUME = {1764},
 PUBLISHER = {Springer-Verlag, Berlin},
      YEAR = {2001},
     PAGES = {xii+217},
      ISBN = {3-540-42195-5},
   MRCLASS = {53Dxx (53-01)},
  MRNUMBER = {1853077},
MRREVIEWER = {Brendan\ J.\ Foreman},
       DOI = {10.1007/978-3-540-45330-7},
       URL = {https://doi-org.proxy-um.researchport.umd.edu/10.1007/978-3-540-45330-7},
}

@article {MR3146593,
    AUTHOR = {Eliasson, L. H. and Fayad, B. and Krikorian, R.},
     TITLE = {K{AM}-tori near an analytic elliptic fixed point},
   JOURNAL = {Regul. Chaotic Dyn.},
  FJOURNAL = {Regular and Chaotic Dynamics. International Scientific
              Journal},
    VOLUME = {18},
      YEAR = {2013},
    NUMBER = {6},
     PAGES = {801--831},
      ISSN = {1560-3547,1468-4845},
   MRCLASS = {37J40 (34D45 70H08)},
  MRNUMBER = {3146593},
MRREVIEWER = {S.\ Yu.\ Pilyugin},
       DOI = {10.1134/S1560354713060154},
       URL = {https://doi-org.proxy-um.researchport.umd.edu/10.1134/S1560354713060154},
}

@article {MR4072793,
    AUTHOR = {Bounemoura, A. and Fayad, B. and Niederman, L.},
     TITLE = {Super-exponential stability for generic real-analytic elliptic
              equilibrium points},
   JOURNAL = {Adv. Math.},
  FJOURNAL = {Advances in Mathematics},
    VOLUME = {366},
      YEAR = {2020},
     PAGES = {107088, 30},
      ISSN = {0001-8708,1090-2082},
   MRCLASS = {37J25},
  MRNUMBER = {4072793},
MRREVIEWER = {Miguel\ \'Angel\ L\'opez},
       DOI = {10.1016/j.aim.2020.107088},
       URL = {https://doi-org.proxy-um.researchport.umd.edu/10.1016/j.aim.2020.107088},
}

@incollection {MR1206072,
    AUTHOR = {Yoccoz, J-C.},
     TITLE = {Travaux de {H}erman sur les tores invariants},
      NOTE = {S\'eminaire Bourbaki, Vol.\ 1991/92},
   JOURNAL = {Ast\'erisque},
  FJOURNAL = {Ast\'erisque},
    NUMBER = {206},
      YEAR = {1992},
     PAGES = {Exp. No. 754, 4, 311--344},
      ISSN = {0303-1179,2492-5926},
   MRCLASS = {58F27 (58F05 58F36)},
  MRNUMBER = {1206072},
MRREVIEWER = {Helmut\ R\"ussmann},
}

@article {MR2173426,
    AUTHOR = {Arnaud, M-C. and Bonatti, C. and Crovisier,
              S.},
     TITLE = {Dynamiques symplectiques g\'en\'eriques},
   JOURNAL = {Ergodic Theory Dynam. Systems},
  FJOURNAL = {Ergodic Theory and Dynamical Systems},
    VOLUME = {25},
      YEAR = {2005},
    NUMBER = {5},
     PAGES = {1401--1436},
      ISSN = {0143-3857,1469-4417},
   MRCLASS = {37J10 (37C05 37C20 37C50 37D20 37D30)},
  MRNUMBER = {2173426},
MRREVIEWER = {Lorenzo\ J.\ D\'iaz},
       DOI = {10.1017/S0143385704000975},
       URL = {https://doi-org.proxy-um.researchport.umd.edu/10.1017/S0143385704000975},
}

@article {MR2090361,
    AUTHOR = {Bonatti, C. and Crovisier, S.},
     TITLE = {R\'ecurrence et g\'en\'ericit\'e},
   JOURNAL = {Invent. Math.},
  FJOURNAL = {Inventiones Mathematicae},
    VOLUME = {158},
      YEAR = {2004},
    NUMBER = {1},
     PAGES = {33--104},
      ISSN = {0020-9910,1432-1297},
   MRCLASS = {37C05 (37B20 37C20 37C29 37C50 37C70 37D05 37J10)},
  MRNUMBER = {2090361},
MRREVIEWER = {Lorenzo\ J.\ D\'iaz},
       DOI = {10.1007/s00222-004-0368-1},
       URL = {https://doi-org.proxy-um.researchport.umd.edu/10.1007/s00222-004-0368-1},
}

@article {MR1790659,
    AUTHOR = {Shang, Z-J.},
     TITLE = {A note on the {KAM} theorem for symplectic mappings},
   JOURNAL = {J. Dynam. Differential Equations},
  FJOURNAL = {Journal of Dynamics and Differential Equations},
    VOLUME = {12},
      YEAR = {2000},
    NUMBER = {2},
     PAGES = {357--383},
      ISSN = {1040-7294,1572-9222},
   MRCLASS = {37J40 (37J10)},
  MRNUMBER = {1790659},
MRREVIEWER = {Guido\ Gentile},
       DOI = {10.1023/A:1009068425415},
       URL = {https://doi.org/10.1023/A:1009068425415},
}

@article {MR668410,
    AUTHOR = {Pöschel, J.},
     TITLE = {Integrability of {H}amiltonian systems on {C}antor sets},
   JOURNAL = {Comm. Pure Appl. Math.},
  FJOURNAL = {Communications on Pure and Applied Mathematics},
    VOLUME = {35},
      YEAR = {1982},
    NUMBER = {5},
     PAGES = {653--696},
      ISSN = {0010-3640,1097-0312},
   MRCLASS = {58F07 (70H05)},
  MRNUMBER = {668410},
MRREVIEWER = {Helmut\ R\"ussmann},
       DOI = {10.1002/cpa.3160350504},
       URL = {https://doi.org/10.1002/cpa.3160350504},
}

@book {MR595866,
    AUTHOR = {Pöschel, J.},
     TITLE = {Über invariante {T}ori in differenzierbaren {H}amiltonschen
              {S}ystemen},
 PUBLISHER = {Universit\"at Bonn, Mathematisches Institut, Bonn},
      YEAR = {1980},
     PAGES = {103},
   MRCLASS = {58F05 (34C35 70Hxx)},
  MRNUMBER = {595866},
MRREVIEWER = {D.\ Schmidt},
}

@article {MR286137,
    AUTHOR = {Weinstein, Alan},
     TITLE = {Symplectic manifolds and their {L}agrangian submanifolds},
   JOURNAL = {Advances in Math.},
  FJOURNAL = {Advances in Mathematics},
    VOLUME = {6},
      YEAR = {1971},
     PAGES = {329--346},
      ISSN = {0001-8708},
   MRCLASS = {57.50},
  MRNUMBER = {286137},
MRREVIEWER = {D.\ G.\ Ebin},
       DOI = {10.1016/0001-8708(71)90020-X},
       URL = {https://doi-org.proxy-um.researchport.umd.edu/10.1016/0001-8708(71)90020-X},
}

@misc{LiTuraevPreprint,
      title={Symplectic blenders near whiskered tori and persistence of saddle-center homoclinics}, 
      author={Li, D. and  Turaev, D.},
      year={2026},
      eprint={2603.20830},
      archivePrefix={arXiv},
      primaryClass={math.DS},
      url={https://arxiv.org/abs/2603.20830}, 
}

@article {MR163026,
    AUTHOR = {Arnold, V. I.},
     TITLE = {Instability of dynamical systems with many degrees of freedom},
   JOURNAL = {Dokl. Akad. Nauk SSSR},
  FJOURNAL = {Doklady Akademii Nauk SSSR},
    VOLUME = {156},
      YEAR = {1964},
     PAGES = {9--12},
      ISSN = {0002-3264},
   MRCLASS = {34.65 (57.48)},
  MRNUMBER = {163026},
MRREVIEWER = {J.\ Moser},
}

@article {MR1949441,
    AUTHOR = {Cresson, J.},
     TITLE = {Symbolic dynamics and {A}rnold diffusion},
   JOURNAL = {J. Differential Equations},
  FJOURNAL = {Journal of Differential Equations},
    VOLUME = {187},
      YEAR = {2003},
    NUMBER = {2},
     PAGES = {269--292},
      ISSN = {0022-0396,1090-2732},
   MRCLASS = {37J40 (37B10 37D05 37J45 70H08)},
  MRNUMBER = {1949441},
MRREVIEWER = {Jean-Pierre\ Marco},
       DOI = {10.1016/S0022-0396(02)00053-0},
       URL = {https://doi.org/10.1016/S0022-0396(02)00053-0},
}

@book {MR442980,
    AUTHOR = {Moser, J.},
     TITLE = {Stable and random motions in dynamical systems},
    SERIES = {Annals of Mathematics Studies},
    VOLUME = {No. 77},
 PUBLISHER = {Princeton University Press, Princeton, NJ; University of Tokyo
              Press, Tokyo},
      YEAR = {1973},
     PAGES = {viii+198},
   MRCLASS = {58FXX (34C35 70.58)},
  MRNUMBER = {442980},
MRREVIEWER = {Clark\ Robinson},
}

@article {Clarke25,
    AUTHOR = {Clarke, A. and Fejoz, J. and Guardia, M.},
     TITLE = {Why are inner planets not inclined?},
   JOURNAL = {Publ. Math. Inst. Hautes \'Etudes Sci.},
  FJOURNAL = {Publications Math\'ematiques. Institut de Hautes \'Etudes
              Scientifiques},
    VOLUME = {141},
      YEAR = {2025},
     PAGES = {1--98},
      ISSN = {0073-8301,1618-1913},
   MRCLASS = {70F15 (37N05 85A05)},
  MRNUMBER = {4916233},
       DOI = {10.1007/s10240-024-00151-z},
       URL = {https://doi-org.sire.ub.edu/10.1007/s10240-024-00151-z},
}

@book {MR2269239,
    AUTHOR = {Arnold, V. I. and Kozlov, V. V. and Neishtadt,
              A. I.},
     TITLE = {Mathematical aspects of classical and celestial mechanics},
    SERIES = {Encyclopaedia of Mathematical Sciences},
    VOLUME = {3},
   EDITION = {Third},
 PUBLISHER = {Springer-Verlag, Berlin},
      YEAR = {2006},
     PAGES = {xiv+518},
      ISBN = {978-3-540-28246-4; 3-540-28246-7},
   MRCLASS = {70-02 (37Jxx 70-01 70H03 70H05)},
  MRNUMBER = {2269239},
MRREVIEWER = {Ernesto\ A.\ Lacomba},
}

@article {MR4033892,
    AUTHOR = {Gidea, M. and de la Llave, R. and Seara, T.M.},
     TITLE = {A general mechanism of diffusion in {H}amiltonian systems:
              qualitative results},
   JOURNAL = {Comm. Pure Appl. Math.},
  FJOURNAL = {Communications on Pure and Applied Mathematics},
    VOLUME = {73},
      YEAR = {2020},
    NUMBER = {1},
     PAGES = {150--209},
      ISSN = {0010-3640,1097-0312},
   MRCLASS = {37J06 (70H08)},
  MRNUMBER = {4033892},
MRREVIEWER = {Arturo\ Vieiro},
       DOI = {10.1002/cpa.21856},
       URL = {https://doi.org/10.1002/cpa.21856},
}

@book {BeyondUH,
    AUTHOR = {Bonatti, C. and D\'iaz, L. J. and Viana, M.},
     TITLE = {Dynamics beyond uniform hyperbolicity},
    SERIES = {Encyclopaedia of Mathematical Sciences},
    VOLUME = {102},
 PUBLISHER = {Springer-Verlag, Berlin},
      YEAR = {2005},
     PAGES = {xviii+384},
      ISBN = {3-540-22066-6},
   MRCLASS = {37-02 (37C20 37C29 37D25 37D30)},
  MRNUMBER = {2105774},
MRREVIEWER = {Sheldon\ E.\ Newhouse},
}

@article {DelaLLavegaps,
    AUTHOR = {Delshams, A. and de la Llave, R. and Seara, T.M.},
     TITLE = {A geometric mechanism for diffusion in {H}amiltonian systems
              overcoming the large gap problem: heuristics and rigorous
              verification on a model},
   JOURNAL = {Mem. Amer. Math. Soc.},
  FJOURNAL = {Memoirs of the American Mathematical Society},
    VOLUME = {179},
      YEAR = {2006},
    NUMBER = {844},
      ISSN = {0065-9266,1947-6221},
   MRCLASS = {37J40 (70H09)},
  MRNUMBER = {2184276},
MRREVIEWER = {Jean-Pierre\ Marco},
       DOI = {10.1090/memo/0844},
       URL = {https://doi.org/10.1090/memo/0844},
}

@article{HormanderGdlLS,
AUTHOR={Gidea, M. and de la Llave, R. and Seara , T.M.},
TITLE={Methods of geometric control in the {A}rnold diffusion problem in Hamiltonian systems},
NOTE ={In preparation},
}

@article {MR4069249,
    AUTHOR = {Obata, D.},
     TITLE = {On the stable ergodicity of {B}erger-{C}arrasco's example},
   JOURNAL = {Ergodic Theory Dynam. Systems},
  FJOURNAL = {Ergodic Theory and Dynamical Systems},
    VOLUME = {40},
      YEAR = {2020},
    NUMBER = {4},
     PAGES = {1008--1056},
      ISSN = {0143-3857,1469-4417},
   MRCLASS = {37A25 (37C40 37D25)},
  MRNUMBER = {4069249},
MRREVIEWER = {Xinsheng\ Wang},
       DOI = {10.1017/etds.2018.65},
       URL = {https://doi-org.proxy-um.researchport.umd.edu/10.1017/etds.2018.65},
}

@article {MR3682778,
    AUTHOR = {Horita, V. and Sambarino, M.},
     TITLE = {Stable ergodicity and accessibility for certain partially
              hyperbolic diffeomorphisms with bidimensional center leaves},
   JOURNAL = {Comment. Math. Helv.},
  FJOURNAL = {Commentarii Mathematici Helvetici. A Journal of the Swiss
              Mathematical Society},
    VOLUME = {92},
      YEAR = {2017},
    NUMBER = {3},
     PAGES = {467--512},
      ISSN = {0010-2571,1420-8946},
   MRCLASS = {37D30 (37A25 37C40)},
  MRNUMBER = {3682778},
MRREVIEWER = {Cristina\ Lizana},
       DOI = {10.4171/CMH/417},
       URL = {https://doi-org.proxy-um.researchport.umd.edu/10.4171/CMH/417},
}

@article {MR4544807,
    AUTHOR = {Capi\'nski, M. J. and Gidea, M.},
     TITLE = {Arnold diffusion, quantitative estimates, and stochastic
              behavior in the three-body problem},
   JOURNAL = {Comm. Pure Appl. Math.},
  FJOURNAL = {Communications on Pure and Applied Mathematics},
    VOLUME = {76},
      YEAR = {2023},
    NUMBER = {3},
     PAGES = {616--681},
      ISSN = {0010-3640,1097-0312},
   MRCLASS = {70F07 (37N05)},
  MRNUMBER = {4544807},
MRREVIEWER = {Patricia\ Yanguas},
       DOI = {10.1002/cpa.22014},
       URL = {https://doi-org.proxy-um.researchport.umd.edu/10.1002/cpa.22014},
}

@article {MR2104598,
    AUTHOR = {Marco, J-P. and Sauzin, D.},
     TITLE = {Wandering domains and random walks in {G}evrey near-integrable
              systems},
   JOURNAL = {Ergodic Theory Dynam. Systems},
  FJOURNAL = {Ergodic Theory and Dynamical Systems},
    VOLUME = {24},
      YEAR = {2004},
    NUMBER = {5},
     PAGES = {1619--1666},
      ISSN = {0143-3857,1469-4417},
   MRCLASS = {37J40 (37B10 70H08 70K30)},
  MRNUMBER = {2104598},
MRREVIEWER = {Mikhail\ B.\ Sevryuk},
       DOI = {10.1017/S0143385703000786},
       URL = {https://doi-org.proxy-um.researchport.umd.edu/10.1017/S0143385703000786},
}

@book {MR501173,
    AUTHOR = {Hirsch, M. W. and Pugh, C. C. and Shub, M.},
     TITLE = {Invariant manifolds},
    SERIES = {Lecture Notes in Mathematics},
    VOLUME = {Vol. 583},
 PUBLISHER = {Springer-Verlag, Berlin-New York},
      YEAR = {1977},
     PAGES = {ii+149},
   MRCLASS = {58F15 (58F10)},
  MRNUMBER = {501173},
MRREVIEWER = {M.\ C.\ Irwin},
}

@book {PalisMelo,
    AUTHOR = {Palis, Jr., J. and de Melo, W.},
     TITLE = {Geometric theory of dynamical systems},
 PUBLISHER = {Springer-Verlag, New York-Berlin},
      YEAR = {1982},
     PAGES = {xii+198},
      ISBN = {0-387-90668-1},
   MRCLASS = {58-01 (34C35 58F09 58Fxx)},
  MRNUMBER = {669541},
MRREVIEWER = {Russell\ B.\ Walker},
}

@article {DelaLLaveScattmap,
    AUTHOR = {Delshams, A. and de la Llave, R. and Seara, T.M. },
     TITLE = {Geometric properties of the scattering map of a normally
              hyperbolic invariant manifold},
   JOURNAL = {Adv. Math.},
  FJOURNAL = {Advances in Mathematics},
    VOLUME = {217},
      YEAR = {2008},
    NUMBER = {3},
     PAGES = {1096--1153},
      ISSN = {0001-8708,1090-2082},
   MRCLASS = {37J40 (37D10 37J99 37N20 81U05)},
  MRNUMBER = {2383896},
MRREVIEWER = {Enrico\ Valdinoci},
       DOI = {10.1016/j.aim.2007.08.014},
       URL = {https://doi.org/10.1016/j.aim.2007.08.014},
}

@incollection {MR4729212,
    AUTHOR = {Marco, J-P.},
     TITLE = {A symplectic approach to {A}rnold diffusion problems},
 BOOKTITLE = {Hamiltonian systems: dynamics, analysis, applications},
    SERIES = {Math. Sci. Res. Inst. Publ.},
    VOLUME = {72},
     PAGES = {229--295},
 PUBLISHER = {Cambridge Univ. Press, Cambridge},
      YEAR = {2024},
      ISBN = {978-1-009-32071-9},
   MRCLASS = {53D05 (37J40)},
  MRNUMBER = {4729212},
}

@article {MR4509324,
    AUTHOR = {Gidea, M. and Marco, J-P.},
     TITLE = {Diffusing orbits along chains of cylinders},
   JOURNAL = {Discrete Contin. Dyn. Syst.},
  FJOURNAL = {Discrete and Continuous Dynamical Systems},
    VOLUME = {42},
      YEAR = {2022},
    NUMBER = {12},
     PAGES = {5737--5782},
      ISSN = {1078-0947,1553-5231},
   MRCLASS = {37J40 (37E40 70M20)},
  MRNUMBER = {4509324},
MRREVIEWER = {Peng\ Huang},
       DOI = {10.3934/dcds.2022121},
       URL = {https://doi-org.proxy-um.researchport.umd.edu/10.3934/dcds.2022121},
}

@book {MR240836,
    AUTHOR = {Abraham, R. and Robbin, J.},
     TITLE = {Transversal mappings and flows},
 PUBLISHER = {W. A. Benjamin, Inc., New York-Amsterdam},
      YEAR = {1967},
     PAGES = {x+161},
   MRCLASS = {57.55},
  MRNUMBER = {240836},
MRREVIEWER = {J.\ Palis},
}

@article {MR455049,
    AUTHOR = {Newhouse, S. E.},
     TITLE = {Quasi-elliptic periodic points in conservative dynamical
              systems},
   JOURNAL = {Amer. J. Math.},
  FJOURNAL = {American Journal of Mathematics},
    VOLUME = {99},
      YEAR = {1977},
    NUMBER = {5},
     PAGES = {1061--1087},
      ISSN = {0002-9327,1080-6377},
   MRCLASS = {58F20 (58F10)},
  MRNUMBER = {455049},
MRREVIEWER = {Zbigniew\ Nitecki},
       DOI = {10.2307/2374000},
       URL = {https://doi-org.proxy-um.researchport.umd.edu/10.2307/2374000},
}

@article {MR2981810,
    AUTHOR = {Kaloshin, V. and Saprykina, M.},
     TITLE = {An example of a nearly integrable {H}amiltonian system with a
              trajectory dense in a set of maximal {H}ausdorff dimension},
   JOURNAL = {Comm. Math. Phys.},
  FJOURNAL = {Communications in Mathematical Physics},
    VOLUME = {315},
      YEAR = {2012},
    NUMBER = {3},
     PAGES = {643--697},
      ISSN = {0010-3616,1432-0916},
   MRCLASS = {37J40},
  MRNUMBER = {2981810},
MRREVIEWER = {Paolo\ Perfetti},
       DOI = {10.1007/s00220-012-1532-x},
       URL = {https://doi-org.proxy-um.researchport.umd.edu/10.1007/s00220-012-1532-x},
}

@article {MR3951693,
    AUTHOR = {Guardia, M. and Kaloshin, V. and Zhang, J.},
     TITLE = {Asymptotic density of collision orbits in the restricted
              circular planar 3 body problem},
   JOURNAL = {Arch. Ration. Mech. Anal.},
  FJOURNAL = {Archive for Rational Mechanics and Analysis},
    VOLUME = {233},
      YEAR = {2019},
    NUMBER = {2},
     PAGES = {799--836},
      ISSN = {0003-9527,1432-0673},
   MRCLASS = {70F07 (70F16)},
  MRNUMBER = {3951693},
MRREVIEWER = {Martha\ Alvarez-Ram\'irez},
       DOI = {10.1007/s00205-019-01368-7},
       URL = {https://doi-org.proxy-um.researchport.umd.edu/10.1007/s00205-019-01368-7},
}

@article {MR629685,
    AUTHOR = {Alekseev, V. M.},
     TITLE = {Final motions in the three-body problem and symbolic dynamics},
   JOURNAL = {Uspekhi Mat. Nauk},
  FJOURNAL = {Akademiya Nauk SSSR i Moskovskoe Matematicheskoe Obshchestvo.
              Uspekhi Matematicheskikh Nauk},
    VOLUME = {36},
      YEAR = {1981},
    NUMBER = {4(220)},
     PAGES = {161--176, 248},
      ISSN = {0042-1316},
   MRCLASS = {58F11 (70F07)},
  MRNUMBER = {629685},
MRREVIEWER = {J.\ Dane\v s},
}

@article {MR5013757,
    AUTHOR = {Guardia, M. and Paradela, J. and Seara, T.M.},
     TITLE = {A degenerate {A}rnold diffusion mechanism in the restricted
              3-body problem},
   JOURNAL = {Ann. Sci. \'Ec. Norm. Sup\'er. (4)},
  FJOURNAL = {Annales Scientifiques de l'\'Ecole Normale Sup\'erieure.
              Quatri\`eme S\'erie},
    VOLUME = {58},
      YEAR = {2025},
    NUMBER = {6},
     PAGES = {1401--1478},
      ISSN = {0012-9593,1873-2151},
   MRCLASS = {37J40 (30B40)},
  MRNUMBER = {5013757},
}

@article {GelfreichTuraev,
    AUTHOR = {Gelfreich, V. and Turaev, D.},
     TITLE = {Arnold diffusion in a priori chaotic symplectic maps},
   JOURNAL = {Comm. Math. Phys.},
  FJOURNAL = {Communications in Mathematical Physics},
    VOLUME = {353},
      YEAR = {2017},
    NUMBER = {2},
     PAGES = {507--547},
      ISSN = {0010-3616,1432-0916},
   MRCLASS = {37J40 (37J10)},
  MRNUMBER = {3649479},
MRREVIEWER = {Dongfeng\ Zhang},
       DOI = {10.1007/s00220-017-2867-0},
       URL = {https://doi.org/10.1007/s00220-017-2867-0},
}

@incollection {Moeckeldrift,
    AUTHOR = {Moeckel, R.},
     TITLE = {Generic drift on {C}antor sets of annuli},
 BOOKTITLE = {Celestial mechanics ({E}vanston, {IL}, 1999)},
    SERIES = {Contemp. Math.},
    VOLUME = {292},
     PAGES = {163--171},
 PUBLISHER = {Amer. Math. Soc., Providence, RI},
      YEAR = {2002},
      ISBN = {0-8218-2902-5},
   MRCLASS = {37J40 (70H08)},
  MRNUMBER = {1884898},
MRREVIEWER = {Pere\ Guti\'errez},
       DOI = {10.1090/conm/292/04922},
       URL = {https://doi.org/10.1090/conm/292/04922},
}

@article {DiffusionEllipticKaloshin,
    AUTHOR = {Delshams, A. and Kaloshin, V. and de la Rosa, A.
              and Seara, T.M.},
     TITLE = {Global instability in the restricted planar elliptic three
              body problem},
   JOURNAL = {Comm. Math. Phys.},
  FJOURNAL = {Communications in Mathematical Physics},
    VOLUME = {366},
      YEAR = {2019},
    NUMBER = {3},
     PAGES = {1173--1228},
      ISSN = {0010-3616,1432-0916},
   MRCLASS = {70F07 (37N05 70F15 70K50)},
  MRNUMBER = {3927089},
MRREVIEWER = {Eduardo\ S. G. Leandro},
       DOI = {10.1007/s00220-018-3248-z},
       URL = {https://doi.org/10.1007/s00220-018-3248-z},
}

@incollection {PoschelKAM,
    AUTHOR = {Pöschel, J.},
     TITLE = {A lecture on the classical {KAM} theorem},
 BOOKTITLE = {Smooth ergodic theory and its applications ({S}eattle, {WA},
              1999)},
    SERIES = {Proc. Sympos. Pure Math.},
    VOLUME = {69},
     PAGES = {707--732},
 PUBLISHER = {Amer. Math. Soc., Providence, RI},
      YEAR = {2001},
      ISBN = {0-8218-2682-4},
   MRCLASS = {37J40 (37K55 70H08)},
  MRNUMBER = {1858551},
MRREVIEWER = {Enrico\ Valdinoci},
       DOI = {10.1090/pspum/069/1858551},
       URL = {https://doi.org/10.1090/pspum/069/1858551},
}

@article {BonattiDiazblenders,
    AUTHOR = {Bonatti, C. and D\'iaz, L. J.},
     TITLE = {Persistent nonhyperbolic transitive diffeomorphisms},
   JOURNAL = {Ann. of Math. (2)},
  FJOURNAL = {Annals of Mathematics. Second Series},
    VOLUME = {143},
      YEAR = {1996},
    NUMBER = {2},
     PAGES = {357--396},
      ISSN = {0003-486X,1939-8980},
   MRCLASS = {58F12 (58F10 58F15)},
  MRNUMBER = {1381990},
MRREVIEWER = {Marcelo\ Viana},
       DOI = {10.2307/2118647},
       URL = {https://doi.org/10.2307/2118647},
}

@article {NassiriPujalsTransitivity,
    AUTHOR = {Nassiri, M. and Pujals, E. R.},
     TITLE = {Robust transitivity in {H}amiltonian dynamics},
   JOURNAL = {Ann. Sci. \'Ec. Norm. Sup\'er. (4)},
  FJOURNAL = {Annales Scientifiques de l'\'Ecole Normale Sup\'erieure.
              Quatri\`eme S\'erie},
    VOLUME = {45},
      YEAR = {2012},
    NUMBER = {2},
     PAGES = {191--239},
      ISSN = {0012-9593,1873-2151},
   MRCLASS = {37J10},
  MRNUMBER = {2977619},
MRREVIEWER = {Luigi\ Chierchia},
       DOI = {10.24033/asens.2164},
       URL = {https://doi.org/10.24033/asens.2164},
}
\bibliographystyle{alpha}

\end{document}